\documentclass[11pt, a4paper]{article}

\usepackage{amsthm}
\usepackage{lscape}
\usepackage{pdflscape}
\usepackage{stmaryrd}
\usepackage{cmll}
\usepackage{comment}
\usepackage{amscd}
\usepackage{amssymb,amsmath}
\usepackage{mathptmx}
\usepackage{mathrsfs}
\usepackage{color}
\usepackage{xspace}
\usepackage{bussproofs}
\EnableBpAbbreviations
\usepackage{tikz}
\usepackage{enumitem}
\usepackage[colorlinks=true,
            linkcolor=blue,
           urlcolor=red,
           citecolor=cyan]{hyperref}

\usepackage[left=2.5cm,top=3cm,right=2.5cm,bottom=3cm]{geometry}

\newtheorem{theorem}{Theorem}[section]

\newtheorem{lemma}[theorem]{Lemma}

\newtheorem{definition}[theorem]{Definition}
\newtheorem{remark}[theorem]{Remark}
\newtheorem{example}[theorem]{Example}
\newtheorem{prop}[theorem]{Proposition}

\newcommand{\X}{\mathbb{X}}
\newcommand{\z}{z}
\newcommand{\jty}{J^\infty}

\newcommand{\marginnote}[1]{\marginpar{\raggedright\tiny{#1}}} 
\newcommand{\val}[1]{[\![{#1}]\!]}
\newcommand{\descr}[1]{(\![{#1}]\!)}

\newcommand{\starfor}{{/\!\!}_{\star}}
\newcommand{\circfor}{{/\!}_{\circ}}

\newcommand{\starback}{\backslash_{\star}}
\newcommand{\circback}{\backslash_{\circ}}

\newcommand{\bba}{\mathbb{A}}
\newcommand{\bbA}{\mathbb{A}}

\newcommand{\jira}{J^{\infty}(\bba)}
\newcommand{\mira}{M^{\infty}(\bba)}

\renewcommand{\phi}{\varphi}
\renewcommand{\emptyset}{\varnothing}

\newcommand{\leftsquigarrow}{\mathbin{\rotatebox[origin=c]{180}{$\rightsquigarrow$}}}

\renewcommand{\epsilon}{\varepsilon}

\def\aol{\rule[0.5865ex]{1.38ex}{0.1ex}}

\def\pdra{\mbox{$\,>\mkern-8mu\raisebox{-0.065ex}{\aol}\,$}}

\newcommand{\upsE}{\raisebox{1ex}{\rotatebox{180}{$e$}}}

\title{Non-distributive logics: from semantics to meaning}

\author{Willem Conradie$^{a}$, Alessandra Palmigiano$^{b,c}$ and Claudette Robinson$^{c}$ and Nachoem Wijnberg$^{d,e}$,\\ \\ 
$^a$School of Mathematics, University of the Witwatersrand \\
$^{b}$School of Business and Economics, Vrije Universiteit, Amsterdam\\
$^c$Department of Mathematics and Applied Mathematics, University of Johannesburg\\
$^{b}$ Amsterdam  Business School,  University of Amsterdam\\
$^c$School of Business and Economics, University of Johannesburg\\
} 

\date{}

\begin{document}

\maketitle 
\begin{abstract}
We discuss an ongoing line of research in the relational (non topological) semantics of non-distributive logics. The developments we consider are technically  rooted in dual characterization results and insights from unified correspondence theory. However, they also have broader, conceptual ramifications for the intuitive meaning of non-distributive logics, which we explore. 
\end{abstract}

\section{Introduction}\label{sec:Introduction}
The term `non-distributive logics' (cf.~\cite{chiara1976general}) refers to the wide family of non-classical propositional logics in which the distributive laws $\alpha \wedge (\beta\vee \gamma) \vdash (\alpha\wedge \beta)\vee (\alpha\wedge \gamma)$ and $(\alpha\vee \beta)\wedge (\alpha\vee \gamma)\vdash \alpha \vee (\beta\wedge \gamma) $ do not need to be valid. 
Since the rise of very well known instances such as quantum logic \cite{mackey1957quantum}, interest in non-distributive logics has been building steadily over the years. This interest has been motivated by insights from a range of fields in logic and neighbouring disciplines. Techniques and ideas have come from  pure mathematical areas such as lattice theory, duality and representation (cf.~\cite{urquhart1978topological, dunn1991gaggle, hartung1992topological, allwein1993duality, Pl95, hartonasDunn1997stone, hartonas1997duality, GH01}, and more recently \cite{dunn2005canonical, gehrke2006generalized, craig2013fresh, moshier2014topological, moshier2014topologicalII, CGH15, hartonas2017order, hartonas2018order, hartonas2018stone}), and areas in mathematical logic such as algebraic proof theory (cf.~\cite{Belardinelli2004, GaJiKoOn07, ciabattoni2012algebraic}), but also from the philosophical and formal foundations of quantum physics \cite{birkhoff1936logic, Goldblatt:Ortho:74, chiara1976general}, philosophical logic \cite{routley1973semantics} theoretical computer science and formal linguistics \cite{lambek1958mathematics, girard1987linear, moortgat1997categorial}.

The present paper discusses an ongoing line of research in the relational (non topological) semantics of non-distributive logics. The developments we will consider are technically  rooted in dual characterization results and insights from unified correspondence theory \cite{UnifiedCor, CoPa:non-dist, conradie2016constructive}. However, they also have broader, {\em conceptual} ramifications for the {\em intuitive meaning} of non-distributive logics. Specifically, the slogan from the title, `from semantics to meaning',  intends to convey the idea that,  not dissimilarly from the conceptual contribution of Kripke frames to the intuitive understanding of modal logics in various signatures, the relational semantics of non-distributive logics can help to illuminate what these logics are about at a more fundamental and conceptual level than what the technical insights, constraints and desiderata typically yield. For instance, a natural question is whether relational semantics of (some) non-distributive logics can  provide an intuitive explanation of why, or under which circumstances, the failure of distributivity is a reasonable and desirable feature; i.e.~whether a given relational semantics supports one or more intuitive interpretations under which the failure of distributivity is an essential part of what `correct reasoning patterns' are in  certain specific contexts. Perhaps even more interestingly, whether relational semantics can be used  to unambiguously identify those contexts. Such an intuitive explanation also requires a different interpretation of the connectives $\vee$ and $\wedge$ which coherently fits with the interpretation of the other logical connectives, and which coherently extends to the meaning of axioms in various signatures. 

The starting point and background of the present paper is the theory of those non-distributive logics which arise as the logics canonically associated with varieties of normal lattice expansions (cf.~Section \ref{sec:LE:logics}), i.e.~algebras based on general lattices with additional connectives of any finite arity, which satisfy certain finite-distributivity properties coordinate-wise. Throughout the paper, these logics will be referred to as  {\em normal LE-logics}, or just {\em LE-logics}. By their definition, LE-logics are presented via their algebraic semantics; however, relational semantics for LE-logics in any signature can be introduced by a process of {\em dual characterization} which pivots on well-known adjunctions and representation theorems \cite{birkhoff1940lattice, Pl95} {\em complete} lattices. In Section \ref{sec:Interpretation}, we illustrate the dual characterization methodology in detail, and apply it to obtain the definitions of  two relational semantic frameworks for LE-logics: the {\em polarity-based} frames and the {\em graph-based} frames. Although  polarity-based  and  graph-based semantics are tightly connected and arise from the application of the same methodology, they give rise to two radically different intuitive interpretations of  LE-logics: namely, the polarity-based semantics supports the interpretation of LE-logics as logics of {\em formal concepts} (and consequently, of specific LE-signatures such as those of  lattice-based normal modal logics as e.g.~epistemic logics of categories and concepts); the graph-based semantics supports a view of LE-logics as {\em hyper-constructivist} logics, i.e.~logics in which the principle of excluded middle fails at the {\em meta-linguistic level} (in the sense that, at states in graph-based models, formulas can be satisfied, refuted or {\em neither}), and hence their propositional base  generalizes intuitionistic logic in the same way in which intuitionistic logic generalizes classical logic.  Consequently, we will argue that graph-based semantics supports the interpretation of specific LE-signatures such as those of  lattice-based normal modal logics as e.g.~epistemic logics of {\em informational entropy}). All this is discussed in  Section \ref{sec:Interpretation}. 

In the present paper, we will only scratch the surface of a broad research program, and hint at the existence of an extremely rich conceptual and technical landscape. Rather than giving an exhaustive survey of the existing results in the relational semantics of non-distributive logic, we focus on: (a) highlighting the methodological aspects involved in the definition of the relational semantics for LE-logics; this methodology can be applied also to different semantics than the ones discussed in the present paper (and it would be very interesting to discover which possibly different interpretations of LE-logics are supported by other semantics), and (b) substantiating the claim that relational semantics for non-distributive logics can bring the same types of benefits that Kripke semantics has brought to modal logic, both from a technical point of view (concerning results such as the finite model property, correspondence theory, transfer results driven by G\"odel-McKinsey-Tarski-style translations, Goldblatt-Thomason-style characterization theorems, semantic cut elimination) and from a conceptual viewpoint, concerning the extent to which these logics are suited to model a range of situations and phenomena.
Finally, the picture already emerging from the preliminary account  presented in this paper suggests that, similarly to what happens in the case of modal logics, many important results and insights can be obtained in complete uniformity, also across different types of relational semantics, so that various proposals and solutions developed for specific signatures (e.g.~Routley-Meyer semantics for substructural logics) can be systematically connected and extended to other signatures. 

\section{LE-logics}\label{sec:LE:logics}

Informally, LE-logics are the logics of (varieties of) lattice expansions: 
for disjoint sets of connectives $\mathcal{F}$ and $\mathcal{G}$, a {\em  lattice expansion} (abbreviated as LE) is a tuple $\bba = (\mathbb{L}, \mathcal{F}^\bbA, \mathcal{G}^\bbA)$ such that $\mathbb{L}$ is a bounded  lattice, $\mathcal{F}^\bbA = \{f^\bbA\mid f\in \mathcal{F}\}$ and $\mathcal{G}^\bbA = \{g^\bbA\mid g\in \mathcal{G}\}$, such that every $f^\bbA\in\mathcal{F}^\bbA$ (resp.\ $g^\bbA\in\mathcal{G}^\bbA$) is an $n_f$-ary (resp.\ $n_g$-ary) operation on $\bbA$. 
We say that an LE is {\em normal} if every $f^\bbA\in\mathcal{F}^\bbA$ (resp.\ $g^\bbA\in\mathcal{G}^\bbA$) is coordinatewise either monotone or antitone, and preserves finite (hence also empty) joins (resp.\ meets) in each monotone coordinate and reverses finite (hence also empty) meets (resp.\ joins) in each antitone coordinate.  Henceforth, since every LE is assumed to be normal, the adjective will be dropped.
An LE as above is {\em complete} if $\mathbb{L}$ is a complete lattice, and every $f^\bbA\in\mathcal{F}^\bbA$ (resp.\ $g^\bbA\in\mathcal{G}^\bbA$) preserves arbitrary joins (resp.\ meets) in each monotone coordinate and reverses arbitrary meets (resp.\ joins) in each antitone coordinate.




For any  algebraic LE-signature as above, 
the set of formulas of the language $\mathcal{L}_\mathrm{LE}: = \mathcal{L}_\mathrm{LE}(\mathcal{F},\mathcal{G})$ over a denumerable set $\mathsf{Prop}$ of proposition letters 
is defined recursively as follows:
\[
\phi ::= p \mid \bot \mid \top \mid \phi \wedge \phi \mid \phi \vee \phi \mid f(\overline{\phi}) \mid g(\overline{\phi})
\]
where $p \in \mathsf{Prop}$, $f \in \mathcal{F}$ and $g \in \mathcal{G}$.

\begin{example}\label{example:LE languages} 
	The language $\mathcal{L}_{\mathrm{Bi}}$ of \emph{bi-intuitionistic logic} \cite{rauszer1974formalization} is obtained by instantiating $\mathcal{F}: = \{\pdra\}$   and $\mathcal{G}: = \{\rightarrow\}$. Both connectives are binary, and are antitone in the first coordinate and monotone in the second one.  The language $\mathcal{L}_{\mathrm{IL}}$ of \emph{intuitionistic logic} is the $\{\pdra\}$-free fragment of $\mathcal{L}_{\mathrm{Bi}}$. The language $\mathcal{L}_{\mathrm{DML}}$ of \emph{distributive modal logic} (cf.\ \cite{GNV}, \cite{ALBAPaper}) is obtained by letting $\mathcal{F}: = \{\Diamond, {\lhd}\}$ and and   $\mathcal{G}: = \{\Box, {\rhd}\}$; all connectives are  unary, and  $\Diamond$ and $\Box$ are monotone while $\lhd$ and $\rhd$ negative. The language $\mathcal{L}_{\mathrm{PML}}$ of  {\em positive modal logic} \cite{DunnPosML} is the $\{{\lhd}, {\rhd} \}$-free fragment of $\mathcal{L}_{\mathrm{DML}}$. The language $\mathcal{L}_{\mathrm{BiML}}$ of \emph{bi-intuitionistic modal logic} \cite{wolter1998CoImplication} is obtained by instantiating $\mathcal{F}: = \{\pdra, \Diamond\}$ and $\mathcal{G}: = \{\rightarrow, \Box\}$. The language $\mathcal{L}_{\mathrm{IML}}$ of \emph{intuitionistic modal logic} \cite{fischerservi1977modal} is the $\{\pdra\}$-free fragment of $\mathcal{L}_{\mathrm{BiML}}$. The language $\mathcal{L}_{\mathrm{SDM}}$ of  {\em semi-De Morgan logic}  (cf.\ \cite{sankappanavar1987semi}) coincides with the language of Orthologic \cite{Goldblatt:Ortho:74} and is the $\{\Diamond, \Box, \lhd\}$-free fragment of $\mathcal{L}_{\mathrm{DML}}$. The language $\mathcal{L}_{\mathrm{FL}}$ of the {\em Full Lambek calculus} \cite{la61, GaJiKoOn07} is obtained by instantiating $\mathcal{F}: = \{e, \circ\}$ with $e$ nullary  and $\circ$ binary and monotone in both coordinates,  and   $\mathcal{G}: = \{\circfor, \circback\}$ with $\circfor$ (resp.~$\circback$) binary and monotone in its first (resp.~second) coordinate and antitone in its second (resp.~first) one. The language $\mathcal{L}_{\mathrm{LG}}$ of the {\em Lambek-Grishin calculus} (cf.\ \cite{Moortgat}) is  obtained by letting $\mathcal{F}: = \{e, \circ, \starfor, \starback \}$   and   $\mathcal{G} := \{\upsE, \star, \circfor, \circback\}$. Here $\upsE$ is nullary, $\star$ is binary and monotone in both coordinates, and $\starfor$ (resp.~$\starback$) binary and monotone in its first (resp.~second) coordinate and antitone in its second (resp.~first) one.  The language $\mathcal{L}_{\mathrm{MALL}}$ of the  {\em multiplicative-additive fragment of linear logic} (cf.\ \cite{GaJiKoOn07}) is the $\{\star, \starfor, \starback, \circfor\}$-free fragment of $\mathcal{L}_{\mathrm{LG}}$.
\end{example}
In what follows, we will mainly focus on LE-signatures the `expansion' part of which consists of unary connectives, and take them as our running examples throughout the remainder of the paper (see \cite{greco2018algebraic,CoPa:non-dist} for a treatment of arbitrary signatures). The basic framework is given by 
%
%
the basic normal {\em non-distributive modal logic} $\mathbf{L} :  = \mathbf{L}_{\mathcal{L}_{\mathrm{DML}}}$, defined as the smallest set of sequents $\phi\vdash\psi$ in the language $\mathcal{L}_\mathrm{DML}$ of distributive modal logic,  containing the following axioms:
\begin{itemize}
	\item Sequents for propositional connectives:
	\begin{align*}
	&p\vdash p, && \bot\vdash p, && p\vdash \top, & &  &\\
	&p\vdash p\vee q, && q\vdash p\vee q, && p\wedge q\vdash p, && p\wedge q\vdash q, &
	\end{align*}
	\item Sequents for modal operators: 
	\begin{align*}
	&\top\vdash \Box \top &&
	\Box p\wedge \Box q \vdash \Box ( p\wedge q)
	&\Diamond \bot\vdash \bot &&
	\Diamond ( p \vee q)\vdash \Diamond p \vee \Diamond  q  \\
	  &\top\vdash {\rhd} \bot &&
	  {\rhd}p \land {\rhd}q \vdash {\rhd }(p\vee q )
	   & {\lhd}\top\vdash \bot &&
	   {\lhd} (p \land q) \vdash {\lhd } p \vee {\lhd} q \\
	\end{align*}
\end{itemize}
and  closed under the following inference rules:
\begin{displaymath}
\frac{\phi\vdash \chi\quad \chi\vdash \psi}{\phi\vdash \psi}
\quad\quad
\frac{\phi\vdash \psi}{\phi\left(\chi/p\right)\vdash\psi\left(\chi/p\right)}
\quad\quad
\frac{\chi\vdash\phi\quad \chi\vdash\psi}{\chi\vdash \phi\wedge\psi}
\quad\quad
\frac{\phi\vdash\chi\quad \psi\vdash\chi}{\phi\vee\psi\vdash\chi}
\quad\quad
\end{displaymath}
\begin{displaymath}
\frac{\phi\vdash\psi}{\Box \phi\vdash \Box \psi}
\quad\quad
\frac{\phi\vdash\psi}{\Diamond \phi\vdash \Diamond \psi}
\quad\quad
\frac{\phi\vdash\psi}{{\rhd} \psi\vdash {\rhd} \phi}
\quad\quad
\frac{\phi\vdash\psi}{{\lhd} \psi\vdash {\lhd} \phi}
\end{displaymath}



\section{Polarity-based and graph-based semantics of LE-logics}
In the present section, we discuss two types of relational semantics for LE-logics, starting with the common methodology used to define them, both for specific LE-signatures \cite{gehrke2006generalized, conradie2015relational, graph-based-wollic} and for arbitrary ones \cite{CoPa:non-dist,greco2018algebraic,conradie2018goldblatt}.
This methodology uses adjunctions involving complete lattices to define the interpretation of LE-formulas on relational structures by `translating' their interpretation on algebras by means of a dual characterization process on which we expand and exemplify in the next subsection (cf.~also discussions in \cite{UnifiedCor,CoPa:non-dist}). The resulting semantic environments are those we discuss below and in the appendix. 

\subsection{Polarities, reflexive graphs and complete lattices}
\label{ssec:polarities graphs}
\paragraph{Polarities and representation of complete lattices.} A {\em polarity} \cite{ganter2012formal} is a structure $\mathbb{P} = (A, X, I)$ such that $A$ and $X$ are sets and $I\subseteq A\times X$ is a binary relation. 

For every polarity $\mathbb{P}$, maps $(\cdot)^\uparrow: \mathcal{P}(A)\to \mathcal{P}(X)$ and $(\cdot)^\downarrow: \mathcal{P}(X)\to \mathcal{P}(A)$ can be defined as follows:
$B^\uparrow: = \{x\in X \mid\forall a \in A(a \in B\to aIx)\}$ and $Y^\downarrow: = \{a\in A\mid \forall x \in X(x \in Y \to aIx)\}$. The maps $(\cdot)^\uparrow$ and $(\cdot)^\downarrow$ form a \textit{Galois connection} between $(\mathcal{P}(A), \subseteq)$ and $(\mathcal{P}(X), \subseteq)$, i.e. $Y \subseteq B^\uparrow$ iff $B\subseteq Y^\downarrow$
for all $B \in \mathcal{P}(A)$ and $Y\in \mathcal{P}(X)$. 

A {\em formal concept} of $\mathbb{P}$ is a pair 
$c = (\val{c}, \descr{c})$ such that $\val{c}\subseteq A$, $\descr{c}\subseteq X$, and $\val{c}^{\uparrow} = \descr{c}$ and $\descr{c}^{\downarrow} = \val{c}$. It follows immediately from this definition that if $(\val{c}, \descr{c})$ is a formal concept, then $\val{c}^{\uparrow\downarrow} = \val{c}$ and $\descr{c}^{\downarrow\uparrow} = \descr{c}$, that is $\val{c}$ and $\descr{c}$ are \textit{Galois-stable}.  The set $\mathbb{L}(\mathbb{P})$  of the formal concepts of $\mathbb{P}$ can be partially ordered as follows: for any $c, d\in \mathbb{L}(\mathbb{P})$, \[c\leq d\quad \mbox{ iff }\quad \val{c}\subseteq \val{d} \quad \mbox{ iff }\quad \descr{d}\subseteq \descr{c}.\]
With this order, $\mathbb{L}(\mathbb{P})$ is a complete lattice such that, for any $\mathcal{X}\subseteq  \mathbb{L}(\mathbb{P})$,
\begin{eqnarray*}
\bigwedge \mathcal{X}& =  &(\bigcap \{\val{c}\mid c\in \mathcal{X}\}, (\bigcap \{\val{c}\mid c\in \mathcal{X}\})^\uparrow)\\
 \bigvee \mathcal{X} & =  &( (\bigcap \{\descr{c}\mid c\in \mathcal{X}\})^\downarrow, \bigcap \{\descr{c}\mid c\in \mathcal{X}\}). 
\end{eqnarray*}
This complete lattice is referred to as the {\em concept lattice} $\mathbb{P}^+$ of $\mathbb{P}$. Moreover,
\begin{prop}
\label{prop:join and meet generators}
For any polarity $\mathbb{P} = (A, X, I)$, the complete lattice $\mathbb{P}^+$ is completely join-generated by the set $\{\mathbf{a}: = (a^{\uparrow\downarrow}, a^{\uparrow})\mid a\in A\}$ and is completely meet-generated by the set $\{\mathbf{x}: =(x^{\downarrow},x^{\downarrow\uparrow})\mid x\in X\}$.
\end{prop}
\begin{proof}
Every formal concept is both of the form $c = (Y^{\downarrow},Y^{\downarrow\uparrow})$ for some $Y\subseteq X$ and of the form $c = (B^{\uparrow\downarrow}, B^{\uparrow})$ for some $B\subseteq A$. Since $Y = \bigcup\{x\mid x\in Y\}$ and  $B = \bigcup\{a\mid a\in B\}$, and $(\cdot)^\downarrow: \mathcal{P}(X)\to \mathcal{P}(A)$ and $(\cdot)^\uparrow: \mathcal{P}(A)\to \mathcal{P}(X)$, being  Galois-adjoint, are completely join-reversing, $\val{c} = Y^{\downarrow} = \bigcap\{ \val{\mathbf{x}}\mid x\in Y\} = \val{\bigwedge \{\mathbf{x}\mid x\in Y\}}$ and $\val{c} = B^{\uparrow\downarrow} = (\bigcap\{ \descr{\mathbf{a}}\mid a\in B\})^{\downarrow} = \val{\bigvee \{\mathbf{a}\mid a\in B\}}$, as required. 
\end{proof}

The following theorem is a converse to the proposition above.\footnote{For every partial order $(X, \leq)$ and every $x\in X$, we let $x{\uparrow}: = \{y\in X\mid x\leq y\}$ and $x{\downarrow}: = \{y\in X\mid y\leq x\}$.}
\begin{theorem}[Birkhoff's representation theorem] Any complete lattice $\mathbb{L}$ is isomorphic to the concept lattice $\mathbb{P}^+$ of some polarity $\mathbb{P}$.
\end{theorem}
\begin{proof}
Let $\mathbb{P}: = (L, L, \leq)$, where $L$ is the domain of $\mathbb{L}$ and $\leq$ is  the lattice order. Then it is easy to see that the formal concepts of $\mathbb{P}$ are of the form $((\bigwedge X){\downarrow}, (\bigwedge X){\uparrow})$ for any $X\subseteq L$, and  since $\mathbb{L}$ is complete, the assignments $a\mapsto (a{\downarrow}, a{\uparrow})$ and $((\bigwedge X){\downarrow}, (\bigwedge X){\uparrow})\mapsto \bigwedge X$ define a pair of order isomorphisms between $\mathbb{L}$ and $\mathbb{P}^{+}$. 
\end{proof}

\paragraph{Reflexive graphs and representation of complete lattices.} A {\em reflexive graph} is a structure $\X = (Z, E)$ such that $Z$ is a nonempty set, and $E\subseteq Z\times Z$ is a reflexive relation, i.e.~$\Delta\subseteq E$, where $\Delta: = \{(z, z)\mid z\in Z\}$. From now on, we will assume that all graphs we consider are reflexive even when we drop the adjective.  
Any graph $\X = (Z, E)$  defines the polarity $\mathbb{P_X} = (Z_A,Z_X, I_{E^{c}})$ where $Z_A = Z = Z_X$ and  $I_{E^{c}}\subseteq Z_A\times Z_X$ is defined as $aI_{E^{c}} x$ iff $aE^cx$ iff $(a, x)\notin E$.  

The complete lattice $\X^{+}$ associated with a graph $\X$ is defined as the concept lattice of $\mathbb{P_X}$.
Conversely, for any lattice $\mathbb{L}$, let $\mathsf{Flt}(\mathbb{L})$ and $\mathsf{Idl}(\mathbb{L})$ denote the set of filters and ideals of $\mathbb{L}$, respectively. The graph associated with $\mathbb{L}$ is $\mathbb{X_L}:=(Z,E)$ where $Z$ is the set of tuples
$(F, J)\in \mathsf{Flt}(\mathbb{L})\times \mathsf{Idl}(\mathbb{L})$ such that  $F\cap J = \varnothing$. For $z \in Z$, we denote by $F_z$ the filter part of $z$ and by $J_z$ the ideal part of $z$. Clearly, filter parts and ideal parts of states of $\mathbb{X_L}$ must be proper.  The (reflexive) $E$ relation is defined by 
$zEz'$ if and only if $F_z \cap J_{z'}= \emptyset$. 

\subsection{From algebraic to relational semantics}
\label{ssec: from alg to rel}
In the present subsection, we discuss how the adjunctions outlined in the previous subsection can be used to define an interpretation of the propositional language of general lattices  on polarities and reflexive graphs,  starting from the standard interpretation of this logic on complete lattices. The same method will be applied to define polarity-based and graph-based semantics 
for any LE-language $\mathcal{L}_{\mathrm{LE}} = \mathcal{L}_{\mathrm{LE}} (\mathcal{F}, \mathcal{G})$   from complete $\mathcal{L}_{\mathrm{LE}}$-algebras, via the adjunctions above, suitably expanded to account for the interpretation of the additional connectives, in a uniform and modular way.

This method stems from the observation that interpretations on complex algebras and satisfaction relations on frames  {\em correspond} to one another along the adjunctions outlined above. 


Let L be the propositional language of the logic of general lattices. In what follows we will abuse notation and identify the logic with its language.  Let us briefly recall how the correspondence between interpretations  on complex algebras and satisfaction relations on frames works in the Boolean and distributive settings.

In the Boolean and distributive settings, for any partially ordered set $\mathbb{F} = (W, \leq)$ (in the Boolean case, $\leq$ coincides with the identity $\Delta_W$), and any satisfaction relation $\Vdash \ \subseteq \ W\times \mathrm{L}$ between elements of $\mathbb{F}$ and formulas, an interpretation $\overline{v}: \mathrm{L}\to \mathbb{F}^+$ can be defined\footnote{In the Boolean setting, $\mathbb{F}^+$ is the powerset algebra $\mathcal{P}(W)$; in the distributive setting, $\mathbb{F}^+$ is the  algebra $\mathcal{P}^\uparrow(W)$ of the upward-closed subsets of $\mathbb{F}$.}, which is a lattice homomorphism, and is obtained as the unique homomorphic extension of the equivalent functional representation of the relation $\Vdash$ as a map $v: \mathsf{Prop}\to \mathbb{F}^+$, defined as $v(p) = {\Vdash}^{-1}[p]$\footnote{Notice that in order for this equivalent functional representation to be well defined, we need to assume that the relation $\Vdash$ is $\mathbb{F}^+$-{\em compatible}, i.e.\ that ${\Vdash}^{-1}[p]\in \mathbb{F}^+$ for every $p\in \mathsf{Prop}$. In the Boolean case, every relation from $W$ to $\mathrm{L}$ is clearly $\mathbb{F}^+$-compatible, but already in the distributive case this is not so: indeed ${\Vdash}^{-1}[p]$ needs to be an upward- or downward-closed subset of $\mathbb{F}$. This gives rise to the persistence condition, e.g.\ in the relational semantics of intuitionistic logic.}. In this way, interpretations on complete lattices can be derived from satisfaction relations, so that for every $w\in W$ and every $\mathrm{L}$-formula $\phi$, \begin{equation}\label{eq: desiderata satisfaction} w\Vdash \phi\quad \mbox{ iff }\quad \mathbf{w}\leq \overline{v}(\phi),\end{equation}

\noindent where, on the right-hand side, $\mathbf{w}\in \jty(\mathbb{F}^+)$ is the completely join-irreducible element\footnote{In the Boolean setting, $\mathbf{w}: = \{w\}$; in the distributive setting, $\mathbf{w}: = w{\uparrow} = \{w'\in W\mid w\leq w'\}$.} of $\mathbb{F}^+$ arising from $w\in W$. 
Conversely, for any such $\mathbb{F}$, any lattice homomorphism $\overline{v}: \mathrm{L}\to \mathbb{F}^+$ gives rise to a satisfaction relation $\Vdash \ \subseteq \ W\times \mathrm{L}$ defined as in \eqref{eq: desiderata satisfaction}. Instantiating condition \eqref{eq: desiderata satisfaction} according to the syntactic shape of each formula in $\mathrm{L}$, we obtain the familiar satisfaction conditions of $\mathrm{L}$-formulas in the distributive setting; for instance,  the satisfaction clause of $\vee$-formulas can be obtained as follows:
\begin{center}
\begin{tabular}{rcll}
$w\Vdash \phi\vee\psi$ &   iff & $\mathbf{w}\leq \overline{v}(\phi\vee\psi) $& definition of $\Vdash$ as in \eqref{eq: desiderata satisfaction}\\
&   iff & $\mathbf{w}\leq \overline{v}(\phi)\vee \overline{v}(\psi) $ & $\overline{v}$ is a homomorphism\\
&   iff & $\mathbf{w}\leq \overline{v}(\phi)$ or  $\mathbf{w}\leq\overline{v}(\psi) $ & $\mathbf{w}\in \jty(\mathbb{F}^+)$ and $\mathbb{F}^+$ distributive\\
&   iff & $w\Vdash \phi$ or  $w\Vdash\psi $ & \eqref{eq: desiderata satisfaction}  on $\phi$ and $\psi$ by induction hypothesis. \\
\end{tabular}
\end{center}
To define an interpretation of $\mathrm{L}$  on polarities and reflexive graphs, we are going to apply the homomorphism-to-relation direction of the argument illustrated above. That is, for an arbitrary polarity $\mathbb{P} = (A, X, I)$ 
any homomorphic assignment $\overline{v}: \mathrm{L} \rightarrow \mathbb{P}^+$  will give rise to  {\em pairs} of relations $(\Vdash, \succ)$ such that $\Vdash\ \subseteq\  A\times \mathrm{L}$ and  $\succ\ \subseteq\ X\times \mathrm{L}$, so that for every $a\in A$ and $x\in X$, and every $\mathrm{L}$-formula $\phi$, \begin{equation}\label{eq: desiderata satisfaction-polarity} a\Vdash \phi\quad \mbox{ iff }\quad \mathbf{a}\leq \overline{v}(\phi),\end{equation}
\begin{equation}\label{eq: desiderata co-satisfaction-polarity} x\succ \phi\quad \mbox{ iff }\quad \overline{v}(\phi)\leq \mathbf{x},\end{equation}
where, on the right-hand side of the equivalences above, $\mathbf{a} = (a^{\uparrow\downarrow}, a^{\uparrow})\in \mathbb{P}^+$ and  $\mathbf{x}=(x^{\downarrow},x^{\downarrow\uparrow})\in \mathbb{P}^+$. For instance, spelling out conditions \eqref{eq: desiderata satisfaction-polarity} and \eqref{eq: desiderata co-satisfaction-polarity} for $\vee$-formulas, we obtain the following clauses:
{\footnotesize
\begin{center}
\begin{tabular}{rcll}
$x\succ \phi\vee\psi$ &   iff & $\overline{v}(\phi\vee\psi) \leq \mathbf{x}$& definition of $\succ$ as in \eqref{eq: desiderata co-satisfaction-polarity}\\
&   iff & $\overline{v}(\phi)\vee \overline{v}(\psi) \leq \mathbf{x} $ & $\overline{v}$ homomorphism\\
&   iff & $\overline{v}(\phi)\leq \mathbf{x}$ and  $\overline{v}(\psi)\leq \mathbf{x} $ & definition of $\vee$\\
&   iff & $x\succ \phi$ and  $x\succ\psi $ & \eqref{eq: desiderata co-satisfaction-polarity}  on $\phi$ and $\psi$ by induction hypothesis. \\
\end{tabular}
\end{center}

\begin{center}
\begin{tabular}{rcll}
&&$a\Vdash \phi\vee\psi$\\ 
&   iff & $\mathbf{a}\leq \overline{v}(\phi\vee\psi) $& definition of $\Vdash$ as in \eqref{eq: desiderata satisfaction-polarity}\\
&   iff & $\mathbf{a}\leq \overline{v}(\phi)\vee \overline{v}(\psi) $ & $\overline{v}$ homomorphism\\
&   iff & $\val{\mathbf{a}}\subseteq \val{\overline{v}(\phi)\vee \overline{v}(\psi)}$ & definition of  order  in $\mathbb{P}^+$\\
&   iff & $a^{\uparrow\downarrow}\subseteq (\descr{\overline{v}(\phi)}\cap\descr{\overline{v}(\psi)})^{\downarrow} $ & definition of $\mathbf{a}$ and $\vee$ in $\mathbb{P}^+$\\
&   iff & $a\in (\descr{\overline{v}(\phi)}\cap\descr{\overline{v}(\psi)})^{\downarrow} $ & definition of Galois-closure\\
&   iff & for all $x\in X$, if  $x\in \descr{\overline{v}(\phi)}\cap\descr{\overline{v}(\psi)}$ then $aIx$& definition of $(\cdot)^{\downarrow}$\\
&   iff & for all $x\in X$, if  $x\in \descr{\overline{v}(\phi)}$ and $x\in \descr{\overline{v}(\psi)}$ then $aIx$& definition of $\cap$\\
&   iff & for all $x\in X$, if  $x^{\downarrow\uparrow} \subseteq \descr{\overline{v}(\phi)}$ and $x^{\downarrow\uparrow} \subseteq  \descr{\overline{v}(\psi)}$ then $aIx$& definition of Galois-closure\\
&   iff & for all $x\in X$, if  $\overline{v}(\phi)\leq \mathbf{x}$ and $\overline{v}(\psi)\leq \mathbf{x}$ then $aIx$& definition of  order  in $\mathbb{P}^+$\\

&   iff & for all $x\in X$, if  $x\succ\phi$ and $x\succ\psi$ then $aIx$.& \eqref{eq: desiderata co-satisfaction-polarity}  on $\phi$ and $\psi$ by ind.~hyp.\\
\end{tabular}
\end{center}
}

Notice that, unlike the argument in the distributive setting, we did not need to appeal to join- or (meet-)irreducibility.
Reasoning in an analogous way for the remaining connectives, we obtain the following recursive definition of $\Vdash$ and $\succ$ on polarities for all $\mathrm{L}$-formulas:
{\footnotesize
\begin{flushleft}
	\begin{tabular}{lllllll}
		 $a \Vdash \bot$ && $aIx$ for all $x\in X$  & & $x \succ \bot$ && always\\
		$ a \Vdash \top$ &&always & &$x \succ \top$ &&$aIx$ for all $a\in A$\\
		$a \Vdash p$ & iff & $a\in \val{\overline{v}(p)}$ & &$x \succ p$ & iff & $x\in \descr{\overline{v}(p)}$\\
		$a \Vdash \phi\wedge \psi$ & iff & $a \Vdash \phi$ and $a \Vdash  \psi$ \\
		$ x \succ \phi\wedge \psi$ & iff &  for all $a\in A$, if $a \Vdash \phi$ and $a \Vdash \psi $ then $a I x$\\
		$ a \Vdash \phi\vee \psi$ & iff & for all $x\in X$, if $x \succ \phi$ and 
		$x \succ \psi $ then $a I x$\\  
		$ x \succ \phi\vee \psi$ & iff &  $ x \succ \phi$ and $x \succ  \psi$. \\
		\end{tabular}
\end{flushleft}
}




Likewise, for an arbitrary  reflexive graph $\X = (Z, E)$,
any homomorphic assignment  $\overline{v}: \mathrm{L} \rightarrow \mathbb{X}^+$ will give rise to  {\em pairs} of relations $(\Vdash, \succ)$ such that $\Vdash\ \subseteq\  Z\times \mathrm{L}$ and  $\succ\ \subseteq\ Z\times \mathrm{L}$, so that for every  $z\in Z$ and every $\mathrm{L}$-formula $\phi$, 
\begin{equation}\label{eq: desiderata satisfaction-graph} z\Vdash \phi\quad \mbox{ iff }\quad \mathbf{z}_s\leq \overline{v}(\phi),\end{equation}
\begin{equation}\label{eq: desiderata co-satisfaction-graph} z\succ \phi\quad \mbox{ iff }\quad \overline{v}(\phi)\leq \mathbf{z}_r,\end{equation}
where, on the right-hand side of the equivalences above, $\mathbf{z}_s = (z^{\uparrow\downarrow}, z^{\uparrow})\in \mathbb{X}^+$ and  $\mathbf{z}_r=(z^{\downarrow},z^{\downarrow\uparrow})\in \mathbb{X}^+$. For instance, spelling out conditions \eqref{eq: desiderata satisfaction-graph} and \eqref{eq: desiderata co-satisfaction-graph} for $\vee$-formulas, we obtain the following clauses:
{\footnotesize
\begin{center}
\begin{tabular}{rcll}
$z\succ \phi\vee\psi$ &   iff & $\overline{v}(\phi\vee\psi) \leq \mathbf{z}_r$& definition of $\succ$ as in \eqref{eq: desiderata co-satisfaction-graph}\\
&   iff & $\overline{v}(\phi)\vee \overline{v}(\psi) \leq \mathbf{z}_r $ & $\overline{v}$ homomorphism\\
&   iff & $\overline{v}(\phi)\leq \mathbf{z}_r$ and  $\overline{v}(\psi)\leq \mathbf{z}_r $ & definition of $\vee$\\
&   iff & $z\succ \phi$ and  $z\succ\psi $. & \eqref{eq: desiderata co-satisfaction-graph}  on $\phi$ and $\psi$ by induction hypothesis. \\
\end{tabular}
\end{center}
}

{\footnotesize
\begin{center}
\begin{tabular}{rcll}
&&$z\Vdash \phi\vee\psi$ \\
&   iff & $\mathbf{z}_s\leq \overline{v}(\phi\vee\psi) $& definition of $\Vdash$ as in \eqref{eq: desiderata satisfaction-graph}\\
&   iff & $\mathbf{z}_s\leq \overline{v}(\phi)\vee \overline{v}(\psi) $ & $\overline{v}$ homomorphism\\
&   iff & $\val{\mathbf{z}_s}\subseteq \val{\overline{v}(\phi)\vee \overline{v}(\psi)}$ & definition of  order  in $\mathbb{X}^+$\\
&   iff & $z^{\uparrow\downarrow}\subseteq (\descr{\overline{v}(\phi)}\cap\descr{\overline{v}(\psi)})^{\downarrow} $ & definition of $\mathbf{z}_s$ and $\vee$ in $\mathbb{X}^+$\\
&   iff & $z\in (\descr{\overline{v}(\phi)}\cap\descr{\overline{v}(\psi)})^{\downarrow} $ & definition of Galois-closure\\
&   iff & for all $z'\in Z$, if  $z'\in \descr{\overline{v}(\phi)}\cap\descr{\overline{v}(\psi)}$ then $zE^cz'$& definition of $(\cdot)^{\downarrow}$\\
&   iff & for all $z'\in Z$, if  $z'\in \descr{\overline{v}(\phi)}$ and $z'\in \descr{\overline{v}(\psi)}$ then $zE^cz'$& definition of $\cap$\\
&   iff & for all $z'\in Z$, if  $z^{\downarrow\uparrow} \subseteq \descr{\overline{v}(\phi)}$ and $x^{\downarrow\uparrow} \subseteq  \descr{\overline{v}(\psi)}$ then $zE^cz'$& definition of Galois-closure\\
&   iff & for all $z'\in Z$, if  $\overline{v}(\phi)\leq \mathbf{x}$ and $\overline{v}(\psi)\leq \mathbf{z'}_r$ then $zE^cz'$& definition of  order  in $\mathbb{P}^+$\\

&   iff & for all $z'\in Z$, if  $z'\succ\phi$ and $z'\succ\psi$ then $zE^cz'$& \eqref{eq: desiderata co-satisfaction-graph}  on $\phi$ and $\psi$ by ind.~hyp.\\
&   iff & for all $z'\in Z$, if $zEz'$ then $z'\not\succ\phi$ or $z'\not\succ\psi$.  & contraposition.\\
\end{tabular}
\end{center}
}

Reasoning in an analogous way, we obtain the following recursive definition of $\Vdash$ and $\succ$ on graphs for all $\mathrm{L}$-formulas:
{\footnotesize
\begin{flushleft}
	\begin{tabular}{lllllll}
		$z \succ \bot$ && always && $z \Vdash \bot$ && never\\
		$z \Vdash \top$ &&always & &$z \succ \top$ &&never\\
		$z \Vdash p$ & iff & $\z\in V(p)$ & &$z \succ p$ & iff & for all $z'$, if $z'Ez$ then \\
		$z \succ \phi \vee \psi$ &iff &$z\succ \phi \text{ and } z \succ \psi$ & & & & $z' \not\Vdash p$ \\
		$z \Vdash \phi \vee \psi$ &iff &for all $z'$, if $zEz'$ then $z'\not\succ\phi$ or $z'\not\succ\psi$\\
		$z \Vdash \phi \wedge \psi$ &iff &$z\Vdash \phi \text{ and } z \Vdash \psi$ \\
		$z \succ \phi \wedge\psi$ &iff & for all $z'$, if $z'Ez$ then $z'\not\Vdash \phi$ or $z'\not\Vdash \psi$\\
\end{tabular}
\end{flushleft}
}		
When the assignment $\overline{v}$ is clear from the context we will sometimes write $\val{\phi}$ and $\descr{\phi}$ for $\val{\overline{v}(\phi)}$ and $\descr{\overline{v}(\phi)}$, respectively.

\subsection{Relational interpretation of additional connectives}
\label{ssec:rel interpret modalities}
In the present subsection, we apply the method discussed in the previous subsection to define polarity-based and graph-based semantics 
for any LE-language $\mathcal{L}_{\mathrm{LE}} = \mathcal{L}_{\mathrm{LE}} (\mathcal{F}, \mathcal{G})$. Starting    from complete $\mathcal{L}_{\mathrm{LE}}$-algebras, we will translate homomorphic assignments of $\mathcal{L}_{\mathrm{LE}}$-formulas into relations $\Vdash$ and $\succ$ via suitable expansions of the adjunctions between complete lattices and polarities and (complete) lattices and reflexive graphs.


Let us first recall  how  the usual satisfaction relation clauses can be retrieved from the algebraic interpretation in the Boolean and distributive case for a unary diamond $\Diamond$. Let ML be the propositional language of the logic of general lattices expanded with a unary and positive $f$-type connective $\Diamond$. In what follows, we will abuse notation and identify the minimal normal ML-logic with its language.
 Let $\mathbb{W} = (W, \leq)$ be a partially ordered set (in the Boolean case, $\leq$ coincides with the identity $\Delta_W$), and let us expand $\mathbb{W}^+$ with a completely join-preserving unary operation $\Diamond^{\mathbb{W}^+}$ so as to obtain a (Boolean or distributive) modal algebra $\mathbb{A} = (\mathbb{W}^+, \Diamond^{\mathbb{W}^+})$.
Let $\overline{v}: \mathrm{ML}\to \mathbb{A}$ be a  homomorphic assignment, hence $\overline{v}(\Diamond \phi) = \Diamond^{\mathbb{W}^+}\overline{v}(\phi)$. 

As done in the previous subsection, the recursive definition of the  relation $\Vdash\ \subseteq W\times \mathrm{ML}$ corresponding to the assignment  $\overline{v}$  is obtained by spelling out equation (\ref{eq: desiderata satisfaction}). 

For the case of $\Diamond$-formulas, 
since $\mathbb{W}^+$ is a perfect\footnote{A complete lattice $\bba$ is {\em perfect} if it is both completely join-generated by the set $\jira$ of its completely join-irreducible elements, and completely meet-generated by the set $\mira$ of its completely meet-irreducible elements.} distributive lattice and $\overline{v}(\Diamond\psi)\in \mathbb{W}^+$, we get $\overline{v}(\psi) = \bigvee\{\mathbf{w'}\in \jty(\mathbb{W}^+)\mid \mathbf{w'}\leq \overline{v}(\psi)\} = \bigvee\{\mathbf{w'}\in \jty(\mathbb{W}^+)\mid w'\Vdash \psi\}$. Since by assumption $\overline{v}$ is a homomorphism, $\overline{v}(\Diamond \psi) = \Diamond^{\mathbb{W}^+}\overline{v}(\psi) = \Diamond^{\mathbb{W}^+}(\bigvee\{\mathbf{w'}\in \jty(\mathbb{W}^+)\mid w'\Vdash \psi\})$, and since $\Diamond^{\mathbb{W}^+}$ is completely join-preserving, we get:

$$\overline{v}(\Diamond \psi) = \bigvee\{\Diamond^{\mathbb{W}^+}\mathbf{w'}\mid w'\Vdash \psi\}.$$

\noindent Hence, for any $w\in W$,
{\footnotesize
\begin{center}
\begin{tabular}{r c l l}
$w\Vdash \Diamond\psi$& iff &$\mathbf{w}\leq \overline{v}(\Diamond\psi)$\\
& iff & $\mathbf{w}\leq \bigvee\{\Diamond^{\mathbb{W}^+}\mathbf{w'}\mid  w'\Vdash \psi\}$ &\\
& iff & $\mathbf{w}\leq \Diamond^{\mathbb{W}^+}\mathbf{w'}$ for some $w'\in W$ s.t.~$w'\Vdash \psi$ & ($\mathbf{w}$ completely join-prime)\\
& iff & $w R_\Diamond w'$ for some $w'\in W$ s.t.~$w'\Vdash \psi$.& \\
\end{tabular}
\end{center}
}

\noindent So we have done two things at the same time: Firstly, we have {\em defined} the accessibility relation $R_\Diamond\subseteq W\times W$ corresponding to the interpretation of $\Diamond$ as $\Diamond^{\mathbb{W}^+}$ as follows:
\[w R_\Diamond w' \quad\mbox{ iff }\quad \mathbf{w}\leq \Diamond^{\mathbb{W}^+}\mathbf{w'}. \]
Secondly, we have derived the corresponding defining clause for  $\Diamond$-formulas. 
The same can be done in the general lattice case, starting e.g.~from a polarity $\mathbb{P}= (A, X, I)$ and expanding $\mathbb{P}^+$ with a completely join preserving unary operation $\Diamond^{\mathbb{P}^+}$ so as to obtain a  lattice-based complete modal algebra $\mathbb{A} = (\mathbb{P}^+, \Diamond^{\mathbb{P}^+})$. 
Analogously to the way we argued above and appealing to Proposition \ref{prop:join and meet generators}, we can write $\overline{v}(\psi) = \bigvee\{\mathbf{a}\in \mathbb{P}^+\mid \mathbf{a}\leq \overline{v}(\psi)\} = \bigvee\{\mathbf{a}\in \mathbb{P}^+\mid a\Vdash \psi\}$. Since by assumption $\overline{v}$ is a homomorphism, $\overline{v}(\Diamond \psi) = \Diamond^{\mathbb{P}^+}\overline{v}(\psi) = \Diamond^{\mathbb{P}^+}(\bigvee\{\mathbf{a}\in \mathbb{P}^+\mid a\Vdash \psi\})$, and since $\Diamond^{\mathbb{P}^+}$ is completely join-preserving, we get:

\begin{equation}
\label{eq: diamondpsi as a bigvee}
\overline{v}(\Diamond \psi) = \bigvee\{\Diamond^{\mathbb{P}^+}\mathbf{a}\mid a\Vdash \psi\}.\end{equation}

However, the chain of equivalences above breaks down in the third step, since the elements $\mathbf{a}$ are not  in general completely join-{\em prime} anymore, but only complete join-generators (cf.~Proposition \ref{prop:join and meet generators}).  However, we can obtain a reduction also in this case, by crucially making use of the completely meet-generating elements $\mathbf{x}$: 

\begin{center}
\begin{tabular}{r c l l}
$x\succ \Diamond\psi$& iff &$\overline{v}(\Diamond\psi)\leq \mathbf{x}$ & definition of $\succ$ as in \eqref{eq: desiderata co-satisfaction-polarity}\\
& iff & $\bigvee\{\Diamond^{\mathbb{P}^+}\mathbf{a}\mid a\Vdash \psi\}\leq \mathbf{x}$ & \eqref{eq: diamondpsi as a bigvee}\\
& iff & for all $a\in A$, if $a\Vdash \psi$ then $\Diamond^{\mathbb{P}^+}\mathbf{a}\leq \mathbf{x}$ & definition of $\bigvee$\\
& iff & for all $a\in A$, if $a\Vdash \psi$ then $x R_{\Diamond} a$, & \eqref{eq:def rdiamond} 
\end{tabular}
\end{center}
where we have {\em defined} the accessibility relation $R_\Diamond\subseteq X\times A$ corresponding to the interpretation of $\Diamond$ as $\Diamond^{\mathbb{P}^+}$ as follows:
\begin{equation}\label{eq:def rdiamond} x R_\Diamond a \quad\mbox{ iff }\quad \Diamond^{\mathbb{P}^+}\mathbf{a}\leq \mathbf{x}. \end{equation}
Now, using the fact that the set of elements $\mathbf{x}$ for $x\in X$ are meet-generators of $\mathbb{P}^+$, we can write:
\begin{center}
\begin{tabular}{r c l l}
&&$a\Vdash \Diamond\psi$\\
&iff & $\mathbf{a}\leq \overline{v}(\Diamond\psi)$ & definition of $\Vdash$ as in \eqref{eq: desiderata satisfaction-polarity}\\
& iff & $\mathbf{a}\leq \bigwedge\{\mathbf{x}\mid  \overline{v}(\Diamond\psi)\leq \mathbf{x}\}$ & $\mathbf{x}$ for $x\in X$ meet-generators of $\mathbb{P}^+$\\
& iff & for all $x\in X$, if $\overline{v}(\Diamond\psi)\leq \mathbf{x}$ then $ \mathbf{a}\leq \mathbf{x}$& definition of $\bigwedge$\\
& iff & for all $x\in X$, if $x\succ \Diamond\psi$ then $ a Ix$ &definition of $\succ$ as in \eqref{eq: desiderata co-satisfaction-polarity}.\\
\end{tabular}
\end{center}
Similar arguments yield  the recursive definition of $\Vdash$ and $\succ$ on polarity-based relational structures for formulas in any  $\mathcal{L}_{\mathrm{LE}}$-signature, for instance:
\begin{center}
	\begin{tabular}{lll}
	$ a \Vdash \Box\phi$ & iff & for all $x\in X$, if $x \succ \phi$ then $a R_{\Box} x$ \\
		$x \succ \Box\phi$ & iff &  for all $a\in A$, if $a \Vdash \Box\phi $ then $ a I x$\\
		$a \Vdash \lhd\phi$ & iff & for all $x\in X$, if $x \succ \lhd\phi $ then $ a I x$ \\
		$ x \succ \lhd\phi$ & iff &  for all $x'\in X$, if $x' \succ \phi $ then $ x R_{\lhd} x'$ \\
		$ a \Vdash \rhd\phi$ & iff & for all $a'\in A$, if $a' \Vdash \phi $ then $ a R_{\rhd} a'$\\
		$x \succ \rhd\phi$ & iff & for all $a\in A$, if $a \Vdash \rhd\phi $ then $ a I x$\\
\end{tabular}
\end{center}
%
\noindent
where the relations $R_{\Box}\subseteq A\times X$, $R_{\lhd}\subseteq X\times X$ and $R_{\rhd}\subseteq A\times A$ are defined as follows:
\begin{equation}\label{eq:def rboxand rtriangles} 
aR_{\Box}  x \; \mbox{ iff }\; \mathbf{a}\leq \Box^{\mathbb{P}^+}\mathbf{x}\quad\quad x R_{\lhd} x' \; \mbox{ iff }\;  {\lhd}^{\mathbb{P}^+}\mathbf{x'}\leq \mathbf{x}\quad\quad a R_{\rhd} a' \; \mbox{ iff }\; \mathbf{a}\leq {\rhd}^{\mathbb{P}^+}\mathbf{a'}.
\end{equation}
More generally, for any connective $f\in \mathcal{F}$ of arity $n_f$ and any connective $g\in \mathcal{G}$ of arity $n_g$, any interpretation of $f$ and $g$ on 
$\mathbb{P}^+$  
will yield  relations $R_f\subseteq X\times \overline{A}^{(n_f)}$  
and $R_g\subseteq A\times \overline{X}^{(n_g)}$, where $\overline{A}^{(n_f)}$ denotes the $n_f$-fold cartesian product of $A$ and $X$ such that for each $1\leq i\leq n_f$ if $f$ is monotone in its $i$th coordinate then the $i$th coordinate of $\overline{A}^{(n_f)}$ is $A$, whereas if $f$ is antitone in its $i$th coordinate then the $i$th coordinate of $\overline{A}^{(n_f)}$ is $X$, and  $\overline{X}^{(n_g)}$ is defined in a similar way w.r.t.~$g$. The relations $R_f$ and $R_g$ are defined as follows:
\begin{equation}\label{eq:def rf and rg} R_f (x, \overline{a}^{(n_f)}) \; \mbox{ iff }\;  f^{\mathbb{P}^+}(\overline{\mathbf{a}}^{(n_f)})\leq \mathbf{x} \quad\quad R_g (a, \overline{x}^{(n_g)}) \; \mbox{ iff }\;   \mathbf{a}\leq g^{\mathbb{P}^+}(\overline{\mathbf{x}}^{(n_g)}). \end{equation}
The corresponding clauses for $\Vdash$ and $\succ$ are then the following ones: 
\begin{flushleft}
	\begin{tabular}{lllllll}
	$ a \Vdash g(\overline{\phi})$ & iff & for all $\overline{x}^{(n_g)}\in \overline{X}^{(n_g)}$, if $\overline{x}^{(n_g)} \succ^{(n_g)}\overline{ \phi}$ then $R_{g} (a, \overline{x}^{(n_g)})$ \\
		$x \succ g(\overline{\phi})$ & iff &  for all $a\in A$, if $a \Vdash g(\overline{\phi})$ then $ a I x$\\
		$ x \succ f(\overline{\phi})$ & iff & for all $\overline{a}^{(n_f)}\in \overline{A}^{(n_f)}$, if $\overline{a}^{(n_f)} \Vdash^{(n_f)}\overline{ \phi}$ then $R_{f} (x, \overline{a}^{(n_f)})$ \\
		$a \Vdash f(\overline{\phi})$ & iff &  for all $x\in X$, if $x\succ g(\overline{\phi})$ then $ a I x$.\\
\end{tabular}
\end{flushleft}

where, if  $\overline{a}^{(n_f)}\in \overline{A}^{(n_f)}$, the notation $\overline{a}^{(n_f)} \Vdash^{(n_f)}\overline{ \phi}$  refers to the conjunction over $1\leq i\leq n_f$ of statements of the form $a_i\Vdash\phi_i$, if $f$ is positive in its $i$th coordinate, or $x_i\succ \phi_i$, if $f$ is negative in its $i$th coordinate, whereas if  $\overline{x}^{(n_g)}\in \overline{X}^{(n_g)}$, the notation $\overline{x}^{(n_g)} \succ^{(n_g)}\overline{ \phi}$  refers to the conjunction over $1\leq i\leq n_g$ of statements of the form  $x_i\succ \phi_i$, if $g$ is positive in its $i$th coordinate, or $a_i\Vdash\phi_i$, if $g$ is negative in its $i$th coordinate.

As an example, we instantiate the above general clauses for a binary implication-like connective $\to$ in $\mathcal{G}$ which is antitone in the first coordinate and monotone in the second. The relation corresponding to this connective is then $R_{\to} \subseteq A \times A \times X$ such that $R_{\to} (a, b, x)$  iff $\mathbf{a}\leq \mathbf{b} \to^{\mathbb{P}^+} \mathbf{x}$, and the clauses for $\Vdash$ and $\succ$ become: 
\begin{flushleft}
	\begin{tabular}{lllllll}
		$ a \Vdash \phi \to \psi$ & iff & for all $b \in A$ and $x \in X$, if $b \Vdash \phi$ and $x \succ \psi$, then $R_{\to} (a, b, x)$ and\\
		$x \succ \phi \to \psi$ & iff &  for all $a\in A$, if $a \Vdash \phi \to \psi$ then $a I x$.\\
	\end{tabular}
\end{flushleft}

Likewise, starting from a reflexive graph $\mathbb{X} = (Z,E)$ and a homomorphic assignment $\overline{v}:\mathcal{L}_{\mathrm{LE}} \to \mathbb{X}^+$ we will extend the relation $\Vdash$ and ${\succ}$ to the whole of $\mathcal{L}_{\mathrm{LE}}$ via suitable relations associated with each logical connective. We do not expand further on the derivations (some of which are more extensively reported in Section \ref{appendix:rel:int:addnl:cnctvs}), but limit ourselves to reporting the definition for the unary modalities and the general $f\in \mathcal{F}$ and  $g\in \mathcal{G}$. 

{\footnotesize
\begin{flushleft}
	\begin{tabular}{lllllll}
		$z \succ \Diamond\phi$ &iff & for all $z'$, if  $z R_{\Diamond}z'$  then $z' \not \Vdash \phi$&&
		$z \Vdash \Diamond\phi$ &iff &for all $z'$, if  $z Ez' $  then $z' \not\succ \Diamond \phi$\\
		$z \Vdash \Box \psi$ &iff &for all $z'$, if  $z R_{\Box}z'  $  then $z' \not\succ \psi$&&
		$z \succ \Box\psi$ &iff &for all $z'$, if  $z' Ez  $  then $z' \not\Vdash \Box \psi$\\
		$z \succ \lhd\psi$ & iff &  for all $z'$, if  $z R_{\lhd}\z'  $  then $z' \not\succ \psi$ &&
		$z \Vdash \lhd\psi$ & iff & for all $z'$, if  $z Ez' $  then $z' \not\succ \psi$ \\ 
		$z \Vdash \rhd\psi$ & iff & for all $z'$, if  $z R_{\rhd} z' $  then $z' \not\Vdash \psi$&&
		$z \succ \rhd\psi$ & iff & for all $z'$, if  $z'Ez $  then $z' \not\Vdash \rhd\psi$\\
	\end{tabular}
\end{flushleft}
}
where the relations $R_{\Box}$, $R_{\Diamond}$, $R_{\lhd}$ and $R_{\rhd}$ are defined as follows:

\begin{equation}\label{eq:def:rels:xtra:cnctvs:on:grphs} 
\begin{tabular}{ll}
$z R_{\Diamond} z' \; \mbox{ iff }\;  \Diamond^{\mathbb{X}^+} \mathbf{z'}_s \not\leq \mathbf{z}_r \quad\quad$ &$zR_{\Box}  z' \; \mbox{ iff }\; \mathbf{z}_s \not\leq \Box^{\mathbb{X}^+} \mathbf{z}'_r$ \\
$z R_{\lhd} z' \; \mbox{ iff }\;  {\lhd}^{\mathbb{X}^+}\mathbf{z'}_r \not\leq \mathbf{z}_r \quad\quad$ &$z R_{\rhd} z' \; \mbox{ iff }\; \mathbf{z}_s \not\leq {\rhd}^{\mathbb{X}^+}\mathbf{z'}_s$.
\end{tabular}
\end{equation}

As in the case of polarities, we can also generalize these definitions to arbitrary connectives $f\in \mathcal{F}$ and $g\in \mathcal{G}$ with arity $n_f$ and $n_g$, respectively. Any interpretation of $f$ and $g$ on $\mathbb{X}^+$ will yield relations $R_f \subseteq Z^{n_f + 1}$ and $R_g \subseteq Z^{n_g + 1}$ defined as follows: 
\begin{equation}\label{eq:def rf and rg:graph} R_f (z, \overline{z}) \; \mbox{ iff }\;  f^{\mathbb{X}^+}(\overline{\mathbf{z}}_s^{(n_f)}) \not\leq \mathbf{z}_r \quad\quad R_g (z, \overline{z}^{(n_g)}) \; \mbox{ iff }\;   \mathbf{z}_s \not\leq g^{\mathbb{X}^+}(\overline{\mathbf{z}}_r^{(n_g)}), \end{equation}
where the $i$-th component of $\overline{\mathbf{z}}_s^{(n_f)}$ is $\mathbf{z}_s$ ($\mathbf{z}_r$) if $f$ is monotone (antitone) in its $i$-th argument, and the $i$-th component of $\overline{\mathbf{z}}_r^{(n_g)}$ is $\mathbf{z}_r$ ($\mathbf{z}_s$) if $g$ is monotone (antitone) in its $i$-th argument. 

Given an $n_g$-tuple of formulas $\overline{\phi} = (\phi_1, \ldots, \phi_{n_g})$ and $\overline{z} = (z_1, \ldots, z_{n_g}) \in Z^{n_g}$, we write $\overline{z} \; \succ^{(n_g)} \; \overline{\phi}$ to indicate that  $z_i \succ \phi_i$ for all $1 \leq i \leq n_g$ for which $g$ is monotone in the $i$-th coordinate and that  $z_j \Vdash \phi_j$ for all $1 \leq j \leq n_g$ for which $g$ is antitone in the $j$-th coordinate. Similarly, given an $n_f$-tuple of formulas $\overline{\phi} = (\phi_1, \ldots, \phi_{n_f})$ and $\overline{z} = (z_1, \ldots, z_{n_f}) \in Z^{n_f}$, we write $\overline{z} \; \Vdash^{(n_f)} \; \overline{\phi}$ to indicate that  $z_i \Vdash \phi_i$ for all $1 \leq i \leq n_f$ for which $f$ is monotone in the $i$-th coordinate and that  $z_j \succ \phi_j$ for all $1 \leq j \leq n_f$ for which $f$ is antitone in the $j$-th coordinate. Using this notation, the corresponding clauses for $\Vdash$ and $\succ$ are:

\begin{flushleft}
	\begin{tabular}{lllllll}
		$z \Vdash g(\overline{\phi})$ & iff &for all $\overline{z} \in Z^{n_g}$, if $R_{g} (z, \overline{z})$ then it is not the case that $\overline{z} \; \succ^{(n_g)} \; \overline{\phi}$\\
		$z \succ g(\overline{\phi})$ & iff & for all $z'\in Z$, if $z' E z$ then $z'\not \Vdash g(\overline{\phi})$\\
		$z \succ f(\overline{\phi})$ & iff & for all $\overline{z} \in Z^{n_f}$, if $R_{f} (z, \overline{z})$ then it is not the case that $\overline{z} \; \Vdash^{(n_f)} \; \overline{\phi}$\\
		$z \Vdash f(\overline{\phi})$ & iff & for all $z'\in Z$, if $z E z'$ then $z'\not \succ f(\overline{\phi})$.\\
	\end{tabular}
\end{flushleft}

\subsection{Projecting onto the classical setting}
\label{ssec: projection}

We finish this section by  specifying how the polarity-based and graph-based semantics  discussed in the previous subsections project onto the Kripke semantics of classical normal modal logic. For ease of presentation we address this issue in the signature of $\mathcal{L}_{\mathrm{PML}}$. 

Algebraically, the polarity-based and graph-based structures $\mathbb{F}$ that can be recognized as ``classical'' are exactly those the complex algebra  $\mathbb{F}^+$ of which is (isomorphic to) a powerset algebra (possibly endowed with extra operations).
This is  the case of structures based on graphs $\mathbb{X} = (Z, E)$ 
such that $E = \Delta$, or based on polarities $\mathbb{P}= (A, X, I)$ such that $A = X = Z$ for some set $Z$, and $aIx$ iff $a\neq x$ for all $a, x\in Z$. In these cases, the
powerset algebra $\mathcal{P}(Z)$ can be represented as a concept lattice
each element of which is of the form $(Y, Y^c)$ for any $Y\subseteq Z$. Therefore, for any formula $\phi$ interpreted in these structures, $\descr{\phi} = \{z\in Z\mid z\succ \phi\} = \val{\phi}^c$, and hence $z\succ \phi$ iff $z\not\Vdash \phi$. This of course provides a more direct way to reduce $\succ$ to $\Vdash$ than the one defined in terms of the operations $(\cdot)^{\uparrow}$ and $(\cdot)^{\downarrow}$, which allows to formulate the well known recursive definition of satisfaction in the classical setting purely in terms of $\Vdash$.
In particular, the satisfaction clause for $\vee$-formulas on the ``classical'' polarity-based models described above can be rewritten as follows: for any $a\in Z$ and all formulas $\phi$ and $\psi$,
\begin{center}
	\begin{tabular}{rcll}
		&&$a\Vdash \phi\vee\psi$\\ 
		&   iff & for all $x\in X$, if  $x\succ\phi$ and $x\succ\psi$ then $aIx$& \\
		&   iff & for all $x\in X$, if  $x\not\Vdash\phi$ and $x\not\Vdash\psi$ then $a\neq x$& \\
		&   iff & for all $x\in X$, if $a=x$ then $x\Vdash\phi$ or $x\Vdash\psi$ & \\
		&   iff & $a\Vdash\phi$ or $a\Vdash\psi$. & \\
	\end{tabular}
\end{center}
Likewise, the satisfaction clause for $\bot$ can be rewritten as
\begin{center}
	\begin{tabular}{rcll}
		$a\Vdash \bot$ & iff &$a\not \succ \bot$    & which is never the case.\\
	\end{tabular}
\end{center}
As for the interpretation of $\Diamond$-formulas,  if $J_{R_{\Diamond}^c}\subseteq X\times A$ is the corresponding relation on the given  ``classical'' polarity-based models as described above, applying the following clause
\begin{center}
	\begin{tabular}{r c l l}
		$x\succ \Diamond\psi$& iff & for all $a\in A$, if $a\Vdash \psi$ then $x J_ {R_{\Diamond}^c} a$ & \\ 
	\end{tabular}
\end{center}
yields 
\begin{center}
	\begin{tabular}{r c l l}
		$x\Vdash \Diamond\psi$& iff & $x\not\succ \Diamond\psi$\\
		& iff & for some $a\in A$,  $a\Vdash \psi$ and $(x, a)\notin  J_ {R_{\Diamond}^c}$ & \\ 
		& iff & for some $a\in A$,  $a\Vdash \psi$ and $x R_{\Diamond} a$. & \\ 
	\end{tabular}
\end{center}
where $x J_ {R_{\Diamond}^c} a$ iff $(x, a)\notin R_{\Diamond} $ for every $x, a \in Z$.  
As for the interpretation of $\Box$-formulas, stipulating, likewise, that $a I_ {R_{\Box}^c} x$ iff $(a, x)\notin R_{\Box} $ for every $x, a \in Z$, the satisfaction clause for $\Box$-formulas can be rewritten as follows:
\begin{center}
	\begin{tabular}{r c l l}
		$ a \Vdash \Box\phi$ & iff & for all $x\in X$, if $x \succ \phi$ then $a I_{R_{\Box}^c} x$ 
		& \\
		& iff & for all $x\in X$, if $x \not\Vdash \phi$ then $(a, x)\notin  R_{\Box}$ 
		& \\
		& iff & for all $x\in X$, if $a R_{\Box} x$  then $x \Vdash \phi$.
		& \\
	\end{tabular}
\end{center}

Similar computations show that also in the setting of the graph-based semantics, the satisfaction and refutation clauses  project to the well known ones.  For instance, when $zE z'$ iff $z = z'$ and $z\succ \phi$ iff $z\not \Vdash\phi$, 
\begin{center}
	\begin{tabular}{rcll}
		$z\Vdash \phi\vee\psi$ 
		&   iff & for all $z'\in Z$, if $zEz'$ then $z'\not\succ\phi$ or $z'\not\succ\psi$  & \\
		&   iff & for all $z'\in Z$, if $z =z'$ then $z'\Vdash\phi$ or $z'\Vdash\psi$  & \\
		&   iff & $z\Vdash\phi$ or $z\Vdash\psi$. & \\
	\end{tabular}
\end{center}
Analogously, the following clause 
\begin{center}
	\begin{tabular}{rcll}
		$z \succ \Diamond\phi$ &iff & for all $z'$, if  $z R_{\Diamond}z'$  then $z' \not \Vdash \phi$\\
	\end{tabular}
\end{center}
yields
\begin{center}
	\begin{tabular}{rcll}
		$z \Vdash \Diamond\phi$ &iff & $z \not \succ \Diamond\phi$\\
		&iff & for some $z'$,   $z R_{\Diamond}z'$  and $z' \Vdash \phi.$\\
	\end{tabular}
\end{center}
The remaining computations are omitted.

In conclusion,  in the present section we have discussed a uniform methodology for defining relational semantics for normal LE-logics in any signature on the basis of their standard interpretation on  complete LE-algebras of compatible signature. We have concretely illustrated how this methodology works in the case of two different semantic environments, and we have discussed how these environments project onto the Kripke semantics of classical normal modal logic. However, in order for these environments to `make sense' in a more fundamental way, we need to couple them with  extra-mathematical interpretations which simultaneously account for the meaning of {\em all} connectives, and which coherently extend to the meaning of logical axioms and of their first-order correspondents.  In the next section, we discuss two possible  interpretations, and the views on LE-logics elicited by each of them.

\section{From semantics to meaning}\label{sec:Interpretation}
Any extra-mathematical interpretation of LE-logics must account for the failure of distributivity. 
Although, as discussed in the previous section, polarity-based and graph-based semantic structures arise from the standard interpretation in LE-algebras via the same dual-characterization methodology, they give rise to two radically different views on what LE-logics are and {\em mean}. The key difference lies in a  dichotomy between two interpretive strategies, each of which identifies different sources of non-distributivity. The first such strategy, supported by the polarity-based semantics, drops the interpretation of $\wedge$ and $\vee$ as conjunction and disjunction in natural language and stipulates that LE-formulas do not denote sentences describing states of affairs, but rather objects with a different ontology, such as categories, concepts, questions, theories, to which a truth value might not necessarily be applicable. The second interpretive strategy  is supported by the graph based semantics and retains  the sentential denotation of formulas in the context of a propositional logic which can be thought of as a kind of `hyper-constructivism'. 



\subsection{LE-logics as logics of formal concepts, categories, theories, interrogative agendas...}\label{ssec:Interpretation:Polarities}

\paragraph{Polarities as abstract databases.} The idea that lattices are the proper mathematical environment for  discussing ``especially systems which are in any sense hierarchies" goes back to 
Birkhoff \cite{birkhoff1938lattices-applications}. 
Based on this idea, Wille \cite{ganter2012formal} and his collaborators developed
Formal Concept Analysis (FCA) as a theory in information science aimed at the formal representation and analysis of  conceptual structures, which has been applied to a wide range of fields ranging from  psychology, sociology and linguistics to biology and chemistry. 

Building on philosophical insights developed by the school of Port-Royal \cite{arnauld1996antoine}, Wille specified concepts 
in terms of their {\em extension}, i.e.~the set of objects which exemplify the given concept, and their {\em  intension}, i.e.~the set of attributes shared by the objects in the extension of the given concept, and identified Birkhoff's polarities $\mathbb{P} = (A, X, I)$ (cf.~Section \ref{ssec:polarities graphs}), also referred to as {\em formal contexts}, as the appropriate mathematical environment in which to  formally represent these ideas: indeed, a polarity $\mathbb{P}$ as above can be understood as an abstract representation of a {\em database}, recording information about a given  set $A$ of {\em objects}  (relevant to a given context or situation), and a set $X$  of relevant attributes or {\em features}. In this representation, the (incidence) relation  $I\subseteq A\times X$ encodes  whether object $a\in A$ has feature $x\in X$ as $aIx$. 
The Galois-adjoint pair of maps $(\cdot)^\uparrow: \mathcal{P}(A)\to \mathcal{P}(X)$ and $(\cdot)^\downarrow: \mathcal{P}(X)\to \mathcal{P}(A)$ can be understood as {\em concept-generating maps}: namely, as maps taking any set $B$ of objects to the intension $B^\uparrow$ which uniquely determines  the formal concept $(B^{\uparrow\downarrow}, B^{\uparrow})$ generated by $B$, and any set $Y$ of attributes to the extension $Y^{\downarrow}$ which uniquely determines  the formal concept $(Y^{\downarrow}, Y^{\downarrow\uparrow})$   generated by $Y$. 
Hence, the philosophical and cognitive insight that concepts do not occur in isolation, but rather arise within a hierarchy of other concepts, finds a very natural representation in the construction of the complete lattice $\mathbb{P}^+$ and its natural order as the {\em sub-concept} relation. Indeed, a subconcept of a given concept, understood as a more restrictive concept, will have a smaller extension (i.e.~fewer examples) and a larger intension (i.e.~a larger set of requirements that objects need to satisfy in order to count as examples of the given sub-concept). This interpretation accounts for the failure of distributivity, as we will concretely illustrate below. 

\paragraph{Propositional lattice logic as the basic logic of formal concepts.} 
Imposing the FCA interpretation of polarities discussed above on the polarity-based semantics of the logic $\mathrm{L}$ (Section \ref{ssec: from alg to rel}) yields an interpretation of $\mathrm{L}$-formulas as terms (i.e.~names) denoting  formal concepts.
%
%
Starting from assignments to proposition variables, any $\mathrm{L}$-formula $\phi$ is then interpreted on a given polarity $\mathbb{P} = (A, X, I)$ as a formal concept $(\val{\phi}, \descr{\phi})\in \mathbb{P}^+$; specifically, for each object $a\in A$ and feature $x\in X$, the relations $a\Vdash \phi$ and  $x\succ \phi$ can be respectively understood as `object $a$ is a member of (or exemplifies) concept $\phi$' and `feature $x$ {\em describes} concept $\phi$', in the sense that $x$ is a required attribute of every example/member of $\phi$.
Accordingly, this reading suggests that $\phi\wedge \psi$ can be understood  as `the greatest (i.e.~least restrictive) common subconcept of concept $\phi$ and concept $\psi$', i.e.~the concept the extension of which is the intersection of the extensions of  $\phi$ and $\psi$. Similarly, $\phi\vee \psi$ is `the least (i.e.~most restrictive) common superconcept of concept $\phi$ and concept $\psi$', i.e.~the concept the intension of which is the intersection of the intensions of $\phi$ and $\psi$; the constant $\top$ can be understood as the most generic (or comprehensive) concept (i.e.~the one that, when interpreted in any given polarity $\mathbb{P}$ as above, allows all objects $a\in A$ as examples) while $\bot$ as the most restrictive  (i.e.~the one that, when interpreted in any given polarity $\mathbb{P}$ as above, requires its examples to have all attributes $x\in X$). Finally, $\phi\vdash \psi$ can be understood as the statement that `concept $\phi$ is a sub-concept of concept $\psi$'. 

As mentioned above, this interpretation accounts for the failure of distributivity. Indeed,  objects in the extension of  concept $\phi\vee \psi$ are only required to have all attributes common to the intentions of concepts $\phi$ and $\psi$; 
this weaker requirement potentially allows objects in $\val{\phi\vee \psi}$ which belong to neither $\val{\phi}$ nor to $\val{\psi}$. To illustrate this point concretely, consider the  `database' of theatrical plays $\mathbb{P} = (A, X, I)$ the set of objects of which is $A: = \{a, b, c\}$, where $a$ is {\em A Midsummer Night's Dream}, $b$ is {\em King Lear}, and $c$ is {\em Julius Caesar}, while its set of features is $X: = \{x, y, z\}$, where $x$ is {\em `no happy end'}, $y$ is {\em `some characters are real historical figures'}, and $z$ is {\em `two characters fall in love with each other'}. The following picture represents $\mathbb{P}$ and its associated concept lattice. 
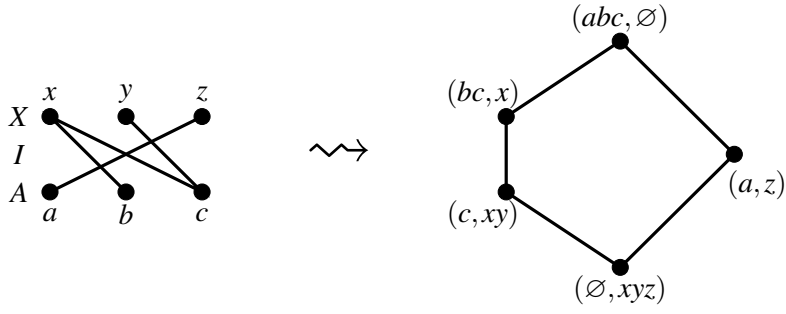
\begin{center}
	\begin{tikzpicture}
	\draw[very thick] (7.5, 2.5) -- (6, 3.5) --
	(6, 4.5) -- (7.5, 5.5) -- (9, 4)  -- (7.5, 2.5);
	\filldraw[black] (7.5,2.5) circle (3 pt);
	\filldraw[black] (6,3.5) circle (3 pt);
	\filldraw[black] (6,4.5) circle (3 pt);
	\filldraw[black] (7.5,5.5) circle (3 pt);
	\filldraw[black] (9,4) circle (3 pt);
	\draw (7.5, 2.2) node {$(\varnothing,xyz)$};
	\draw (5.7, 3.2) node {$(c,xy)$};
	\draw (5.7, 4.8) node {$(bc,x)$};
	\draw (7.5, 5.8) node {$(abc,\varnothing)$};
	\draw (9.3, 3.6) node {$(a,z)$};
	\draw (3.8,4) node{\Huge{$\rightsquigarrow$}};
	
	\draw[very thick] (1, 3.5) -- (0, 4.5) --
	(2, 3.5) -- (1, 4.5);
	\draw[very thick] (0, 3.5) -- (2, 4.5);
	\filldraw[black] (0,3.5) circle (3 pt); 
	\filldraw[black] (0,4.5) circle (3 pt); 
	\filldraw[black] (1,3.5) circle (3 pt); 
	\filldraw[black] (1,4.5) circle (3 pt); 
	\filldraw[black] (2,3.5) circle (3 pt); 
	\filldraw[black] (2,4.5) circle (3 pt); 
	\draw (-0.4,4.5) node {$X$};
	\draw (-0.4,4) node {$I$};
	\draw (-0.4,3.5) node {$A$};
	\draw (0,4.8) node {$x$};
	\draw (1,4.8) node {$y$};
	\draw (2,4.8) node {$z$};
	\draw (0,3.2) node {$a$};
	\draw (1,3.2) node {$b$};
	\draw (2,3.2) node {$c$};

	\end{tikzpicture}
\end{center}
Consider the atomic concept-variables $r, d, h$ where $r$ stands for `romantic comedy', $d$ for `drama' and $h$ for `historical drama'. Consider the assignment into $\mathbb{P}^+$ which maps  $r$ to $(a, z)$, $d$ to $(bc, x)$ and $h$ to $(c, yz)$. Notice that $b\in \val{h\vee r} = \val{\top} $ even though $b\notin \val{h}\cup\val{r}$. Accordingly, under this assignment, $h \vee r$ is interpreted as the top element of the concept lattice. The concept $d\wedge r$  is interpreted as the bottom element and, as historical drama is a subgenre of drama, $d\wedge h$ coincides with $h$.  Hence $(d \wedge h) \vee (d \wedge r)$ coincides with $h$ while $d \wedge (h \vee r)$ coincides with $d$. Thus distribution of $\wedge$ over $\vee$ fails.  

Dually, objects in the extension of  concept $\phi\wedge \psi$ might have more attributes in common than those in $\descr{\phi}\cup\descr{\psi}$; for instance in the example above, $r\wedge d$ is the concept $(\varnothing, xyz)$ which requires feature $y$ which is neither required by $r$ nor by $d$; hence $h \vee (r\wedge d)$ coincides with $h$, while $(h \vee d) \wedge (h \vee r)$  coincides with $d$, witnessing the failure of $\vee$ over $\wedge$ also.



\paragraph{Lattice-based normal modal logic as an epistemic logic of formal concepts.} So far, we have discussed how the polarity-based semantics of propositional lattice logic $\mathrm{L}$ allows for an interpretation of $\mathrm{L}$-formulas as names of formal concepts, and for a coherent interpretation of the meaning of all propositional lattice connectives in a way that makes the failure of distributivity a desirable feature rather than an awkward issue. Next, based on \cite{love2sorts, Tarkpaper}, we discuss a first way in which this interpretation can be extended also to the modal connectives. For the sake of simplicity, let us consider the minimal normal LE-logic $\mathrm{L}_{\Box}$, in the language specified according to the notation of Section \ref{sec:LE:logics} by instantiating $\mathcal{F}: = \varnothing$ and $\mathcal{G}: = \{\Box\}$, with $\Box$ unary and monotone. In what follows, we will abuse notation and identify $\mathrm{L}_{\Box}$ with its language.
As discussed in Section \ref{ssec:rel interpret modalities}, this logic can be interpreted on relational structures $\mathbb{F} = (\mathbb{P}, R_\Box)$ such that $\mathbb{P} = (A, X, I)$ is a polarity, and $R_{\Box}\subseteq A\times X$ is a (compatible, see Definition \ref{def:polarity:based:frm}) relation such that, for any assignment $v:\mathsf{Prop}\to \mathbb{P}^+$, corresponding  relations $\Vdash\ \subseteq\ A\times \mathrm{L}_{\Box}$ and $\succ\ \subseteq\ X\times \mathrm{L}_{\Box}$ can be defined. In the case of $\Box$-formulas, this yields
\begin{center}
	\begin{tabular}{lll}
		$ a \Vdash \Box\phi$ & iff & for all $x\in X$, if $x \succ \phi$ then $a R_{\Box} x$ \\
		$x \succ \Box\phi$ & iff &  for all $a\in A$, if $a \Vdash \Box\phi $ then $ a I x$.
	\end{tabular}
\end{center}
Building on the understanding of polarities as abstract representation of databases, the relational structures $\mathbb{F} = (\mathbb{P}, R_\Box)$ can be understood as (abstract representations of) enriched databases which, together with objective information about objects and their features encoded in the incidence relation $I$ of $\mathbb{P}$, encode also {\em subjective} information regarding whether given objects have given attributes {\em according to a given agent}; this understanding allows us to read $aR_{\Box} x$ as `object $a$ has attribute $x$ {\em according to agent $i$}. Of course, this interpretation can be further specialized so as to represent agents' knowledge ($aR_{\Box} x$ iff `agent $i$ knows that object $a$ has attribute $x$'), beliefs ($aR_{\Box} x$ iff `agent $i$ believes that object $a$ has attribute $x$'), perceptions ($aR_{\Box} x$ iff `agent $i$ sees that object $a$ has attribute $x$'), evidential reasoning ($aR_{\Box} x$ iff `agent $i$ has evidence that object $a$ has attribute $x$'), and so on. Each of these  epistemic interpretations will give rise to a different epistemic reading of  $\Box\phi$ as `concept $\phi$ according to the given agent $i$': namely, `concept $\phi$ as is known/believed/perceived/experienced 
by agent $i$'. Also in the case of $\mathrm{L}_{\Box}$-formulas, for every object $a$ and attribute $x$, the symbols $a\Vdash \Box\phi$ and $x\succ \Box\phi$ can be understood as `object $a$ is a member/example of $\Box \phi$' and `attribute $x$ describes $\Box \phi$', respectively. Interestingly, the condition that 
\[ a \Vdash \Box\phi \quad \mbox{ iff }\quad\mbox{for all } x\in X, \mbox{ if } x \succ \phi\mbox{ then }a R_{\Box} x\]
can then be informally understood as saying that any object $a$ is a member/example of concept $\phi$ according to agent $i$ if and only if agent $i$ attributes to $a$ all the defining features of concept $\phi$. This reading is indeed coherent with our informal understanding of which objects should count as members of `concept $\phi$ according to agent $i$'.

Finally, one would also expect that the different variants of epistemic interpretations would satisfy different axioms; for instance, if $\Box\phi$ is interpreted as `concept $\phi$ as is {\em known} by agent $i$', one would ask whether there is some $\mathrm{L}_{\Box}$-axioms which would encode the counterparts, in the lattice-based setting, of well known classical epistemic principles such as the {\em factivity} condition which distinguishes knowledge from belief, and what would this condition look like in the context of  polarity-based relational structures. As is well known, in the setting of classical normal modal logic, factivity is formalized as the modal axiom $\Box\phi\to \phi$ (which reads `if agent $i$ knows that $\phi$, then $\phi$ is indeed the case'). Moreover, this is a Sahlqvist formula and corresponds on Kripke frames $(W, R)$ to reflexivity or, equivalently, $\Delta\subseteq R$. Since $\mathrm{L}_{\Box}$ is not sentential, the closest approximation to the classical formula $\Box\phi\to \phi$ is the $\mathrm{L}_{\Box}$-sequent $\Box\phi\vdash \phi$, which turns out (cf.~\cite[Proposition 4.3]{roughconcepts}) to correspond on polarity-based structures $\mathbb{F}$ as above to the first-order condition $R_{\Box}\subseteq I$. That is, for every object $a$ and  feature $x$, if $aR_{\Box} x$ (i.e.~if $a$ is endowed with $x$ according to agent $i$) then $aIx$ (i.e.~object $a$ indeed has feature $x$). This condition is arguably an appropriate rendering of factivity in the setting of polarity-based relational structures, which suggests that more modal epistemic principles might retain their intended interpretation even under a substantial generalization step such as the one from the classical (i.e.~Boolean) to the lattice-based setting. Indeed, this is also the case for {\em positive introspection},\label{pageref:positive introspection} which in the language of classical modal logic is formalized as  $\Box\phi\to \Box\Box\phi $ (which reads `if agent $i$ knows that $\phi$, then  agent $i$ knows that she knows $\phi$'). As is well known, this axiom is a Sahlqvist formula the first-order correspondent of which on Kripke frames $(W, R)$ is the condition $R\circ R\subseteq R$, namely that the relation $R$ is transitive. Again, the $\mathrm{L}_{\Box}$-sequent $\Box\Box\phi\vdash \Box \phi$ turns out (cf.~\cite[Proposition 4.3]{roughconcepts}) to correspond on polarity-based structures $\mathbb{F}$ to the first-order condition\footnote{With the aid of the notation $;_{I}$ for relational composition modulo the polarity relation $I$ (cf.\ \cite[Section 3.4]{roughconcepts} for the full definition) this condition can be succinctly captured as $R_{\Box}\subseteq R_{\Box}\ ;_{I} R_{\Box}$.} that reads: for every object $a$ and  feature $x$, if agent $i$ thinks that $a$ has feature $x$, then (agent $i$ must recognize $a$ as an example of what $i$ understands as an $x$-{\em object}, i.e.~as a member of $i$'s understanding of the  formal context generated by feature $x$, and hence) agent $i$ must attribute to  $a$ also all the features that, according to $i$, are shared by all $x$-objects. As in the case of factivity, one can argue that this condition is an appropriate rendering of the principle of positive introspection in the setting of polarity-based relational structures, since is clearly an internal coherence requirement which seeks to justify any given attribution of a feature to an object by linking it to the wider context of those (other) features that are consequences of the given attribution. Lastly, the notion of {\em omniscience}, stipulating that the agent knows everything that is the case, is classically captured by the axiom $p \rightarrow \Box p$ corresponding on Kripke frames to the 
first-order condition
 $R \subseteq \Delta$. On polarity-based structures $\mathbb{F}$, the $\mathrm{L}_{\Box}$-sequent $\phi \vdash \Box \phi$ (cf.~\cite[Proposition 4.3]{roughconcepts}) corresponds to the first-order condition $I \subseteq R_{\Box}$, indicating that, whenever an object has a feature, the agent knows this.

\paragraph{Lattice-based normal modal logic as the logic of rough concepts.} As discussed above, the interpretation of $\mathrm{L}_{\Box}$ as an epistemic logic of formal concepts, facilitated by the polarity-based semantics, extends coherently from the informal meaning of the defining clauses of $\Vdash$ and $\succ$ w.r.t.~to $\Box$-formulas, all the way to the preservation of the meaning of well known epistemic principles. However, the epistemic interpretation is not the only possible one;  in what follows, we give pointers to another family of possible interpretations, proposed in \cite{roughconcepts}, where polarity-based relational structures for the languages $\mathcal{L}_{\mathrm{PML}}$ and $\mathcal{L}_{\mathrm{DML}}$  (cf.~Example \ref{example:LE languages}) are used 
to generalize Rough Set Theory (RST) \cite{Pawlak} to the setting of {\em rough concepts}. The basic models in RST are pairs $(X,R)$, called \emph{approximation spaces}, with $X$ a non-empty set and $R$ an equivalence relation on $X$. The set $X$ is to be thought of as the \emph{domain of discourse} and $R$ as an \emph{indiscernibility relation}. 
The equivalence classes of $R$  establish the granularity of the discourse by setting the limits to the distinctions that can be drawn. This granularity is captured algebraically by 
the upper and lower approximation operators arising from approximation spaces, which, when applied to any given subset $T \subseteq X$, encode the available information about $T$ as follows. The lower approximation of $\underline{T}$ of $T$ consists of those elements whose $R$-equivalence classes are contained in $T$, while the upper approximation of $T$ consists of those elements whose $R$-equivalence classes have non-empty intersection with $T$. In other words, 
\[
\underline{T}: = \bigcup\{R[z]\mid z\in T \mbox{ and } R[z]\subseteq T\} \quad \mbox{ and }\quad \overline{T}: = \bigcup\{R[z]\mid z\in T\}. 
\]
The lower approximation $\underline{T}$ can be thought of as the set of all objects that are {\em definitely} in $T$, while the upper approximation $\overline{T}$ consists of those objects that are {\em possibly} in $T$. 

As the reader would have remarked,  an approximation space is nothing but a frame for the modal logic S5, and the lower and upper approximation of $T \subseteq S$ are obtained by applying the interior and closure operators given by the S5 box and diamond, respectively. This connection with modal logic has indeed not gone unnoticed in the literature and has been elaborated in e.g.~\cite{orlowska1994rough}, \cite{banerjee1996rough} and \cite{orlowska2013incomplete}. 

In \cite{roughconcepts}, {\em conceptual approximation spaces} were defined as polarity-based $\mathcal{L}_{\mathrm{PML}}$-structures $\mathbb{F} = (\mathbb{P}, R_\Box, R_{\Diamond})$  such that $\mathbb{P} = (A, X, I)$ is a polarity, and  $R_{\Box}\subseteq A\times X$ and $R_{\Diamond}\subseteq X\times A$ are (compatible, see Definition \ref{def:polarity:based:frm})  relations verifying the first-order conditions corresponding to the following  modal axioms: $\Box \phi\vdash \Diamond \phi$ ({\em seriality})\footnote{In the presence of reflexivity, seriality becomes redundant; however, for the sake of making the generalization more modular, in \cite{roughconcepts}  the basic framework of conceptual approximation spaces only requires seriality.}; $\Box \phi\vdash  \phi$ and $\phi\vdash \Diamond \phi$ ({\em reflexivity}); $\Box \phi\vdash  \Box\Box\phi$ and $\Diamond\Diamond\phi\vdash \Diamond \phi$ ({\em transitivity}); $\phi\vdash  \Box\Diamond\phi$ and $\Diamond\Box\phi\vdash \phi$ ({\em symmetry}).

Taken together, these conditions guarantee that $\mathbb{F}^+: = (\mathbb{P}, [R_\Box], \langle R_\Diamond\rangle)$ is a complete lattice-based algebra such that  $[R_\Box]$ and $\langle R_\Diamond\rangle$ are an interior and a closure operator respectively;  moreover, $\langle R_\Diamond\rangle$ is the left adjoint of $[R_\Box]$ (i.e.~$aR_{\Box} x$ iff $xR_{\Diamond} a$ for every $a\in A$ and $x\in X$).

Under the usual interpretation of $\mathbb{P} = (A, X, I)$ as a database, one possible way to understand $aR_\Box x$ or equivalently $xR_\Diamond a$ is `there is {\em evidence} that object $a$ has feature $x$', or `object $a$ {\em demonstrably} has feature $x$' (cf.~\cite[Section 5.1]{roughconcepts}). This intuitive understanding makes it plausible to assume that $R_\Box\subseteq I$. Recall that that $\Vdash$ for $\Box$-formulas and of $\succ$ for $\Diamond$-formulas (cf.~discussion in Section \ref{ssec:rel interpret modalities}) were defined as follows:

\begin{center}
	\begin{tabular}{lll}
		$ a \Vdash \Box\phi$ & iff & for all $x\in X$, if $x \succ \phi$ then $a R_{\Box} x$ \\
		$x \succ \Diamond\phi$ & iff &  for all $a\in A$, if $a \Vdash \phi $ then $ x R_\Diamond a$.
	\end{tabular}
\end{center}
\noindent Under the  interpretation discussed above, these clauses can be understood as saying that $\Box\phi$ is the concept the examples/members of which   are exactly those objects that {\em demonstrably} have all the features shared by  $\phi$-objects, and that $\Diamond\phi$ is the concept described by the features  which all  $\phi$-objects {\em demonstrably} have. 
Hence, $\Box\phi$ can be understood as the (sub)concept of the {\em certified members} of $\phi$, while $\Diamond\phi$ as the (super) concept of the {\em potential members} of $\phi$.

Thus, under the interpretation of $R_{\Box}$ and $R_{\Diamond}$ proposed above,  the polarity-based semantics of $\mathcal{L}_{\mathrm{PML}}$ supports the understanding of $\Box\phi$ and $\Diamond\phi$  as the lower and upper approximations of  concept $\phi$, respectively. Notice that, while in approximation spaces  the  relation $R$ relates indiscernible states, and thus directly encodes the extent of our   {\em ignorance}, in the setting of conceptual approximation spaces,  $R_\Box$ (or equivalently $R_\Diamond$) directly encode the (possibly partial) extent of our  {\em knowledge} or information. 

\paragraph{From  concepts to other ontologies.} In \cite{roughconcepts}, other more specific interpretations are proposed concerning situations which span from the analysis of text databases to medical diagnoses and the analysis of markets. Accordingly, in each of these situations, $\Box\phi$ and $\Diamond\phi$  can be given more specific interpretations as  lower and upper approximations of concepts or categories or relevant clusters. 

For instance (cf.~\cite[Section 5.4,]{roughconcepts} modified), text databases can be modelled as polarity-based $\mathcal{L}_{\mathrm{PML}}$-structures $\mathbb{F} = (\mathbb{P}, R_\Box, R_{\Diamond})$ such that  $\mathbb{P} = (A, X, I)$ with $A$ being a set of documents,  $X$  a set of words, and $aIx$ being understood as `document $a$ has word $x$ as a keyword'. Formal concepts arising from such an $\mathbb{F}$ can be understood as {\em themes} or {\em topics}, intensionally described by Galois-stable sets of (key)words. In this situation,  one of the many possible interpretations of 
$aR_\Box x$  or equivalently $xR_\Diamond a$ is `document $a$ has word  $x$  as its {\em first or second} keyword', which again makes it plausible to assume that $R_\Box\subseteq I$. 

As another example (cf.~\cite[Section 5.5]{roughconcepts} modified), let $\mathbb{P} = (A, X, I)$ represent  a hospital, where $A$ is the set of patients, $X$ is the set of symptoms, and $aIx$ iff  ``patient $a$ has symptom $x$''. Concepts arising from this representation are {\em syndromes}, intensionally described by Galois-stable sets of symptoms. In this situation, let   $aR_\Box x$, or equivalently $xR_\Diamond a$, iff `$a$ has been {\em tested} for symptom $x$ with positive outcome'. 

As a third example (cf.~\cite[Section 5.8]{roughconcepts} modified), let $\mathbb{P} = (A, X, I)$  where $A$ is the set of consumers, $X$ is the set of market-products, and $aIx$ iff  `consumer $a$ buys product $x$'. Concepts arising from this representation are {\em consumer segments}, intensionally described by Galois-stable sets of market-products. In this situation, let   $aR_\Box x$, or equivalently $xR_\Diamond a$, iff `$a$ buys $x$ from a certain producer $i$'.
Then  $\Box\phi$ denotes the market share of producer $i$ in consumer segment $\phi$.

As a fourth example, let $\mathbb{P} = (A, X, I)$  where $A$ is the set of empirical hypotheses, $X$ is the set of variables, and $aIx$ iff  `hypothesis $a$ is formulated in terms of variable $x$'. Concepts arising from this representation are empirical  {\em theories}, extensionally described by Galois-stable sets of hypotheses and intensionally described by Galois-stable sets of variables. In this situation, let   $aR_\Box x$, or equivalently $xR_\Diamond a$, iff `$x$ is a  {\em dependent} variable for hypothesis $a$'. Then   $\Box\mathbf{x}$ contains all hypotheses that  compete with each other. 

Finally, let $\mathbb{P} = (A, X, I)$ represent a decision-making situation in which  $A$ is the set of decision-makers, $X$ is the set of issues, and $aIx$ iff  `agent $a$ finds issue $x$ relevant'. Concepts arising from this representation are   {\em interrogative agendas}, extensionally described by Galois-stable {\em coalitions} and intensionally described by Galois-stable sets of issues. In this situation, let $aR_\Box x$, or equivalently $xR_\Diamond a$, iff `agent $a$ regards $x$ as a positive issue'. For example, if $A$ are the members of a hiring committee and $X$ the features of potential applicants, agent $a$ could regard ``the candidate obtained their PhD recently'' as a desirable characteristic, i.e.\ a positive issue, while other agents might prefer a more experienced candidate and therefore not regard this as positive. This would mean that $R_{\Box} \subseteq I$ and that, extensionally, $\Box \phi$ would be the coalition of all agents who are positive towards all issues on interrogative agenda $\phi$.

\subsection{LE-logics as logics of informational entropy}\label{ssec:Interpretation:Graphs}

\paragraph{Reflexive graphs as generalized intuitionistic frames.}
As discussed above, polarity-based semantics supports an interpretation of non-distributive logics as {\em logics of formal concepts} (for specific signatures: epistemic logic of concepts, logic of rough concepts etc). Under this interpretation, formulas do not denote states of affairs but rather are names for formal concepts in the sense of FCA. This interpretation  can be further specialized to entities such as categories, theories, and interrogative agendas. 

Graph-based semantics suggests quite another interpretation of nondistributive logics, in which formulas do denote states of affairs;   below we will argue that, under this interpretation,  non-distributive lattice logic can be understood as a hyper-constructivist logic which generalizes intuitionistic logic just like intuitionistic logic generalizes classical logic. 

As discussed in Section \ref{ssec: from alg to rel},  for any reflexive graph $\mathbb{X} = (Z, E)$, 
homomorphic assignments $\overline{v}:\mathrm{L}\to\mathbb{X}^+$
map $\mathrm{L}$-formulas $\varphi$ to tuples $(\val{\varphi}, \descr{\varphi})$ which, as discussed in Section \ref{ssec: projection}, reduce to $(\val{\varphi}, \val{\varphi}^c)$  when $E: = \Delta$. Hence,  $\val{\varphi} = \{z\in Z\mid z\Vdash\phi\}$ and $\descr{\varphi} = \{z\in Z\mid z\succ\phi\}$ can be respectively understood as  the {\em satisfaction}  and  {\em refutation set} of $\phi$ under $\overline{v}$. Since  $(\val{\varphi}, \descr{\varphi})$ is a Galois-stable pair, also when $E: = \Delta$,  the satisfaction and refutation set of a given formula completely determine each other via the identities $\descr{\varphi} = \val{\varphi}^{\uparrow}$ and $\val{\varphi} = \descr{\varphi}^{\downarrow}$. However, as we will see, in contrast with the classical and intuitionistic setting,  at given state $z$, the truth value of a given  formula $\phi$ can  be {\em undefined} (i.e.~$z\not\Vdash \phi$ and $z\not\succ\phi$). This {\em potential indeterminacy} of formulas at states of graph-based models is the main characterizing feature of this semantic setting, and can be understood as witnessing the failure of the principle of excluded middle not just anymore at the level of the object language (as is the case of intuitionistic logic) but at the more fundamental, {\em meta-linguistic} level of the satisfaction and refutation of formulas. This property of the graph-based semantics of non-distributive logics justifies our view of non-distributive logics as `hyper-constructivist' logics.

In order to discuss this generalization, it will be useful to first recall how the state-based semantics of intuitionistic logic  generalizes the state-based semantic of classical propositional logic. As discussed in Section \ref{ssec: from alg to rel}, a `relational structure' for classical propositional logic is a structure $\mathbb{F} = (S, \Delta)$ where $S$ a nonempty set and $\Delta\subseteq S\times S$ is the identity relation. As is well known, given an assignment $v:\mathsf{Prop}\to \mathbb{F}^+$, any  formula is either true or false at each state (but at no state can a formula be {\em both} true {\em and} false), it is true precisely when its negation is false, and its truth value only depends on the values of its occurring propositional variables at the given state. A relational structure for intuitionistic propositional logic is a structure $\mathbb{F} = (S, \leq)$ where $S$ a nonempty set and $\leq$ is a reflexive and transitive relation. In this case, for a given assignment $v:\mathsf{Prop}\to \mathbb{F}^+$, any  formula is again either true or false at each state (but never both true and false); however, the fact that a given formula is refuted at a given state does not imply that the negation of that formula is satisfied at that state; indeed, the given state   might  refute both the formula and its negation, yielding the well known failure of the classical principle of excluded middle. More generally, the {\em satisfaction}  of a formula at a given state might depend on the truth values of its occurring propositional variables at the {\em successors} of the given state. 
The characterizing properties of intuitionistic satisfaction, those that mark its difference from classical satisfaction, are all grounded on the fact that (homomorphic) assignments of proposition variables are {\em persistent}, in the sense of being upward-closed with respect to $\leq$, i.e.~the fact that $v(\phi)\in \mathbb{F}^+ \cong \mathcal{P}^{\uparrow}(S)$ for any formula $\phi$, and hence  if  $\phi$ is true at a given state $s$, it   will remain true along any forward-looking $\leq$-branch stemming from  $s$ (and dually, if $\phi$ is false at a given state $s$, it  was so throughout any backward-looking $\leq$-branch stemming from  $s$). 

Persistence supports our understanding of intuitionistic truth as an inherently {\em procedural} truth: under the  {\em higher standard} required by having to declare true at a given state only those statements that e.g.~are backed by evidence in support of their truth at that state, or for which a procedure has been completed at that state which effectively establishes their truth, the failure of the principle of excluded middle becomes a desirable feature, since, e.g.~at a given state, there might not be enough evidence in support of a given statement or in support of its negation; however, it is also reasonable to require that, once proven at a given state, a statement cannot be unproven,  therefore its (proven) truth persists at the successors of the given state. 
As mentioned above, when moving from the Boolean to the intuitionistic setting, the meaning of the implication becomes ``intensional'', in the specific sense that the satisfaction of $p\to q$ at any given state $s$ of an intuitionistic frame $(S, \leq)$ does not just depend on the value of $p$ and $q$ at $s$, but also on the values of $p$ and $q$ at the $\leq$-successors of $s$. This change in the interpretation of $\to$-formulas can be explained from a technical ground  as the consequence of translating the interpretation of $\to$-formulas from $\mathbb{F}^+$ to  $\mathbb{F} = (S, \leq)$, as discussed in Section \ref{ssec: from alg to rel}, when $\mathbb{F}^+$ is the perfect Heyting algebra $\mathcal{P}^{\uparrow}(S)$ rather than the perfect Boolean algebra $\mathcal{P}(S)$; however, this  interpretation  is also
coherent with our understanding of intuitionistic truth as a procedural truth, since, as is well known, if  a procedure is available for establishing that $\phi\to\psi$ is true at a given state $s$, then at any given evolution of $s$ (including $s$ itself) it will be possible to use this same procedure to transform a proof of $\phi$ (whenever available) into a proof of $\psi$, and conversely.

Having summarised the salient features of the state-based semantics of intuitionistic logic, the following proposition establishes how it can be seen as a special instance of graph-based semantics. 
\begin{prop}
\label{prop:transitive forces Heyting}
For any reflexive graph $\mathbb{X} = (Z, E)$,
\begin{enumerate} 
\item if $E$ is  transitive (i.e.~$E$ is a preorder), then $\mathbb{X}^+$ is a complete and completely distributive lattice (hence a perfect Heyting algebra). 
\item  if $E$ is antisymmetric (i.e.~$zEz'$ and $z'Ez$ imply $z = z'$ for any $z, z'\in Z$) and $\mathbb{X}^+$ is  completely distributive, then $E$ is transitive.\footnote{The antisymmetry assumption is necessary: in personal communication, Andrew Craig observed that $\mathbb{X} = (Z, E)$ with $Z: = \{u, v, z\}$ and $E: = Z\times Z\setminus \{(u, z)\}$ is an example of a reflexive, non antisymmetric and non-transitive graph such that  $\mathbb{X}^+$ is the 3-element chain (hence is distributive). We thank Apostolos Tzimoulis for suggesting the structure of the proof of item 2 of Proposition \ref{prop:transitive forces Heyting}.}
\end{enumerate}
\end{prop}

\begin{proof}
As to item 1, it is enough to show that the  Galois-stable sets of $\mathbb{P}_{\mathbb{X}}^+\subseteq \mathcal{P}(Z_A)$ (resp.~$\mathbb{P}_{\mathbb{X}}^+\subseteq \mathcal{P}(Z_X)$) are exactly the upward-closed (resp.~the downward-closed) subsets. To this end, it is enough to show that   $Y^{\uparrow\downarrow} = Y {\uparrow}: = \{z\in Z\mid x Ez$ for some $x\in Y\}$ for any $Y\subseteq Z_A$ (or dually, that $Y^{\downarrow\uparrow} = Y {\downarrow}: = \{z\in Z\mid z Ex$ for some $x\in Y\}$ for any $Y\subseteq Z_X$). 
Indeed, $Y^{\uparrow} = \{x\in Z_X\mid$ for all $a$, if $a\in Y$ then $aE^c x\} = {Y{\uparrow}}^c$. Then $Y^{\uparrow\downarrow}  = {{{Y{\uparrow}}^c}{\downarrow}}^{c}$. Given that $Y {\uparrow}$ is by definition an upward-closed set, ${Y {\uparrow}}^c$ is a downward-closed set, hence ${Y {\uparrow}}^c{\downarrow} = {Y {\uparrow}}^c$; hence, ${{Y{\uparrow}}^c{\downarrow}}^{c} = {{Y{\uparrow}}^c}^{c} = Y{\uparrow}$.

As to item 2, for every $z\in Z$, we let $\underline{\mathbf{z}} = \mathbf{z}_s = (z^{\uparrow\downarrow}, z^{\uparrow})$, and $\overline{\mathbf{z}} = \mathbf{z}_r = (z^{\downarrow}, z^{\downarrow\uparrow})$. To show that $E$ is transitive,  it is enough to show that, for any $u, z\in Z$,
\begin{equation}
\label{eq: sufficient for transitivity}
zEu \quad\mbox{ iff }\quad \underline{\mathbf{u}} \leq \underline{\mathbf{z}}.
\end{equation}
By definition,  $\underline{\mathbf{u}} \leq \underline{\mathbf{z}}$ iff $z^{\uparrow} \subseteq u^{\uparrow}$ iff $\{z'\in Z\mid u Ez'\} =(u^{\uparrow})^c \subseteq (z^{\uparrow})^c = \{z'\in Z\mid z Ez'\}$. Since $uEu$ by the reflexivity of $E$, this inclusion implies that $zEu$, as required.   Conversely, let us assume that $z E u$. Then $z\in z^{\uparrow\downarrow}$ and $z\notin u^{\downarrow}$, hence $\underline{\mathbf{z}} \nleq \overline{\mathbf{u}}$. Notice that  when $z = u$ the same argument yields $\underline{\mathbf{u}} \nleq \overline{\mathbf{u}}$. To finish the proof, let
{\footnotesize
\[\kappa (\underline{\mathbf{u}}) :  = \bigvee\{\underline{\mathbf{w}}\mid w\in Z\mbox{ and }\underline{\mathbf{u}}\nleq \underline{\mathbf{w}}\} = \bigvee\{\underline{\mathbf{w}}\mid w\in Z\mbox{ and } u\notin w^{\uparrow\downarrow}\} = \bigvee\{\underline{\mathbf{w}}\mid w\in Z\mbox{ and } w^{\uparrow}\nsubseteq u^{\uparrow}\}. \]
}
  It is enough to show that: (a)    $\underline{\mathbf{u}} \leq \underline{\mathbf{z}}$ iff $\underline{\mathbf{z}}\nleq \kappa(\underline{\mathbf{u}})$, and (b) $\overline{\mathbf{u}} = \kappa (\underline{\mathbf{u}})$. 
  
  Before addressing (a), let us preliminarily show that $\underline{\mathbf{z}}$ is completely join-prime for every $z\in Z$. Because by assumption $\mathbb{X}^+$ is completely distributive, it is enough to show that $\underline{\mathbf{z}}$ is completely join-irreducible, and for this latter claim, it is enough to show that $\underline{\mathbf{z}}\wedge \overline{\mathbf{z}}$ is the greatest element of $\mathbb{X}^+$ that is strictly less than $\underline{\mathbf{z}}$. 
  Since, as discussed above,  $\underline{\mathbf{z}} \nleq \overline{\mathbf{z}}$, it follows that $\underline{\mathbf{z}}\wedge \overline{\mathbf{z}}<\underline{\mathbf{z}}$.  Let us show that $\underline{\mathbf{z}}\wedge \overline{\mathbf{z}}$ is the greatest such element by showing that $\val{\underline{\mathbf{z}}\wedge \overline{\mathbf{z}}} = z^{\uparrow\downarrow}\setminus \{z\}$. By definition, $\val{\underline{\mathbf{z}}\wedge \overline{\mathbf{z}}} = \val{\underline{\mathbf{z}}}\cap \val{ \overline{\mathbf{z}}} = z^{\uparrow\downarrow}\cap z^{\downarrow}$. Hence, we are left to show that  $z^{\uparrow\downarrow}\cap z^{\downarrow}= z^{\uparrow\downarrow}\setminus \{z\}$. By the reflexivity of $E$, if $w\in  z^{\uparrow\downarrow}\cap z^{\downarrow}$ then $w\neq z$, and hence $w\in z^{\uparrow\downarrow}\setminus \{z\}$. Conversely, if $w\in z^{\uparrow\downarrow}\setminus \{z\}$, the antisymmetry of $E$ implies that either $wE^c z$ (i.e.~$w\in z^{\downarrow}$, which is what we require) or $z E^c w$ (i.e.~$w\in z^{\uparrow}$). This latter condition is excluded, again by the reflexivity of $E$ and the assumption that $w\in z^{\uparrow\downarrow}$. 

As to (a), the complete join-primeness of  $\underline{\mathbf{z}}$ implies that  $\underline{\mathbf{z}}\leq \kappa(\underline{\mathbf{u}}) = \bigvee\{\underline{\mathbf{w}}\mid w\in Z\mbox{ and }\underline{\mathbf{u}}\nleq \underline{\mathbf{w}}\}$ iff $\underline{\mathbf{z}}\in \{\underline{\mathbf{w}}\mid w\in Z\mbox{ and }\underline{\mathbf{u}}\nleq \underline{\mathbf{w}}\}$  iff $\underline{\mathbf{u}}\nleq \underline{\mathbf{z}}$, as required. 

Notice that (a) instantiated to $z = u$ yields  $\underline{\mathbf{u}} \nleq \kappa (\underline{\mathbf{u}})$.
Notice also that (a) can be strengthened to (a'): $\underline{\mathbf{u}} \leq \mathbf{c}$ iff $\mathbf{c}\nleq \kappa(\underline{\mathbf{u}})$ for all $u\in Z$ and $\mathbf{c}\in \mathbb{X}^+$. Indeed, by Proposition \ref{prop:join and meet generators}, $\bigvee\{\underline{\mathbf{z}}\mid \underline{\mathbf{z}}\leq \mathbf{c}\} = \mathbf{c}\leq \kappa(\underline{\mathbf{u}})$ iff $\underline{\mathbf{z}}\leq \kappa(\underline{\mathbf{u}})$ for every $z\in Z$ such that $\underline{\mathbf{z}}\leq \mathbf{c}$, iff, by (a), $\underline{\mathbf{u}}\nleq \underline{\mathbf{z}}$ for every $z\in Z$ such that $\underline{\mathbf{z}}\leq \mathbf{c}$. By the complete join-primeness of $\underline{\mathbf{u}}$, the last condition is equivalent to $\underline{\mathbf{u}}\nleq \mathbf{c}$.

These remarks imply  that  $\kappa(\underline{\mathbf{u}})$ is completely meet-irreducible. 
Indeed, let $Y\subseteq Z$ such that $\bigwedge \{\overline{\mathbf{z}}\mid z\in Y\} =  \kappa(\underline{\mathbf{u}})  $.  If $\overline{\mathbf{z}}> \kappa(\underline{\mathbf{u}})$ for every $z\in Y$, then $\overline{\mathbf{z}}\nleq\kappa(\underline{\mathbf{u}})$, i.e.~by (a'), $\underline{\mathbf{u}}\leq \overline{\mathbf{z}}$ for every $z\in Y$, and hence $\underline{\mathbf{u}} \leq \kappa (\underline{\mathbf{u}})$, contradicting $\underline{\mathbf{u}} \nleq \kappa (\underline{\mathbf{u}})$.

As to (b),  by Proposition \ref{prop:join and meet generators}, $\overline{\mathbf{u}} = \bigvee\{\underline{\mathbf{w}}\mid w\in Z\mbox{ and }\underline{\mathbf{w}}\leq \overline{\mathbf{u}}\} = \bigvee\{\underline{\mathbf{w}}\mid w\in Z\mbox{ and } w \in  u^{\downarrow}\}$. Hence, to show that $\overline{\mathbf{u}} \leq \kappa (\underline{\mathbf{u}})$, it is enough to show that, if $w\in u^{\downarrow}$,
 then  $u\notin w^{\uparrow\downarrow} = \{z\in Z\mid \forall z'(z'\in w^{\uparrow}\Rightarrow zE^c z')\}$. By definition, $w\in u^{\downarrow}$ iff $wE^c u$ iff $u\in w^{\uparrow}$. Moreover, by reflexivity, $u Eu$. Hence, $u\notin w^{\uparrow\downarrow}$, as required. Conversely, to show that $\kappa (\underline{\mathbf{u}})\leq \overline{\mathbf{u}} $, notice that Proposition \ref{prop:join and meet generators} and $\kappa (\underline{\mathbf{u}})$ being completely meet-irreducible imply that $\kappa (\underline{\mathbf{u}}) = \overline{\mathbf{w}}$ for some $w\in Z$. Hence, from $\underline{\mathbf{u}} \nleq \kappa (\underline{\mathbf{u}})   = \overline{\mathbf{w}}$ it follows that $u\notin w^{\downarrow}$ (i.e.~$uEw$), while from $\overline{\mathbf{u}} \leq \kappa (\underline{\mathbf{u}}) = \overline{\mathbf{w}}$ it follows that $w\in u^{\downarrow\uparrow}$, i.e.~$u^{\downarrow}\subseteq w^{\downarrow}$. Suppose for contradiction that $u\neq w$. Then, by antisymmetry, $uEw$ implies that $wE^c u$, i.e.~$w\in u^{\downarrow}\subseteq w^{\downarrow}$. Hence, $wE^cw$, contradicting the reflexivity of $E$.
\end{proof}
 
The proposition above shows that the Galois-stability of satisfaction (resp.~refutation) sets projects onto their being $\leq$-upward (resp.~$\leq$-downward) closed in the intuitionistic  setting. This establishes a tight link between, on the one hand, the transitivity of intuitionistic frames and, on the other,  a package of three characterizing properties of intuitionistic logic, namely  {\em persistence} of intuitionistic satisfaction, {\em distributivity} of intuitionistic $\wedge$ and $\vee$, and the relation of refutation coinciding with the one of {\em non-satisfaction}. 
 In the absence of transitivity, all  three properties in this package are expected to fail. 
 Indeed, let us illustrate this fact by considering the  reflexive (and antisymmetric but not transitive) graph $\mathbb{X} = (Z, E)$ shown in the left-hand side of the picture below. The polarity drawn in the centre of the picture is the polarity $\mathbb{P_X}$ associated with $\mathbb{X}$, and the (nondistributive)
 lattice on the right is  $\mathbb{X}^+ = \mathbb{P}_\mathbb{X}^+$. 
 
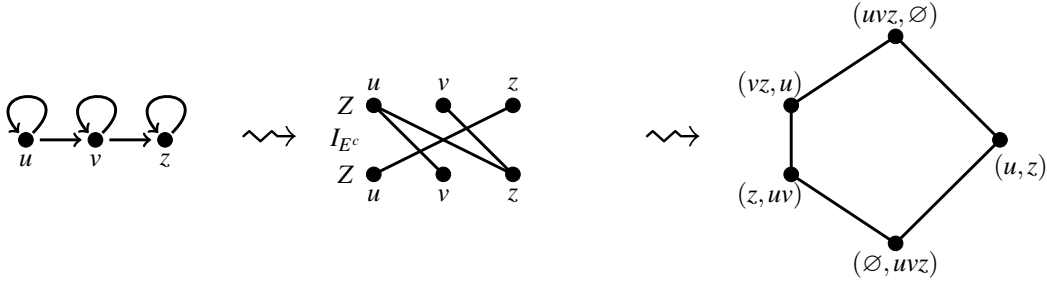
\begin{center}
\resizebox{.9\textwidth}{!}{
\begin{tikzpicture}
\draw[very thick] (7.5, 2.5) -- (6, 3.5) --
	(6, 4.5) -- (7.5, 5.5) -- (9, 4)  -- (7.5, 2.5);
	\filldraw[black] (7.5,2.5) circle (3 pt);
	\filldraw[black] (6,3.5) circle (3 pt);
	\filldraw[black] (6,4.5) circle (3 pt);
	\filldraw[black] (7.5,5.5) circle (3 pt);
	\filldraw[black] (9,4) circle (3 pt);
	\draw (7.5, 2.2) node {$(\varnothing,uvz)$};
	\draw (5.7, 3.2) node {$(z,uv)$};
	\draw (5.7, 4.8) node {$(vz,u)$};
	\draw (7.5, 5.8) node {$(uvz,\varnothing)$};
	\draw (9.3, 3.6) node {$(u,z)$};
    \draw (4.3,4) node{{\Huge{$\rightsquigarrow$}}};

\draw[very thick] (1, 3.5) -- (0, 4.5) --
	(2, 3.5) -- (1, 4.5);
	\draw[very thick] (0, 3.5) -- (2, 4.5);
	\filldraw[black] (0,3.5) circle (3 pt); 
	\filldraw[black] (0,4.5) circle (3 pt); 
	\filldraw[black] (1,3.5) circle (3 pt); 
	\filldraw[black] (1,4.5) circle (3 pt); 
	\filldraw[black] (2,3.5) circle (3 pt); 
	\filldraw[black] (2,4.5) circle (3 pt); 
    \draw (-0.4,4.5) node {$Z$};
    \draw (-0.4,4) node {$I_{E^c}$};
    \draw (-0.4,3.5) node {$Z$};
	\draw (0,4.8) node {$u$};
	\draw (1,4.8) node {$v$};
	\draw (2,4.8) node {$z$};
	\draw (0,3.2) node {$u$};
	\draw (1,3.2) node {$v$};
	\draw (2,3.2) node {$z$};
    
     \draw (-1.5,4) node{{\Huge{$\rightsquigarrow$}}};
     \filldraw[black] (-5,4) circle (3 pt);
     \draw (-5,3.7) node {$u$};
	\filldraw[black] (-4,4) circle (3 pt);
	 \draw (-4,3.7) node {$v$};
	\filldraw[black] (-3,4) circle (3 pt);
	 \draw (-3,3.7) node {$z$};
	\draw[very thick, ->] (-4.8, 4) -- (-4.2, 4);
	\draw[very thick, ->] (-3.8, 4) -- (-3.2, 4);
	\draw[very thick,  <-] (-5.1, 4.1)  .. controls (-5.7, 4.8) and  (-4.2, 4.8)  .. (-4.9, 4.1); 
\draw[very thick,  <-] (-4.1, 4.1)  .. controls (-4.7, 4.8) and  (-3.2, 4.8)  .. (-3.9, 4.1); 
\draw[very thick,  <-] (-3.1, 4.1)  .. controls (-3.7, 4.8) and  (-2.2, 4.8)  .. (-2.9, 4.1); 
\end{tikzpicture}
}
\end{center}

Let us consider a (homomorphic) assignment $\overline{v}:\mathrm{L}\to\mathbb{X}^+$  such that $\overline{v}(p) = (z, uv)$ and $\overline{v}(q) = (u, z)$. The following picture shows how this assignment translates into the interpretation of $p$ and $q$ on $\mathbb{X}$. We follow the convention that for every $y\in Z$ and formula $\phi$, if $y\Vdash \phi$  then $\phi$ appears above $y$, and if $y\succ \phi$  then $\phi$ appears below $y$.

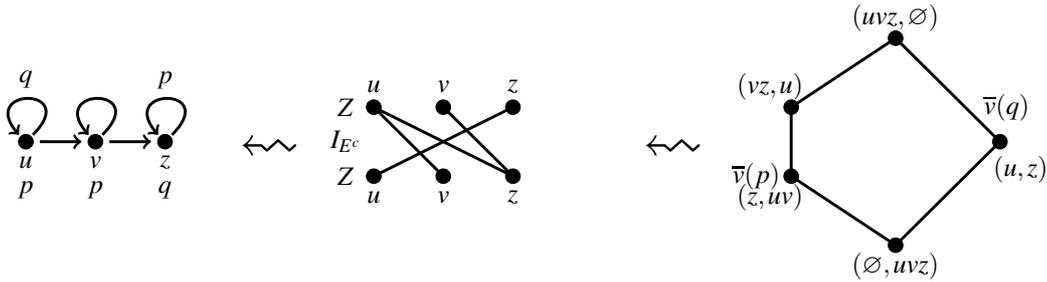
\begin{center}
\resizebox{.9\textwidth}{!}{
\begin{tikzpicture}
\draw[very thick] (7.5, 2.5) -- (6, 3.5) --
	(6, 4.5) -- (7.5, 5.5) -- (9, 4)  -- (7.5, 2.5);
	\filldraw[black] (7.5,2.5) circle (3 pt);
	\filldraw[black] (6,3.5) circle (3 pt);
	\filldraw[black] (6,4.5) circle (3 pt);
	\filldraw[black] (7.5,5.5) circle (3 pt);
	\filldraw[black] (9,4) circle (3 pt);
	\draw (7.5, 2.2) node {$(\varnothing,uvz)$};
	\draw (5.7, 3.2) node {$(z,uv)$};
	\draw (5.7, 4.8) node {$(vz,u)$};
    \draw (5.5, 3.5) node {{$\overline{v}(p)$}};
	\draw (7.5, 5.8) node {$(uvz,\varnothing)$};
	\draw (9.3, 3.6) node {$(u,z)$};
    \draw (9.1, 4.5) node {$\overline{v}(q)$};
    \draw (4.3,4) node{\Huge{$\leftsquigarrow$}};

\draw[very thick] (1, 3.5) -- (0, 4.5) --
	(2, 3.5) -- (1, 4.5);
	\draw[very thick] (0, 3.5) -- (2, 4.5);
	\filldraw[black] (0,3.5) circle (3 pt); 
	\filldraw[black] (0,4.5) circle (3 pt); 
	\filldraw[black] (1,3.5) circle (3 pt); 
	\filldraw[black] (1,4.5) circle (3 pt); 
	\filldraw[black] (2,3.5) circle (3 pt); 
	\filldraw[black] (2,4.5) circle (3 pt); 
    \draw (-0.4,4.5) node {$Z$};
    \draw (-0.4,4) node {$I_{E^c}$};
    \draw (-0.4,3.5) node {$Z$};
	\draw (0,4.8) node {$u$};
	\draw (1,4.8) node {$v$};
	\draw (2,4.8) node {$z$};
	\draw (0,3.2) node {$u$};
	\draw (1,3.2) node {$v$};
	\draw (2,3.2) node {$z$};
    
     \draw (-1.5,4) node{\Huge{$\leftsquigarrow$}};
     \filldraw[black] (-5,4) circle (3 pt);
     \draw (-5,3.7) node {$u$};
      \draw (-5, 4.9) node {{$q$}};
	  \draw (-5, 3.3) node {{$p$}};
	\filldraw[black] (-4,4) circle (3 pt);
	 \draw (-4,3.7) node {$v$};
	   \draw (-4, 3.3) node {{$p$}};
	\filldraw[black] (-3,4) circle (3 pt);
	 \draw (-3,3.7) node {$z$};
	 \draw (-3, 4.9) node {{$p$}};
	  \draw (-3, 3.3) node {{$q$}};

	\draw[very thick, ->] (-4.8, 4) -- (-4.2, 4);
	\draw[very thick, ->] (-3.8, 4) -- (-3.2, 4);
	\draw[very thick,  <-] (-5.1, 4.1)  .. controls (-5.7, 4.8) and  (-4.2, 4.8)  .. (-4.9, 4.1); 
\draw[very thick,  <-] (-4.1, 4.1)  .. controls (-4.7, 4.8) and  (-3.2, 4.8)  .. (-3.9, 4.1); 
\draw[very thick,  <-] (-3.1, 4.1)  .. controls (-3.7, 4.8) and  (-2.2, 4.8)  .. (-2.9, 4.1); 
\end{tikzpicture}
}
\end{center}
The left-hand side of the picture above shows that indeed, 
the formula $q$ which is true at $u$ becomes first indeterminate at $v$ and then false at $z$. Hence, non-distributive satisfaction does not need to be persistent, and non-distributive refutation is {\em different} from  non-satisfaction, since $v\not\Vdash q$ and $v\not\succ q$. Together with the failure of distributivity (which in the example above is yielded by the fact that $\mathbb{X}^+$ is nondistributive), these are the key differences setting apart non-distributive evaluation from  intuitionistic evaluation. Their role in our conceptual interpretation of this semantics will be discussed below.  
So far, we have discussed these differences  in terms of properties holding in intuitionistic models but failing  in graph-based models.
The following proposition identifies a property holding in graph-based models. We will refer to it as  {\em weak persistence}. 
\begin{prop}
For every (reflexive) graph $\mathbb{X} = (Z, E)$ any $\mathrm{L}$-formula $\phi$, any $z, z'\in Z$ and any homomorphic assignment $\overline{v}:\mathrm{L}\to \mathbb{X}^+$, if $zEz'$ and $z\Vdash \phi$, then $z'\not \succ\phi$.
\end{prop}
\begin{proof}
By definition, $\overline{v}(\phi) = (\val{\phi}, \descr{\phi})\in \mathbb{X}^+$, and hence $z'\succ\phi$ iff $z'\in\descr{\phi} = \val{\phi}^{\downarrow} = \{z'\in Z\mid $ for all $z$, if $z\in \val{\phi}$ then $zE^c z'  \}$. The assumptions imply that some $z\in Z$  exists such that $z\in \val{\phi}$ and $zEz'$.
\end{proof}
From this proposition it follows that weak persistence projects onto persistence in the intuitionistic setting: indeed, as shown in the proof of item 1 of Proposition \ref{prop:transitive forces Heyting}, if $E$ is transitive, then $\val{\phi} = \descr{\phi}^{c}$, i.e.\ $z'\not\succ\phi$ iff $z'\Vdash \phi$ and so, if $z \Vdash \phi$ and $zEz'$ then $z' \Vdash \phi$, in every reflexive and transitive graph-based model.

However, as its name suggests, in the wider context of reflexive but not necessarily transitive graph-based models, weak persistence is strictly weaker than persistence, and hence this notion captures the difference between intuitionistic and non-distributive evaluation from yet another angle. Namely, in intuitionistic frames, any given state is bound to {\em accept} any formula supported by  any of its predecessors, including itself, while in graph-based frames, any given state is only bound to {\em not reject} any formula supported by  any of its predecessors, including itself; as discussed above, because of potential indeterminacy, not rejecting does not necessarily imply accepting.  Hence, together with reflexivity, weak persistence guarantees that, at any given state, any given formula cannot be both accepted and rejected.  

As discussed above, intuitionistic truth can be construed as a procedural truth, so that $z\Vdash \phi$ can be read e.g.~as `at $z$,  some finite procedure has effectively established that $\phi$ is the case'. We also observed that, in intuitionistic frames, the refutation relation $\succ$ coincides with the relation $\not\Vdash$ of {\em non-satisfaction}; that is, under the previous reading of the intuitionistic satisfaction of a proposition,  $z\succ \phi$ iff `at $z$, no finite procedure has effectively established (yet) that $\phi$ is the case'. Non-distributive truth can be then construed along these same lines, but requiring an even {\em higher standard}\footnote{For the sake of the telescopic progression of this presentation, we have discussed a possible interpretation of non-distributive evaluation which allows it to be described as being {\em stricter} than the intuitionistic one. However, other readings might suggest that it is just different: for instance, one can read $z\Vdash \phi$ as `$z$ provides an argument/evidence supporting $\phi$', and $z\succ \phi$ as `$z$ provides an argument/evidence against  $\phi$'. The essential difference lies in the fact that, in the non-distributive setting, $\Vdash$ and $\succ$ determine each other in a way that is defined in terms of $E$, and, in contrast to the intuitionistic setting, cannot be simply reduced to identifying one with the complement of the other. } than intuitionistic truth:  for instance,  to conclude that $\phi$ is refuted at $z$, the {\em absence} of a procedure effectively establishing $\phi$ is not enough;
one needs {\em there to be} a procedure effectively {\em disproving} $\phi$.

Above, we remarked that, when moving from classical to intuitionistic logic, the interpretation  of $\to$ changes and becomes intensional, in the sense that  the truth value of $p\to q$ at a given state $s$ does not just depend on the value of $p$ and $q$ at that state, but also on the truth values of $p$ and $q$ at the $\leq$-successors of $s$.  Similarly, in moving from the intuitionistic to the non-distributive setting, conjunction and disjunction become ``intensional'' in the same sense, i.e.~that the satisfaction of $p\vee q$ (and dually, the refutation of $p\wedge q$) at a given state $s$ does not just depend on the value of $p$ and $q$ at that state, but also on the truth values of $p$ and $q$ at the $E$-successors of that state. 
Indeed, under the assignment above, $\overline{v}(p\vee q) = (uvz, \varnothing)$. The following picture shows how this assignment translates as  satisfaction/refutation of $p\vee q$ at  states of $\mathbb{X}$. 
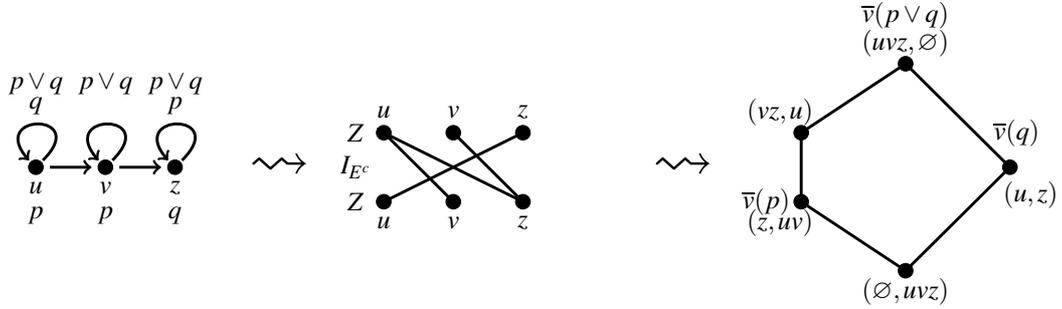
\begin{center}
\resizebox{.9\textwidth}{!}{
\begin{tikzpicture}
\draw[very thick] (7.5, 2.5) -- (6, 3.5) --
	(6, 4.5) -- (7.5, 5.5) -- (9, 4)  -- (7.5, 2.5);
	\filldraw[black] (7.5,2.5) circle (3 pt);
	\filldraw[black] (6,3.5) circle (3 pt);
	\filldraw[black] (6,4.5) circle (3 pt);
	\filldraw[black] (7.5,5.5) circle (3 pt);
	\filldraw[black] (9,4) circle (3 pt);
	\draw (7.5, 2.2) node {$(\varnothing,uvz)$};
	\draw (5.7, 3.2) node {$(z,uv)$};
	\draw (5.7, 4.8) node {$(vz,u)$};
    \draw (5.5, 3.5) node {{$\overline{v}(p)$}};
	\draw (7.5, 5.8) node {$(uvz,\varnothing)$};
	\draw (7.5, 6.2) node {$\overline{v}(p\vee q)$};
	\draw (9.3, 3.6) node {$(u,z)$};
    \draw (9.1, 4.5) node {$\overline{v}(q)$};
    \draw (4.3,4) node{{\Huge{$\rightsquigarrow$}}};

\draw[very thick] (1, 3.5) -- (0, 4.5) --
	(2, 3.5) -- (1, 4.5);
	\draw[very thick] (0, 3.5) -- (2, 4.5);
	\filldraw[black] (0,3.5) circle (3 pt); 
	\filldraw[black] (0,4.5) circle (3 pt); 
	\filldraw[black] (1,3.5) circle (3 pt); 
	\filldraw[black] (1,4.5) circle (3 pt); 
	\filldraw[black] (2,3.5) circle (3 pt); 
	\filldraw[black] (2,4.5) circle (3 pt); 
    \draw (-0.4,4.5) node {$Z$};
    \draw (-0.4,4) node {$I_{E^c}$};
    \draw (-0.4,3.5) node {$Z$};
	\draw (0,4.8) node {$u$};
	\draw (1,4.8) node {$v$};
	\draw (2,4.8) node {$z$};
	\draw (0,3.2) node {$u$};
	\draw (1,3.2) node {$v$};
	\draw (2,3.2) node {$z$};
    
     \draw (-1.5,4) node{{\Huge{$\rightsquigarrow$}}};
     \filldraw[black] (-5,4) circle (3 pt);
     \draw (-5,3.7) node {$u$};
      \draw (-5, 4.9) node {{$q$}};
	  \draw (-5, 3.3) node {{$p$}};
	\filldraw[black] (-4,4) circle (3 pt);
	 \draw (-4,3.7) node {$v$};
	 \draw (-4, 5.2) node {{$p\vee q$}};
	  \draw (-5, 5.2) node {{$p\vee q$}};
	   \draw (-4, 3.3) node {{$p$}};
	\filldraw[black] (-3,4) circle (3 pt);
	 \draw (-3,3.7) node {$z$};
	 \draw (-3, 4.9) node {{$p$}};
	  \draw (-3, 3.3) node {{$q$}};
	    \draw (-3, 5.2) node {{$p\vee q$}};

	\draw[very thick, ->] (-4.8, 4) -- (-4.2, 4);
	\draw[very thick, ->] (-3.8, 4) -- (-3.2, 4);
	\draw[very thick,  <-] (-5.1, 4.1)  .. controls (-5.7, 4.8) and  (-4.2, 4.8)  .. (-4.9, 4.1); 
\draw[very thick,  <-] (-4.1, 4.1)  .. controls (-4.7, 4.8) and  (-3.2, 4.8)  .. (-3.9, 4.1); 
\draw[very thick,  <-] (-3.1, 4.1)  .. controls (-3.7, 4.8) and  (-2.2, 4.8)  .. (-2.9, 4.1); 
\end{tikzpicture}
}
\end{center} 
In the example above, $p\vee q$ is true at state $v$ although neither $p$ nor $q$ are. More specifically, $p$ is false at $v$ and $q$ is neither true nor false. 
Indeed, the  defining clauses of $\succ$ and $\Vdash$ for $\vee$-formulas (cf.~discussion in Section \ref{ssec: from alg to rel}), repeated below for the reader's convenience,

{\footnotesize
\begin{center}
\begin{tabular}{lllllll}
		$z \succ \phi \vee \psi$ &iff &$z\succ \phi \text{ and } z \succ \psi$ &&
		$z \Vdash \phi \vee \psi$ &iff &for all $z'$, if $zEz'$ then $z'\not\succ\phi\vee\psi$\\
		$z \Vdash \phi \wedge \psi$ &iff &$z\Vdash \phi \text{ and } z \Vdash \psi$ &&
		$z \succ \phi \wedge \psi$ &iff &for all $z'$, if $zEz'$ then $z'\not\Vdash\phi\wedge\psi$\\
\end{tabular}
\end{center}
}
\noindent  can be understood as saying that a $\vee$-formula (resp.~$\wedge$-formula) is refuted (resp.~satisfied) at a state iff both immediate subformulas are refuted (resp.~satisfied) at that state (which is verbatim the intuitionistic and even the {\em classical} refutation condition for $\vee$-formulas and satisfaction condition for $\wedge$-formulas), and that a $\vee$-formula (resp.~$\wedge$-formula) is satisfied (resp.~refuted) at a given state iff none of its successors (including itself)   refutes (resp.~satisfies)  it, i.e.~none of its successors refutes (resp.~satisfies) both its immediate subformulas. The case of disjunction in the example above is then clear-cut:
 there is no state which refutes both $p$ and $q$, and hence the satisfaction condition of $p\vee q$ is vacuously verified by all states of the graph above. However, to mark the difference with the distributive interpretation of disjunction, $p\vee q$ is still true at $v$ even if: (a) the falsehood of $p$ is positively established at $v$; (b) at $v$, the truth of $q$ {\em can} be excluded, given that $v$ does not access any state at which $q$ is true, and (c) at $v$, the falsehood of $q$ {\em cannot} be excluded, given that $v$  accesses a state at which $q$ is false.

Below, we discuss some possible conceptual interpretations of the graph-based semantics which will account for the intensional meaning of $\vee$ in a hyper-constructivist context in which effective procedures are needed both for accepting and for rejecting a proposition. 

\paragraph{Reflexive graphs and the inherent  bounds of knowability.} 
In \cite{graph-based-wollic}, it is argued that, based on their graph-based semantics, LE-logics can be taken as the logics of {\em informational entropy}, where informational entropy is understood as an inherent boundary to knowability due e.g.~to perceptual, theoretical, technological, evidential or linguistic limits. The starting point of this interpretation is to regard reflexive graphs $\mathbb{X} = (Z, E)$ in analogy with approximation spaces (cf.~discussion in penultimate paragraph of Section \ref{ssec:Interpretation:Polarities}), and hence to regard  $E$   as an {\em indiscernibility} 
relation\footnote{In what follows, we  propose that $zEz'$ can be sometimes interpreted as $z$ {\em is consistent with} $z'$, or $z$ {\em does not exclude} $z'$.}. However, while indiscernibility is modelled as an equivalence relation in  the best known settings  in rough set theory and epistemic logic (e.g.~\cite{Pawlak,halpern}),  transitivity fails in concrete instances in which e.g.~states are indiscernible when their distance is closer than a certain threshold, and, as it has been argued in the psychological literature (cf.~\cite{tversky1977features,nosofsky1991stimulus}), symmetry fails in situations in which indiscernibility is defined in terms of a relation of similarity where e.g.~$z$ being similar to $z'$ does not necessarily imply that $z'$ is similar to $z$.  
Hence, under this broader understanding, {\em reflexivity}, i.e.~$\Delta\subseteq E$, is the minimal requirement of an indiscernibility relation\footnote{These objections to modelling indiscernibility as an equivalence relation have also been addressed in rough set theory and epistemic logic, see e.g.~\cite{yao1998interpretations, baltag2016logic}}. The limit case in which $E := \Delta$ is the (classical) case in which there are no boundaries to perfect knowability. So,  the generalization from the classical to the non-distributive propositional environment is captured by the graph-based semantics  as the logical internalization of  the {\em  shift} from $\Delta$ to $E$. 

As discussed above, rather than generating modal operators (e.g.~upper and lower approximations, as in rough set theory), $E$ is used to generate the complete lattice $\mathbb{X}^+$ as the concept lattice of the polarity $\mathbb{P_X} = (Z_A,Z_X,I_{E^c})$,  the elements of which are the closures of subsets of $Z$ (representing states of affairs) under ‘all is knowable about them’, i.e.~the theoretical horizon to knowability, given the inherent boundary encoded into $E$. 

As discussed above about the persistence property of intuitionistic satisfaction, in the non-distributive setting, the requirement of   {\em Galois-stability}  is not only mathematically justified by the need of defining a compositional semantics for $\mathrm{L}$, but can also be understood at a more fundamental level: if $E$ encodes an inherent limit to knowability, this limit should be incorporated in the meaning of formulas which are both satisfied and refuted `up to $E$', i.e.~the semantic representation of each formula  should  not be given in terms of arbitrary subsets of the domain of the graph, but  only in terms of those subsets which are preserved (i.e.~faithfully translated) in the shift from $\Delta$ to $E$; these are exactly the Galois-stable sets. In particular, the closure $a^{\uparrow\downarrow}$ of any $a\in Z$ arises by first considering the set $a^{\uparrow}$ of all the states from which $a$ is not indiscernible, and then the set of all the states that can be told apart from every state in $a^{\uparrow}$. Then  $a$ is clearly an element of $a^{\uparrow\downarrow}$, but other states $b$ in $Z$ might be as well, so this is as far as we can go: $a^{\uparrow\downarrow}$ represents the \emph{horizon}, defined in terms of $E$, to the possibility of completely `knowing' $a$. This horizon might be epistemic, cognitive, technological, or evidential. 
Under this understanding of $E$, the following defining clauses of $\succ$ and $\Vdash$ for $\vee$-formulas (cf.~discussion in Section \ref{ssec: from alg to rel})
{\footnotesize
\begin{center}
\begin{tabular}{lllllll}
		$z \succ \phi \vee \psi$ &iff &$z\succ \phi \text{ and } z \succ \psi$ &&
		$z \Vdash \phi \vee \psi$ &iff &for all $z'$, if $zEz'$ then $z'\not\succ\phi$ or $z'\not\succ\psi$\\
\end{tabular}
\end{center}
}
\noindent  reflect a kind of {\em conservative} or  {\em cautious} interpretation of the meaning of $\vee$: indeed, if states of a reflexive graphs represent e.g.~epistemic situations, or hypothetical scenarios or alternative theories, then a given scenario $s$ refutes a $\vee$-formula iff that scenario (contains enough evidence or explanatory power so that it)  can refute both disjuncts; however, if no scenario that can be accessed from $s$ (contains enough evidence so that it) can refute {\em both disjuncts}, then $s$ must accept the given $\vee$-formula. 

Concrete examples in which this interpretation of $\vee$ is arguably closer to reality than the classical or intuitionistic one arise e.g.~in {\em legal} domains. Consider for instance the Rashomon-type story of a judge who is to establish whether a (female) defendant is to be declared  not guilty 
of the physical injuries sustained by a (male) friend of hers. To simplify matters, the defendant is {\em not guilty} if and only if $p\vee q$, where $p$ stands for `the defendant has not willingly caused harm to her friend' and $q$ stands for `the defendant acted in self-defence'. In the trial, three witnesses testify as follows: Witness A: ``I saw the defendant grabbing a tennis racket and hitting her friend. She looked terrified and didn't utter a sound.'' Witness B: ``I saw the defendant grabbing a tennis racket and hitting her friend. She looked frightened, but not necessarily by her friend.'' Witness C: ``I heard the defendant scream that there was a poisonous spider on her friend's shoulder, so she  killed the spider.'' 
The following graph summarizes the information gathered by the judge:
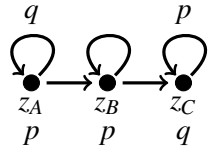
\begin{center}
\begin{tikzpicture}
  \filldraw[black] (-5,4) circle (3 pt);
     \draw (-5,3.7) node {$z_A$};
      \draw (-5, 4.9) node {{$q$}};
	  \draw (-5, 3.3) node {{$p$}};
	\filldraw[black] (-4,4) circle (3 pt);
	 \draw (-4,3.7) node {$z_B$};
	   \draw (-4, 3.3) node {{$p$}};
	\filldraw[black] (-3,4) circle (3 pt);
	 \draw (-3,3.7) node {$z_C$};
	 \draw (-3, 4.9) node {{$p$}};
	  \draw (-3, 3.3) node {{$q$}};

	\draw[very thick, ->] (-4.8, 4) -- (-4.2, 4);
	\draw[very thick, ->] (-3.8, 4) -- (-3.2, 4);
	\draw[very thick,  <-] (-5.1, 4.1)  .. controls (-5.7, 4.8) and  (-4.2, 4.8)  .. (-4.9, 4.1); 
\draw[very thick,  <-] (-4.1, 4.1)  .. controls (-4.7, 4.8) and  (-3.2, 4.8)  .. (-3.9, 4.1); 
\draw[very thick,  <-] (-3.1, 4.1)  .. controls (-3.7, 4.8) and  (-2.2, 4.8)  .. (-2.9, 4.1); 
\end{tikzpicture}
\end{center}
The states of the graph above represent the different testimonies of the three witnesses; the arrows represent the accessibility relation among states, where  e.g.~$z_A Ez_B$ means `A's testimony  {\em is consistent with} B's testimony', and $z_B Ez_C$ means `B's testimony {\em is consistent with} C's' (notice, however, that $z_B$ is arguably not consistent with $z_A$, since the evidence brought by $z_B$ challenges the {\em conclusions} that can be drawn on the base of the evidence brought by $z_A$, and by the same argument, $z_C$ is not consistent with $z_B$). Finally,  $z_A$ and $z_C$ are mutually inconsistent, since they challenge each other's evidence, and each testimony is consistent with itself. Witness A's testimony offers evidence against $p$ (A saw the defendant hitting her friend with a tennis racket) and evidence supporting  $q$ (she was frightened); Witness B's testimony offers evidence against $p$ (B too saw the defendant hitting her friend with a tennis racket) while not offering any evidence in support of or against $q$ (according to B, she might or might not have been frightened by her friend); Witness C's testimony offers evidence against $q$ and supporting $p$. The judge is not in a position  to establish whether the witnesses are lying or not, or which witness is right; however, the judge's task is to decide on $p\vee q$ on the basis of the available pieces of evidence and how these pieces fit with each other. There is no witness that provides enough evidence to refute both $p$ and $q$, hence, no matter how different their versions are, all testimonies lead to the acceptance of a not guilty verdict. 

\paragraph{LE-logics as epistemic logics of information entropy.} In Section \ref{ssec:Interpretation:Polarities}, the first proposal we discussed of an informal understanding of the modal operators on polarity-based semantics was  {\em epistemic}, and was based on the possibility of taking the incidence relation $I$ of any polarity as {\em objectively}  encoding which objects have which features, while, in contrast, the relations $R_{\Box}$ and $R_{\Diamond}$ were understood as representing the {\em subjective} perspective of e.g.~an agent on the same issues. Likewise, as discussed in \cite{graph-based-wollic}, in the graph-based setting, additional relations on graphs-based frames can be regarded as encoding  {\em subjective indiscernibility}, i.e.~$zR_\Box y$ can be understood as  `$z$ is indiscernible from $y$ {\em according to a given agent}', as opposed to the objective, or inherent indiscernibility encoded by $E$.  Under this interpretation, the following defining clauses of $\succ$ and $\Vdash$ for $\Box$-formulas and $\Diamond$-formulas  (cf.~discussion in Section \ref{ssec:rel interpret modalities}),

{\footnotesize
\begin{flushleft}
\begin{tabular}{lllllll}
		$z \succ \Diamond\phi$ &iff & for all $z'$, if  $z R_{\Diamond}z'$  then $z' \not \Vdash \phi$&&
		$z \Vdash \Diamond\phi$ &iff &for all $z'$, if  $z Ez' $  then $z' \not\succ \Diamond \phi$\\
		$z \Vdash \Box \psi$ &iff &for all $z'$, if  $z R_{\Box}z'  $  then $z' \not\succ \psi$&&
		$z \succ \Box\psi$ &iff &for all $z'$, if  $z' Ez  $  then $z' \not\Vdash \Box \psi$\\
		\end{tabular}
		\end{flushleft}
		}
\noindent  can be understood as saying that
 $\Box\phi$ is satisfied at a given state if no state  that the agent considers indistinguishable from the given one refutes $\phi$, and  $\Diamond\phi$ is refuted at a given state if no state  that the agent considers indistinguishable from the given one accepts $\phi$.
Hence, under the interpretation indicated above, these semantic clauses  support the usual reading of   $\Box\phi$ as `the agent knows/believes $\phi$' and $\Diamond \phi$ as `the agent considers $\phi$ plausible'. 

In Section \ref{ssec:Interpretation:Polarities}, we argued that the epistemic interpretation of $\Box$ in the polarity-based semantics carries to the preservation of the meaning of well known epistemic principles.
Let us now finish this section by  discussing that, also in the setting of graph-based semantics, the epistemic interpretation of modal operators coherently preserves the meaning of the same epistemic principles.
In  \cite[Proposition 4]{graph-based-wollic}), the first-order correspondents of  well known modal axioms from epistemic logic have been computed, which turn out to be the parametrized `$E$-counterparts' of the first-order correspondents on Kripke frames. That is, the factivity axiom $\Box\phi\vdash \phi$, which, on classical Kripke frames corresponds to reflexivity ($\Delta\subseteq R_{\Box}$), corresponds to the first-order condition $E\subseteq R_\Box$ on graph-based frames, which in  \cite{graph-based-wollic} is referred to as $E$-{\em reflexivity}; the omniscience axiom $\phi\vdash \Box\phi$, which, on classical Kripke frames corresponds to the first-order condition $R_{\Box}\subseteq\Delta$, corresponds to its $E$-counterpart $R_\Box \subseteq E$ on graph-based frames; the positive introspection axiom $\Box\phi\vdash \Box\Box\phi$, which, on classical Kripke frames corresponds to transitivity ($R_{\Box}\circ R_{\Box}\subseteq R_{\Box}$), corresponds to the first-order condition $R_\Box\bullet_E R_\Box\subseteq R_\Box$ on graph-based frames
\footnote{For any graph $\mathbb{X} = (Z,E)$ and relations $R, S \subseteq Z \times Z$, the relation $R \bullet_{E} S\subseteq Z\times Z$ is defined as follows: for any $a, x\in Z$,
	\[
	a (R \bullet_{E} S) x \quad \text{iff} \quad \exists y(aR y \; \&\; \forall b(bEy\Rightarrow b S x)). 
	\]

If $E = \Delta$, then $\{b\mid bEy\} =  \{y\}$ for every $y\in Z$, hence   $(R \bullet_{E} S)$ reduces to $R\circ S$. } which in  \cite{graph-based-wollic} is referred to as $E$-{\em transitivity}. So it seems possible to capture the shift from $\Delta$ to $E$ also at the level of the first-order correspondents, by establishing a sort of systematic replacement of the role of $\Delta$ with $E$.\footnote{Mutatis mutandis, the same phenomenon is observable in the polarity-based setting. More about this in the conclusions.} This  similarity in shape is also what guarantees that the intended meaning of the epistemic principles is preserved in the translation from $\Delta$ to $E$. Indeed, as discussed in \cite{graph-based-wollic}, the $E$-reflexivity condition $E\subseteq R_\Box$   requires that if the agent is able to distinguish two given states, then these states are not inherently indistinguishable. That is, the agent's assessments are correct, which preserves the meaning of {\em factivity} modulo informational entropy. Similarly, the condition $R_\Box\subseteq E$ requires the agent to be able to distinguish  any two  states that are not inherently indistinguishable, which is indeed what an {\em omniscient} (but not all-powerful) agent should be able to do. Finally, the condition $R_\Box\bullet_E R_\Box\subseteq R_\Box$, i.e.~$ \forall a\forall x\left[\exists y(aR_\Box y \; \&\; \forall b(bEy\Rightarrow b R_\Box x))\Rightarrow aR_\Box x\right]$ requires that if the agent cannot distinguish  $y$ from  $a$ and  $x$ from any state from which $y$ is inherently indistinguishable, then she cannot distinguish $x$ from $a$. Equivalently, if the agent can distinguish $x$ from $a$, then ($x$ does not belong to the set $Y$ of  states indistinguishable from $a$ according to the agent, and hence) every state $y$ which belongs to $Y$ 
must be inherently indistinguishable from some state $b$ from which she can distinguish $x$. That is, if $x\notin Y$ as above, then for every $y\in Y$ the agent must  be able to find an `$E$-proxy' $b$ of $y$ from which she can tell $x$ apart. Hence,  the first-order correspondent of positive introspection on graphs requires that, if the agent can distinguish $x$ from $a$, there is a witness  $b$  (a `justification') for $x$ to be distinct from any  state that the agent cannot distinguish from $a$. Similarly to the first-order condition corresponding to positive introspection on polarities (cf.~page \pageref{pageref:positive introspection}), this requirement for a justification is indeed what positive introspection is about.

\section{Conclusions and further directions}\label{sec:Conclusions}
The present paper discussed ongoing research directions investigating the mathematical theory and conceptual implications of relational (non-topological)  semantics for non-distributive logics, with a special focus on logics which are algebraically captured by varieties of normal lattice expansions in arbitrary signatures. We showed how, starting from well known adjunctions between (complete) lattices and, respectively, polarities and reflexive graphs, one can derive relational semantics for the language of lattice logic on both these types of structures. We illustrated how these basic semantic frameworks can be systematically expanded to accommodate  expansions of the language with arbitrary normal connectives. The main point we wished to convey with this discussion was methodological: starting from an adjunction between a class of relational structures and lattices, one can uniformly obtain a suitable enrichment of these relational structures to support a systematically derived semantics for LE-logics. This methodology allows one to treat LE-logics uniformly, so that various proposals and solutions which were developed for specific signatures (e.g.~the Routley-Meyer semantics for substructural logics) can be systematically connected and extended to other signatures.

Having obtained these relational semantics, one can try and extract from them insights into the nature of non-distributive logics, much in the way that classical Kripke semantics provides insight into the nature of modal logics by enabling one to reason about possibility and necessity in terms of possible worlds. As we saw, the failure of distributivity is born out by both the polarity-based and graph-based semantics, which offer different perspectives on this phenomenon. 
The polarity-based semantics builds on insights from formal concept analysis, seeing polarities as abstract databases of objects and features, and gives rise to an interpretation of LE-logics where formulas act as names for formal concepts. Under this interpretation, the logical connectives $\wedge$ and $\vee$ come to stand, 
not for  counterparts of the natural language conjunction (`and') and disjunction (`or'), but rather for operations returning, respectively, the greatest common sub-concept and least common super-concept of the concepts named by their arguments. This interpretation allows to intuitively understand  the (satisfaction) relation $\Vdash$  as `being a member' in a category, or `exemplifying' a concept, and the (refutation) relation $\succ$ as `being a defining feature of' the given category or concept.

On the other hand, graph-based semantics evaluates formulas  at states as in Kripke frames, and  
gives rise to an {\em intensional} interpretation of $\wedge$ and $\vee$, where 
successor states need to be taken into account in the evaluation of these connectives. The graph-based semantics supports a reading of $\wedge$ and $\vee$ which can be described as a conservative or cautious version of logical conjunction and disjunction, in a type of {\em hyper-constructivist} logic in which the failure of the law of excluded middle happens at the {\em meta-language} level, in the sense that, at any given state, any given formula can be satisfied, refuted or {\em neither}. 

Informed by such general insights into the nature(s) of non-distributive logics, these semantic paradigms can be used to model and reason about more specific phenomena. Regarding polarities as abstract databases enriched with relations modelling other (e.g.\ agent specific) types of information, we saw how non-distributive logics in specific signatures emerge as modal and epistemic logics of formal concepts. Specializing and varying this interpretation allows us to reason about themes and topics in text databases, syndromes in hospitals, consumer segments within markets, competing scientific theories and interrogative agendas.

The graph-based semantics is suitable for capturing situations where relations of similarity, proximity, indistinguishability or other types of relative informational entropy between states play a role. Such relations encode certain bounds to knowability, inherent in the situation under consideration. These could be bounds on human perception, like the inability to distinguish between the colours of light frequencies which are less the two hertz apart, or bounds on the expressiveness of a natural language imposed by relations of synonymy. In this paper we developed an example where the bounds were evidential, and were encoded by the relation of consistency between a piece of testimony and its conclusions and another piece of testimony. Here the `cautious' non-distributive disjunction functioned naturally in formulas expressing the not guilty status of the defendant. 

In the present paper, we have only scratched the surface of a wide-ranging and presently ongoing research programme. Many developments are currently being investigated  and will be investigated which build on the insights, techniques and methodologies presented in this paper. In the following paragraphs we outline some of the main ones.

\paragraph{Many-valued semantics.} The interpretations of graph-based and polarity-based semantics we have discussed in the present paper generalize very naturally to many-valued versions. Objects in databases may posses features to a certain extent, rather than absolutely: instead of a word being a keyword of a document, we might want to know with what frequency it appears in the document and, as accordingly, a document may be about a topic to a certain extent; members of a selection committee naturally find some issues more important than others and so the emergent interrogative agendas emphasise some issues more than others. Similarly, testimonies may be more or less compatible with one another and may support conclusions or claims to varying extents; there are different proportions of overlap between the variables featuring in databases constructed to test scientific theories, and the may offer support hypothesis to a variable extent. In line with these ideas, many valued versions of polarity-based and graph-based semantics of LE-logics have been proposed (see e.g.\ \cite{roughconcepts, socio-political, eusflat, vague-cat}) and deserve further study and development.

\paragraph{Parametricity phenomena.} We have illustrated how the semantic frameworks we have introduced generalize and project onto the standard relational semantics of classical and intuitionistic (modal) logic. Moreover, we saw how the level in this hierarchy is parametric in---and completely determined by---the properties of the polarity relation $I$ or the graph relation $E$. We have also observed how this parametricity phenomenon extends to the interpretation of modal operators, and particularly in how the (suitably expressed) frame-properties corresponding to well-known modal axioms generalize by simply substituting the appropriate relation in containments and compositions. We conjecture that, in a way that will have to be made suitably precise, this holds for all Sahlqvist and inductive formulas, even when we consider many-valued versions of the semantics. This conjecture receives support from the fact that, in many-valued modal logic, all Sahlqvist formulas have verbatim the same first frame-correspondents, albeit with possibly different many-valued meaning \cite{Britz:MThesis:2016}. This suggests that other constructions and metatheorems, including truth preserving model-constructions and characterizations of expressivity (see \cite{Go17, goldblatt2019morphisms, conradie2018goldblatt}), might also lend themselves to a unified, parametric treatment.    

\paragraph{Vectors space semantics.} Throughout the paper, we have particularly stressed the methodological aspects of our approach, which allow for the systematic introduction of relational semantics for non-distributive logics using some forms of representation results for lattices (or subclasses thereof) obtained via dualities or adjunction results. The polarity-based and graph-based semantics discussed in the present paper are by no means the only structures which lend themselves to the application of this methodology, which can be extended also {\em beyond} non-distributive logics. Another instance of this methodology being applied is \cite{greco2019vector}, where a relational semantic framework based on vector spaces has been introduced for the basic (modal) Lambek calculus. Vector-based models are widely used in various areas of computer science such as  computational linguistics \cite{turney2010frequency}, information retrieval \cite{salton1975vector} and machine learning \cite{kaageback2014extractive}, as outcomes of statistical analyses of large databases. The possibility of regarding these structures as relational models of (non-distributive) propositional logics offers a concrete route for addressing the systematic integration of symbolic and sub-symbolic methods in AI.

\appendix
\section{Appendix}
\subsection{Relational interpretation of additional connectives}\label{appendix:rel:int:addnl:cnctvs}

In this section we provide more detail on how the defining clauses of the relations $\Vdash$ and $\succ$ on models based on polarities and reflexive graphs can be retrieved from the algebraic interpretation in complete  $\mathcal{L}_{\textnormal{LE}}$-algebras for  $\Box$, $\lhd$, $\rhd$ and, in general, for any connective $f\in \mathcal{F}$ of arity $n_f$ and any connective $g\in \mathcal{G}$ of arity $n_g$. 

We start with the case of polarities.  Appealing to Proposition \ref{prop:join and meet generators}, we can write $\overline{v}(\psi) = \bigwedge\{\mathbf{x}\in \mathbb{P}^+\mid \overline{v}(\psi)\leq \mathbf{x}\} = \bigwedge\{\mathbf{x}\in \mathbb{P}^+\mid x \succ \psi\}$. Since by assumption $\overline{v}$ is a homomorphism, $\overline{v}(\Box \psi) = \Box^{\mathbb{P}^+}\overline{v}(\psi) = \Box^{\mathbb{P}^+}(\bigwedge\{\mathbf{x}\in \mathbb{P}^+\mid x\succ \psi\})$, and since $\Box^{\mathbb{P}^+}$ is completely meet-preserving, we have:

\begin{equation}
\label{eq: boxpsi as a bigwedge}
\overline{v}(\Box \psi) = \bigwedge\{\Box^{\mathbb{P}^+}\mathbf{x}\mid x\succ \psi\}.
\end{equation}
%
%
Hence, for any $a\in A$, we have 
\begin{center}
	\begin{tabular}{r c l l}
		$a\Vdash \Box\psi$& iff &$\mathbf{a} \leq \overline{v}(\Box\psi)$ & definition of $\Vdash$ as in \eqref{eq: desiderata satisfaction-polarity}\\
		& iff & $\mathbf{a} \leq \bigwedge\{\Box^{\mathbb{P}^+}\mathbf{x}\mid x\succ \psi\}$ & \eqref{eq: boxpsi as a bigwedge}\\
		& iff & for all $x\in X$, if $x\succ \psi$ then $\mathbf{a} \leq \Box^{\mathbb{P}^+}\mathbf{x}$ & definition of $\bigwedge$\\
		& iff & for all $x\in X$, if $x\succ \psi$ then $a R_{\Box} x$, & \eqref{eq:def rbox} 
	\end{tabular}
\end{center}
where we have {\em defined} the accessibility relation $R_\Box\subseteq A\times X$ corresponding to the interpretation of $\Box$ as $\Box^{\mathbb{P}^+}$ as follows:
\begin{equation}\label{eq:def rbox} a R_\Box x \quad\mbox{ iff }\quad \mathbf{a} \leq \Box^{\mathbb{P}^+}\mathbf{x}. \end{equation}
Now, using that the set of elements $\mathbf{a}$ for $a\in A$ are join-generators of $\mathbb{P}^+$, we obtain $\succ$-clause for $\Box$:
{\footnotesize
\begin{center}
	\begin{tabular}{r c l l}
		$x\succ \Box\psi$ &iff & $\overline{v}(\Box\psi) \leq \mathbf{x}$ & definition of $\succ$ as in \eqref{eq: desiderata co-satisfaction-polarity}\\
		& iff & $\bigvee\{\mathbf{a}\mid \mathbf{a} \leq \overline{v}(\Box\psi)\} \leq \mathbf{x}$ & $\mathbf{a}$ for $a\in A$ join-generators of $\mathbb{P}^+$\\
		& iff & for all $a\in A$, if $\mathbf{a} \leq \overline{v}(\Box\psi)$ then $ \mathbf{a}\leq \mathbf{x}$& definition of $\bigvee$\\
		& iff & for all $a\in A$, if $a\Vdash \Box\psi$ then $ a Ix$ &definition of $\Vdash$ as in \eqref{eq: desiderata satisfaction-polarity}.\\
	\end{tabular}
\end{center}
}
Next, we consider $\lhd$. Since $\overline{v}$ is a homomorphism, $\overline{v}(\lhd \psi) = \lhd^{\mathbb{P}^+}\overline{v}(\psi) = \lhd^{\mathbb{P}^+}(\bigwedge\{\mathbf{x}\in \mathbb{P}^+\mid x\succ \psi\})$, and since $\lhd^{\mathbb{P}^+}$ is completely meet-reversing, we have:

\begin{equation}
\label{eq: lhdpsi as a bigvee}
\overline{v}(\lhd \psi) = \bigvee\{\lhd^{\mathbb{P}^+}\mathbf{x}\mid x\succ \psi\}.
\end{equation}
%
%
Hence, for any $x\in X$, we have 
\begin{center}
	\begin{tabular}{r c l l}
		$x\succ \lhd\psi$& iff &$ \overline{v}(\lhd\psi) \leq \mathbf{x}$ & definition of $\succ$ as in \eqref{eq: desiderata co-satisfaction-polarity}\\
		& iff & $\bigvee\{\lhd^{\mathbb{P}^+}\mathbf{x}'\mid x'\succ \psi\}\leq \mathbf{x}$ & \eqref{eq: lhdpsi as a bigvee}\\
		& iff & for all $x'\in X$, if $x'\succ \psi$ then $\lhd^{\mathbb{P}^+}\mathbf{x}'\leq \mathbf{x}$ & definition of $\bigvee$\\
		& iff & for all $x'\in X$, if $x'\succ \psi$ then $x R_{\lhd} x'$, & \eqref{eq:def rlhd} 
	\end{tabular}
\end{center}
where we have {\em defined} the accessibility relation $R_\lhd\subseteq X\times X$ corresponding to the interpretation of $\lhd$ as $\lhd^{\mathbb{P}^+}$ as follows:
\begin{equation}\label{eq:def rlhd} x R_\lhd x' \quad\mbox{ iff }\quad \lhd^{\mathbb{P}^+}\mathbf{x}'\leq \mathbf{x}. \end{equation}
As the set of elements $\mathbf{x}$ for $x\in X$ are meet-generators of $\mathbb{P}^+$, we obtain the following equivalences:
\begin{center}
	\begin{tabular}{r c l l}
		&&$a\Vdash \lhd\psi$\\
		&iff & $\mathbf{a} \leq \overline{v}(\lhd\psi)$ & definition of $\Vdash$ as in \eqref{eq: desiderata satisfaction-polarity}\\
		& iff & $\mathbf{a} \leq \bigwedge \{\mathbf{x}\mid \overline{v}(\lhd\psi) \leq \mathbf{x}\}$ & $\mathbf{x}$ for $x\in X$ meet-generators of $\mathbb{P}^+$\\
		& iff & for all $x\in X$, if $\overline{v}(\lhd\psi) \leq \mathbf{x}$ then $ \mathbf{a}\leq \mathbf{x}$& definition of $\bigwedge$\\
		& iff & for all $x\in X$, if $x\succ \lhd\psi$ then $ a Ix$ &definition of $\succ$ as in \eqref{eq: desiderata co-satisfaction-polarity}.\\
	\end{tabular}
\end{center}
Next, we consider $\rhd$. Since $\overline{v}$ is a homomorphism, $\overline{v}(\rhd \psi) = \rhd^{\mathbb{P}^+}\overline{v}(\psi) = \rhd^{\mathbb{P}^+}(\bigvee\{\mathbf{a}\in \mathbb{P}^+\mid a\Vdash \psi\})$, and since $\rhd^{\mathbb{P}^+}$ is completely join-reversing, we have:

\begin{equation}
\label{eq: rhdpsi as a bigwedge}
\overline{v}(\rhd \psi) = \bigwedge\{\rhd^{\mathbb{P}^+}\mathbf{a}\mid a\Vdash \psi\}.
\end{equation}
%
%
Hence, for any $a\in A$, we obtain the following:
\begin{center}
	\begin{tabular}{r c l l}
		$a\Vdash \rhd\psi$& iff &$ \mathbf{a}\leq \overline{v}(\rhd\psi)$ & definition of $\Vdash$ as in \eqref{eq: desiderata satisfaction-polarity}\\
		& iff & $\mathbf{a}\leq \bigwedge\{\rhd^{\mathbb{P}^+}\mathbf{a}'\mid a'\Vdash \psi\}$ & \eqref{eq: rhdpsi as a bigwedge}\\
		& iff & for all $a'\in A$, if $a'\Vdash \psi$ then $\mathbf{a} \leq \rhd^{\mathbb{P}^+}\mathbf{a}'$ & definition of $\bigwedge$\\
		& iff & for all $a'\in A$, if $a'\Vdash \psi$ then $a R_{\rhd} a'$, & \eqref{eq:def rrhd} 
	\end{tabular}
\end{center}
where we have {\em defined} the accessibility relation $R_\rhd\subseteq A\times A$ corresponding to the interpretation of $\rhd$ as $\rhd^{\mathbb{P}^+}$ as follows:
\begin{equation}\label{eq:def rrhd} a R_\rhd a' \quad\mbox{ iff }\quad \mathbf{a}\leq \rhd^{\mathbb{P}^+}\mathbf{a}'. \end{equation}
As the set of elements $\mathbf{a}$ for $a\in A$ are join-generators of $\mathbb{P}^+$, we obtain the following equivalences:
\begin{center}
	\begin{tabular}{r c l l}
		&&$x\succ \rhd\psi$\\
		&iff & $\overline{v}(\rhd\psi)\leq \mathbf{x}$ & definition of $\succ$ as in \eqref{eq: desiderata co-satisfaction-polarity}\\
		& iff & $\bigvee \{\mathbf{a}\mid \mathbf{a} \leq \overline{v}(\rhd\psi)\}\leq \mathbf{x}$ & $\mathbf{a}$ for $a\in A$ join-generators of $\mathbb{P}^+$\\
		& iff & for all $a\in A$, if $\mathbf{a} \leq \overline{v}(\rhd\psi)$ then $ \mathbf{a}\leq \mathbf{x}$& definition of $\bigvee$\\
		& iff & for all $a\in A$, if $a\Vdash \rhd\psi$ then $ a Ix$ &definition of $\Vdash$ as in \eqref{eq: desiderata satisfaction-polarity}.\\
	\end{tabular}
\end{center}
More generally, if $f$ is monotone in its $i$th coordinate then, by appealing to Proposition \ref{prop:join and meet generators}, we write $\overline{v}(\varphi_i) = \bigvee\{\mathbf{a}\in \mathbb{P}^+\mid \mathbf{a}\leq \overline{v}(\varphi_i)\} = \bigvee\{\mathbf{a}\in \mathbb{P}^+\mid a\Vdash \varphi_i\}$, and if $f$ is antitone in its $i$-th coordinate then we write  $\overline{v}(\varphi_i) = \bigwedge\{\mathbf{x}\in \mathbb{P}^+\mid \overline{v}(\varphi_i)\leq \mathbf{x}\} = \bigwedge\{\mathbf{x}\in \mathbb{P}^+\mid x\succ \varphi_i\}$. Since by assumption $\overline{v}$ is a homomorphism, $\overline{v}(f(\overline{\varphi})) = f^{\mathbb{P}^+}(\overline{\overline{v}(\varphi)})$, and so, since $f^{\mathbb{P}^+}$ preserves arbitrary joins in each monotone coordinate and reverses arbitrary meets in each antitone coordinate, we get 
\begin{equation}
\label{eq: fvarphi as a bigvee}
\overline{v}(f(\overline{\varphi})) = \bigvee\{f^{\mathbb{P}^+}\left(\overline{\mathbf{a}}^{(n_f)}\right)\mid \overline{a}^{(n_f)}\Vdash^{(n_f)} \overline{\varphi}\},\end{equation}
where $\overline{\mathbf{a}}^{(n_f)}$ is $\mathbf{a}$  and  $\Vdash^{(n_f)} $ is $\Vdash$ if $f$ is monotone in its $i$-th coordinate, whereas  $\overline{\mathbf{a}}^{(n_f)}$ is $\mathbf{x}$  and $\Vdash^{(n_f)} $ is $\succ$ if $f$ is antitone in its $i$-th coordinate. 
Hence, for any $x \in X$, we obtain the following:
{\footnotesize
\begin{center}
	\begin{tabular}{r c l l}
		$x\succ f(\overline{\varphi})$& iff &$\overline{v}(f(\overline{\varphi}))\leq \mathbf{x}$ & definition of $\succ$ as in \eqref{eq: desiderata co-satisfaction-polarity}\\
		& iff & $\bigvee\{f^{\mathbb{P}^+}\left(\overline{\mathbf{a}}^{(n_f)}\right)\mid \overline{a}^{(n_f)}\Vdash^{(n_f)} \varphi\}\leq \mathbf{x}$ & \eqref{eq: fvarphi as a bigvee}\\
		& iff & for all $\overline{a}^{(n_f)}\in \overline{A}^{(n_f)}$, if $\overline{a}^{(n_f)}\Vdash \overline{\varphi}$ then $f^{\mathbb{P}^+}(\overline{\mathbf{a}}^{(n_f)})\leq \mathbf{x}$ & definition of $\bigvee$\\
		& iff & for all $\overline{a}^{(n_f)}\in \overline{A}^{(n_f)}$, if $\overline{a}^{(n_f)}\Vdash^{(n_f)} \overline{\varphi}$ then $R_{f}(x, \overline{a}^{(n_f)})$ & \eqref{eq:def rf}. 
	\end{tabular}
\end{center}
}
Here we have {\em defined} the accessibility relation $R_f\subseteq X\times \overline{A}^{(n_f)}$, where $\overline{A}^{(n_f)}$ denotes the $n_f$-fold cartesian product of $A$ and $X$ such that for each $1\leq i\leq n_f$ if $f$ is monotone in its $i$-th coordinate then the $i$-th coordinate of $\overline{A}^{(n_f)}$ is $A$, whereas if $f$ is antitone in its $i$-th coordinate then the $i$-th coordinate of $\overline{A}^{(n_f)}$ is $X$, corresponding to the interpretation of $f$ as $f^{\mathbb{P}^+}$ as follows:
\begin{equation}\label{eq:def rf} R_f (x, \overline{a}^{(n_f)}) \; \mbox{ iff }\;  f^{\mathbb{P}^+}(\overline{\mathbf{a}}^{(n_f)})\leq \mathbf{x}. \end{equation}
Now, using the fact that the set of elements $\mathbf{x}$ for $x\in X$ are meet-generators of $\mathbb{P}^+$, we can write:
\begin{center}
	\begin{tabular}{r c l l}
		&&$a\Vdash f(\overline{\varphi})$\\
		&iff & $\mathbf{a}\leq \overline{v}(f(\overline{\varphi}))$ & definition of $\Vdash$ as in \eqref{eq: desiderata satisfaction-polarity}\\
		& iff & $\mathbf{a}\leq \bigwedge\{\mathbf{x}\mid  \overline{v}(f(\overline{\varphi}))\leq \mathbf{x}\}$ & $\mathbf{x}$ for $x\in X$ meet-generators of $\mathbb{P}^+$\\
		& iff & for all $x\in X$, if $\overline{v}(f(\overline{\varphi}))\leq \mathbf{x}$ then $ \mathbf{a}\leq \mathbf{x}$& definition of $\bigwedge$\\
		& iff & for all $x\in X$, if $x\succ f(\overline{\varphi})$ then $ a Ix$ &definition of $\succ$ as in \eqref{eq: desiderata co-satisfaction-polarity}.\\
	\end{tabular}
\end{center}

If $g$ is monotone in its $i$-th coordinate then, by Proposition \ref{prop:join and meet generators}, we can write $\overline{v}(\varphi_i) = \bigwedge\{\mathbf{x}\in \mathbb{P}^+\mid \overline{v}(\varphi_i) \leq \mathbf{x}\} = \bigwedge\{\mathbf{x}\in \mathbb{P}^+\mid x\succ \varphi_i\}$, and if $f$ is antitone in its $i$-th coordinate then we write  $\overline{v}(\varphi_i) = \bigvee\{\mathbf{a}\in \mathbb{P}^+\mid \mathbf{a}\leq \overline{v}(\varphi_i)\} = \bigvee\{\mathbf{a}\in \mathbb{P}^+\mid a\Vdash \varphi_i\}$. Now, since $\overline{v}$ is a homomorphism, $\overline{v}(g(\overline{\varphi})) = g^{\mathbb{P}^+}(\overline{\overline{v}(\varphi)})$, and  since $g^{\mathbb{P}^+}$ preserves arbitrary meets in each monotone coordinate and reverses arbitrary joins in each antitone coordinate, we get: 
\begin{equation}
\label{eq: gvarphi as a bigwedge}
\overline{v}(g(\overline{\varphi})) = \bigwedge\{g^{\mathbb{P}^+}\left(\overline{\mathbf{x}}^{(n_g)}\right)\mid \overline{x}^{(n_g)}\succ^{(n_g)} \overline{\varphi}\},\end{equation}
where $\overline{\mathbf{x}}^{(n_g)}$ is $\mathbf{x}$  and  $\succ^{(n_g)} $ is $\succ$ if $g$ is monotone in its $i$-th coordinate, whereas  $\overline{\mathbf{x}}^{(n_g)}$ is $\mathbf{a}$  and $\succ^{(n_g)} $ is $\Vdash$ if $g$ is antitone in its $i$-th coordinate. 
Thus, for any $a \in A$, we obtain the following equivalences:
{\footnotesize
\begin{center}
	\begin{tabular}{r c l l}
		$a\Vdash g(\overline{\varphi})$& iff &$\mathbf{a} \leq \overline{v}(f(\overline{\varphi}))$ & definition of $\Vdash$ as in \eqref{eq: desiderata satisfaction-polarity}\\
		& iff & $\mathbf{a} \leq \bigwedge\{g^{\mathbb{P}^+}\left(\overline{\mathbf{x}}^{(n_g)}\right)\mid \overline{x}^{(n_g)}\succ^{(n_g)} \varphi\}$ & \eqref{eq: gvarphi as a bigwedge}\\
		& iff & for all $\overline{x}^{(n_g)}\in \overline{X}^{(n_g)}$, if $\overline{x}^{(n_g)}\succ \overline{\varphi}$ then $\mathbf{a} \leq g^{\mathbb{P}^+}(\overline{\mathbf{x}}^{(n_g)})$ & definition of $\bigwedge$\\
		& iff & for all $\overline{x}^{(n_g)}\in \overline{X}^{(n_g)}$, if $\overline{x}^{(n_g)}\succ^{(n_g)} \overline{\varphi}$ then $R_{g}(a, \overline{x}^{(n_g)})$ & \eqref{eq:def rg}. 
	\end{tabular}
\end{center}
}
Here we have {\em defined} the accessibility relation $R_g\subseteq A\times \overline{X}^{(n_g)}$, where $\overline{X}^{(n_g)}$ denotes the $n_g$-fold cartesian product of $A$ and $X$ such that for each $1\leq i\leq n_g$ if $g$ is monotone in its $i$-th coordinate then the $i$-th coordinate of $\overline{X}^{(n_g)}$ is $X$, whereas if $g$ is antitone in its $i$-th coordinate then the $i$-th coordinate of $\overline{X}^{(n_g)}$ is $A$, corresponding to the interpretation of $g$ as $g^{\mathbb{P}^+}$ as follows:
\begin{equation}\label{eq:def rg} R_g (a, \overline{x}^{(n_g)}) \; \mbox{ iff }\;  \mathbf{a} \leq g^{\mathbb{P}^+}(\overline{\mathbf{x}}^{(n_g)}). \end{equation}
Since the set of elements $\mathbf{a}$ for $a\in A$ are join-generators of $\mathbb{P}^+$, we can write:
\begin{center}
	\begin{tabular}{r c l l}
		&&$x\succ g(\overline{\varphi})$\\
		&iff & $\overline{v}(g(\overline{\varphi})) \leq \mathbf{x}$ & definition of $\succ$ as in \eqref{eq: desiderata co-satisfaction-polarity}\\
		& iff & $\bigvee\{\mathbf{a}\mid  \mathbf{a} \leq \overline{v}(g(\overline{\varphi}))\}\leq \mathbf{x}$ & $\mathbf{a}$ for $a\in A$ join-generators of $\mathbb{P}^+$\\
		& iff & for all $a\in A$, if $\mathbf{a} \leq \overline{v}(g(\overline{\varphi}))$ then $ \mathbf{a}\leq \mathbf{x}$& definition of $\bigvee$\\
		& iff & for all $a\in A$, if $a\Vdash g(\overline{\varphi})$ then $ a Ix$ &definition of $\Vdash$ as in \eqref{eq: desiderata satisfaction-polarity}.\\
	\end{tabular}
\end{center}


Finally, we show how the defining clauses of the relations $\Vdash$ and $\succ$ on models based on graphs can be retrieved from the algebraic interpretation in complete  $\mathcal{L}_{\textnormal{LE}}$-algebras for $\Diamond$ and $\Box$. The clauses for $\lhd$, $\rhd$, $f$ and $g$  follow
by similar arguments. 

We first treat $\Diamond$. Analogously to the way we argued above, by appealing to Proposition \ref{prop:join and meet generators}, we can write $\overline{v}(\psi) = \bigvee\{\mathbf{z}_s\in \mathbb{X}^+\mid \mathbf{z}_s\leq \overline{v}(\psi)\} = \bigvee\{\mathbf{z}_s\in \mathbb{X}^+\mid z\Vdash \psi\}$. Since by assumption $\overline{v}$ is a homomorphism, $\overline{v}(\Diamond \psi) = \Diamond^{\mathbb{X}^+}\overline{v}(\psi) = \Diamond^{\mathbb{X}^+}(\bigvee\{\mathbf{z}_s\in \mathbb{X}^+\mid z\Vdash \psi\})$, and since $\Diamond^{\mathbb{X}^+}$ is completely join-preserving, we obtain:

\begin{equation}
\label{eq: graph: diamondpsi as a bigvee }
\overline{v}(\Diamond \psi) = \bigvee\{\Diamond^{\mathbb{X}^+}\mathbf{z}_s\mid z\Vdash \psi\}.\end{equation}

\begin{center}
	\begin{tabular}{r c l l}
		$x\succ \Diamond\psi$& iff &$\overline{v}(\Diamond\psi)\leq \mathbf{z}_r$ & definition of $\succ$ as in \eqref{eq: desiderata co-satisfaction-graph}\\
		& iff & $\bigvee\{\Diamond^{\mathbb{X}^+}\mathbf{z}_s\mid z\Vdash \psi\}\leq \mathbf{z}_r$ & \eqref{eq: graph: diamondpsi as a bigvee }\\
		& iff & for all $z'\in Z$, if $z'\Vdash \psi$ then $\Diamond^{\mathbb{X}^+}\mathbf{z}_s\leq \mathbf{z}_r$ & definition of $\bigvee$\\
		& iff & for all $z'\in Z$, if $\Diamond^{\mathbb{X}^+}\mathbf{z}_s\nleq \mathbf{z}_r$ then $z'\nVdash \psi$ & contraposition\\
		& iff & for all $z'\in Z$, if  $z R_{\Diamond} z'$ then $z'\nVdash \psi$, & \eqref{eq:def rcdiamond} 
	\end{tabular}
\end{center}
where we have {\em defined} the accessibility relation $R_\Diamond$ corresponding to the interpretation of $\Diamond$ as $\Diamond^{\mathbb{X}^+}$ as follows:
\begin{equation}\label{eq:def rcdiamond} z R_\Diamond z' \quad\mbox{ iff }\quad \Diamond^{\mathbb{X}^+}\mathbf{z}_z\nleq \mathbf{z}_r. \end{equation}
{\footnotesize
\begin{center}
	\begin{tabular}{r c l l}
		&&$z\Vdash \Diamond\psi$\\
		&iff & $\mathbf{z}_s\leq \overline{v}(\Diamond\psi)$ & definition of $\Vdash$ as in \eqref{eq: desiderata satisfaction-graph}\\
		& iff & $\mathbf{z}_s\leq \bigwedge\{\mathbf{z}_r\mid  \overline{v}(\Diamond\psi)\leq \mathbf{z}_r\}$ & $\mathbf{z}_r$ for $z\in Z$ meet-generators of $\mathbb{X}^+$\\
		& iff & for all $z'\in Z$, if $\overline{v}(\Diamond\psi)\leq \mathbf{z}_r$ then $ \mathbf{z}_s\leq \mathbf{z}_r$& definition of $\bigwedge$\\
		& iff & for all $z'\in Z$, if $z'\succ \Diamond\psi$ then $ z I_{E^c}z'$ &definition of $\succ$ as in \eqref{eq: desiderata co-satisfaction-graph}.\\
		& iff & for all $z'\in Z$, if $z'\succ \Diamond\psi$ then $(z, z')\notin E$ &definition of $I_{E^c}$ \\
		& iff & for all $z'\in Z$, if $zEz'$ then $z'\nsucc \Diamond\psi$ & contraposition.  \\
	\end{tabular}
\end{center}
}

Since by assumption $\overline{v}$ is a homomorphism, $\overline{v}(\Box \psi) = \Box^{\mathbb{X}^+}\overline{v}(\psi) = \Box^{\mathbb{X}^+}(\bigwedge\{\mathbf{z}_r\in \mathbb{X}^+\mid z\succ \psi\})$, and since $\Box^{\mathbb{X}^+}$ is completely meet-preserving, we have:

\begin{equation}
\label{eq: graph: boxpsi as a bigwedge}
\overline{v}(\Box \psi) = \bigwedge\{\Box^{\mathbb{X}^+}\mathbf{z}_r\mid z\succ \psi\}.
\end{equation}
Hence, for any $z\in Z$, we have 
\begin{center}
	\begin{tabular}{r c l l}
		$z\Vdash \Box\psi$& iff &$\mathbf{z}_s \leq \overline{v}(\Box\psi)$ & definition of $\Vdash$ as in \eqref{eq: desiderata satisfaction-graph}\\
		& iff & $\mathbf{z}_s \leq \bigwedge\{\Box^{\mathbb{X}^+}\mathbf{z}_r\mid z\succ \psi\}$ & \eqref{eq: graph: boxpsi as a bigwedge}\\
		& iff & for all $z'\in Z$, if $z'\succ \psi$ then $\mathbf{z_s} \leq \Box^{\mathbb{X}^+}\mathbf{z}_r$ & definition of $\bigwedge$\\
		& iff & for all $z'\in Z$, if $\mathbf{z_s} \nleq \Box^{\mathbb{X}^+}\mathbf{z}_r$ then $z'\nsucc \psi$ & contraposition\\
		& iff & for all $z'\in Z$, if $zR_\Box z'$ then $z'\nsucc \psi$ & \eqref{eq:def rcbox}\\
	\end{tabular}
\end{center}
where we have {\em defined} the accessibility relation $R_\Box$ corresponding to the interpretation of $\Box$ as $\Box^{\mathbb{X}^+}$ as follows:
\begin{equation}\label{eq:def rcbox} z R_\Box z' \quad\mbox{ iff }\quad \mathbf{z}_s \nleq \Box^{\mathbb{X}^+}\mathbf{z}_r. \end{equation}
\begin{center}
	\begin{tabular}{r c l l}
		&&$z\succ \Box\psi$\\
		&iff & $\overline{v}(\Box\psi) \leq \mathbf{z}_r$ & definition of $\succ$ as in \eqref{eq: desiderata co-satisfaction-graph}\\
		& iff & $\bigvee\{\mathbf{z}_s\mid \mathbf{z}_s \leq \overline{v}(\Box\psi)\} \leq \mathbf{z}_r$ & $\mathbf{z}_s$ for $z\in Z$ join-generators of $\mathbb{X}^+$\\
		& iff & for all $z'\in Z$, if $\mathbf{z}_s \leq \overline{v}(\Box\psi)$ then $ \mathbf{z}_s\leq \mathbf{z}_r$& definition of $\bigvee$\\
		& iff & for all $z'\in Z$, if $z'\Vdash \Box\psi$ then $ z'I_{E^c}z$ &definition of $\Vdash$ as in \eqref{eq: desiderata satisfaction-graph}.\\
		& iff & for all $z'\in Z$, if $z'\Vdash \Box\psi$ then $ (z', z)\notin E$ &definition of $I_{E^c}$ \\
		& iff & for all $z'\in Z$, if $z'Ez$ then $z'\nVdash \Box\psi$ & contraposition.  \\
		
	\end{tabular}
\end{center}



\subsection{Polarity-based frames}\label{apendix:pol:based}
In this subsection we give 
the definition of a polarity-based frame, as well as 
a suitable expanded adjunction based on the adjunction between complete lattices and polarities described in Section \ref{ssec:polarities graphs}. We will focus mainly on the treatment of the additional connectives. 

Before we give the definition of a polarity-based frame, we introduce some notation. For any sets $A, B$ and any relation $S \subseteq A \times B$, we let, for any $A' \subseteq A$ and $B' \subseteq B$,
$$S^{(1)}[A'] := \{b \in B\mid  \forall a(a \in A' \Rightarrow a Sb ) \} \,\, \mathrm{and}\,\, S^{(0)}[B'] := \{a \in A \mid \forall b(b \in B' \Rightarrow a S b)  \}.$$
For all sets $A, B_1,\ldots B_n,$ and any relation $S \subseteq A \times B_1\times \cdots\times B_n$, for any $\overline{C}: = (C_1,\ldots, C_n)$ where $C_i\subseteq B_i$ and $1\leq i\leq n$ 
we let, for all $A'$,
\begin{equation*}\label{eq:notation bari}
\overline{C}^{\,i}:  = (C_1,\ldots,C_{i-1}, C_{i+1},\ldots, C_n)
\end{equation*}
\begin{equation*}\label{eq:notation baridown}
\overline{C}^{\,i}_{A'}: = (C_1\ldots,C_{i-1}, A', C_{i+1},\ldots, C_n)
\end{equation*}
When $C_i: = \{c_i\}$ and $A': =\{a'\} $, we will write $\overline{c}$ for $\overline{\{c\}}$,  and $\overline{c}^{\,i}$ for $\overline{\{c\}}^{\,i}$, and $\overline{c}^{\,i}_{a'}$ for $\overline{\{c\}}^{\,i}_{\{a'\}}$.
We also let:
\begin{enumerate}
	\item $S^{(0)}[\overline{C}] := \{a \in A\mid  \forall \overline{b}(\overline{b}\in \overline{C} \Rightarrow aS \overline{b} ) \}$.
	
	\item $S_i \subseteq B_i \times B_1 \times \cdots \times B_{i-1} \times A \times B_{i + 1} \times \cdots\times B_n$ be defined by
	\[(b_i, \overline{c}_{a}^{\, i})\in S_i \mbox{ iff } (a,\overline{c})\in S.\]
	
	\item $S^{(i)}[A', \overline{C}^{\,i}] := S_i^{(0)}[\overline{C}^{\, i}_{A'}]$.
\end{enumerate}

\begin{definition}\label{def:polarity:based:frm}
	A {\em polarity-based frame} for $\mathcal{L}_{\text{LE}}$ is a tuple $\mathbb{F} = (\mathbb{P}, \mathcal{R}_{\mathcal{F}}, \mathcal{R}_{\mathcal{G}})$, where  $\mathbb{P} = (A, X,  I)$ is a polarity, $\mathcal{R}_{\mathcal{F}} = \{R_f\mid f\in \mathcal{F}\}$, and $\mathcal{R}_{\mathcal{G}} = \{R_g\mid g\in \mathcal{G}\}$, such that  for each $f\in \mathcal{F}$ and $g\in \mathcal{G}$, the symbols $R_f$ and  $R_g$ respectively denote $(n_f+1)$-ary and $(n_g+1)$-ary relations on $\mathbb{P}$,
	\begin{equation*}
	R_f \subseteq X \times \overline{A}^{(n_f)}    \ \mbox{ and }\ R_g \subseteq A \times \overline{X}^{(n_g)},
	\end{equation*}
	where $\overline{A}^{(n_f)}$ denotes the $n_f$-fold cartesian product of $A$ and $X$ such that for each $1\leq i\leq n_f$ if $f$ is monotone in its $i$-th coordinate then the $i$th coordinate of $\overline{A}^{(n_f)}$ is $A$, whereas if $f$ is antitone in its $i$-th coordinate then the $i$th coordinate of $\overline{A}^{(n_f)}$ is $X$, and $\overline{X}^{(n_g)}$ denotes the $n_g$-fold cartesian product of $A$ and $X$ such that for each $1\leq i\leq n_g$ if $g$ is monotone in its $i$-th coordinate then the $i$-th coordinate of $\overline{X}^{(n_g)}$ is $X$, whereas if $g$ is antitone in its $i$-th coordinate then the $i$-th coordinate of $\overline{X}^{(n_g)}$ is $A$.
	In addition, all relations $R_f$ and $R_g$ are required to be {\em compatible}, i.e.\ the following sets are assumed to be  Galois-stable for all $a \in A$, $x \in X $, $\overline{a} \in A^{(n_f)}$, and $\overline{x} \in X^{(n_g)}$:
	\[
	R_f^{(0)}[\overline{a}]\text{ and }R_f^{(i)}[x, \overline{a}^{\, i}] \quad\quad R_g^{(0)}[\overline{x}]\text{ and }R_g^{(i)}[a, \overline{x}^{\,i}].
	\]
\end{definition}

\begin{definition}
	For a complete lattice expansion $\mathbb{A} = (\mathbb{L}, \mathcal{F}^{\mathbb{L}}, \mathcal{G}^{\mathbb{L}})$, the  $\mathcal{L}_{\text{LE}}$-frame associated with $\mathbb{A}$ is the structure
	$\mathbb{F}_\mathbb{A} = 
	(\mathbb{P}_\mathbb{L}, {R}_\mathcal{F}, \mathcal{R}_\mathcal{G})$, where $\mathbb{P}_\mathbb{L}$ is the polarity $(L, L, \leq)$, $\mathcal{R}_{\mathcal{F}} = \{R_f\mid f\in \mathcal{F}\}$, and $\mathcal{R}_{\mathcal{G}} = \{R_g\mid g\in \mathcal{G}\}$, such that  for each $f\in \mathcal{F}$ and $g\in \mathcal{G}$, the relations $R_f$ and  $R_g$ are defined as follows:
	\begin{equation*} R_f (x, \overline{a}^{(n_f)}) \; \mbox{ iff }\;  f^{\mathbb{P}_{\mathbb{L}}^+}(\overline{\mathbf{a}}^{(n_f)})\leq \mathbf{x} \quad\quad R_g (a, \overline{x}^{(n_g)}) \; \mbox{ iff }\;   \mathbf{a}\leq g^{\mathbb{P}^+_\mathbb{L}}(\overline{\mathbf{x}}^{(n_g)}). \end{equation*}
\end{definition}

\begin{prop}
	For a complete lattice expansion $\mathbb{A} = (\mathbb{L}, \mathcal{F}^{\mathbb{L}}, \mathcal{G}^{\mathbb{L}})$, the associated $\mathcal{L}_{\text{LE}}$-frame $\mathbb{F}_\mathbb{A}$ is a polarity-based frame. 
\end{prop}
%

	%
	


The {\em complex algebra} of  a polarity-based frame $\mathbb{F} = (\mathbb{P}, \mathcal{R}_\mathcal{F}, \mathcal{R}_\mathcal{G})$ for  $\mathcal{L}_{\text{LE}}$ is the algebra
$$\mathbb{F}^+ = (\mathbb{L}, \{f_{R_f}\mid f\in \mathcal{F}\}, \{g_{R_g}\mid g\in \mathcal{G}\}),$$
where $\mathbb{L} := \mathbb{P}^+$, and for all $f \in \mathcal{F}$ and all $g \in \mathcal{G}$,
we let

{\footnotesize
\begin{enumerate}
	
	\item $ f_{R_f}: \mathbb{L}^{n_f}\to \mathbb{L}$ be defined by the assignment $f_{R_f}(\overline{c}) = \left(\left(R_f^{(0)}[\overline{\val{c}}^{(n_f)}]\right)^{\downarrow}, R_f^{(0)}[\overline{\val{c}}^{(n_f)}]\right)$;
	
	\item $g_{R_g}: \mathbb{L}^{n_g}\to \mathbb{L}$ be defined by the assignment  $g_{R_g}(\overline{c}) = \left(R_g^{(0)}[\overline{\descr{c}}^{(n_g)}], \left(R_g^{(0)}[\overline{\descr{c}}^{(n_g)}]\right)^{\uparrow}\right)$. 
	
\end{enumerate}
}
\noindent
Here $\overline{\val{c}}^{(n_f)}$ denotes the $n_f$-tuple of $\val{c}$ and $\descr{c}$ such that for
each $1 \leq i \leq n_f$ if $f$ is monotone in its $i$-th coordinate then the $i$-th coordinate of $\overline{\val{c}}^{(n_f)}$ is $\val{c}$, whereas if $f$ is antitone in its $i$-th coordinate then the $i$th coordinate of $\overline{\val{c}}^{(n_f)}$ is $\descr{c}$, and $\overline{\descr{c}}^{(n_g)}$ denotes the $n_g$-tuple of $\val{c}$ and $\descr{c}$ such that for
each $1 \leq i \leq n_g$ if $g$ is monotone in its $i$-th coordinate then the $i$-th coordinate of $\overline{\val{c}}^{(n_f)}$ is $\descr{c}$, whereas if $g$ is antitone in its $i$-th coordinate then the $i$-th coordinate of $\overline{\descr{c}}^{(n_g)}$ is $\val{c}$. 

\begin{prop}[cf.\ \cite{greco2018algebraic} Proposition 21]\label{prop:F plus is complete LE}
	If $\mathbb{F} = (\mathbb{P}, \mathcal{R}_\mathcal{F}, \mathcal{R}_\mathcal{G})$ is a polarity-based frame for  $\mathcal{L}_{\text{LE}}$, then $\mathbb{F}^+ = (\mathbb{L}, \{f_{R_f}\mid f\in \mathcal{F}\}, \{g_{R_g}\mid g\in \mathcal{G}\})$ is a complete lattice expansion. 
\end{prop}


\subsection{Graph-based frames}\label{apendix:graph:based}

Here we define a graph-based frame and give a suitable expanded adjunction based on the adjunction between lattices and graphs described in Section \ref{ssec:polarities graphs}.

For any sets $A, B$ and any relation $S \subseteq A \times B$, we let, for any $A' \subseteq A$ and $B' \subseteq B$,
\begin{equation*}\label{eq:def:square brackets}A^{[1]}[A']:=\{b\mid \forall a(a\in A'\Rightarrow aS^cb) \}  \quad\quad S^{[0]}[B']:=\{a\mid \forall b(b\in B'\Rightarrow aS^cb) \}\footnote{We will sometimes abbreviate $E^{[0]}[X]$ and $E^{[1]}[Y]$ as $X^{[0]}$ and $Y^{[1]}$, respectively, for each $X, Y\subseteq Z$. If $X= \{x\}$ and  $Y = \{y\}$, we write $x^{[0]}$ and $y^{[1]}$ for $\{x\}^{[0]}$ and $\{y\}^{[1]}$, and write $X^{[01]}$ and $Y^{[10]}$ for $(X^{[0]})^{[1]}$ and $(Y^{[1]})^{[0]}$, respectively. Notice that $X^{[0]} = X^{\downarrow}$ and $Y^{[1]} = Y^{\uparrow}$, where the maps $(\cdot)^\downarrow$ and $(\cdot)^\uparrow$ are those associated with the polarity $\mathbb{P_X}$.}.\end{equation*}
Hence, $S^{[1]}[A'] = (S^c)^{(1)}[A']$ and $S^{[0]}[B'] = (S^c)^{(0)}[B']$. More generally, for all sets $A, B_1,\ldots B_n,$ and any relation $S \subseteq A \times B_1\times \cdots\times B_n$, for any $\overline{C}: = (C_1,\ldots, C_n)$ where $C_i\subseteq B_i$ and $1\leq i\leq n$ 
we let, for all $A'$,
\begin{eqnarray*}
	S^{[0]}[\overline{C}] := \{a \in A\mid  \forall \overline{b}(\overline{b}\in \overline{C} \Rightarrow aS^c \overline{b} ) \} \quad\quad S^{[i]}[A', \overline{C}^{\,i}] := S_i^{[0]}[\overline{C}^{\, i}_{A'}].
\end{eqnarray*}


\begin{definition}\label{def:graph:basedL:frm}
	A {\em graph-based frame} for $\mathcal{L}_{\text{LE}}$ is a tuple $\mathbb{F} = (\mathbb{X}, \mathcal{R}_{\mathcal{F}}, \mathcal{R}_{\mathcal{G}})$, where  $\mathbb{X} = (Z, E)$ is a reflexive graph, $\mathcal{R}_{\mathcal{F}} = \{R_f\mid f\in \mathcal{F}\}$, and $\mathcal{R}_{\mathcal{G}} = \{R_g\mid g\in \mathcal{G}\}$, such that  for each $f\in \mathcal{F}$ and $g\in \mathcal{G}$, the symbols $R_f$ and  $R_g$ respectively denote $(n_f+1)$-ary and $(n_g+1)$-ary relations on $\mathbb{P}_{\mathbb{X}}$
	%
	satisfying the following {\em compatibility} conditions:
	for all $z'  \in Z$, $\overline{z}'' \in Z^{n_f}$ and $\overline{z}''' \in Z^{n_g}$,
	\begin{align*}
	(R_g^{[0]}[\overline{z}'''])^{[10]} \subseteq R_g^{[0]}[\overline{z}'''] \qquad   & \qquad	(R_f^{[0]}[\overline{z}''])^{[01]} \subseteq R_f^{[0]}[\overline{z}''],\\
	\end{align*}
	and for every $1\leq i\leq n_g$ and every $1\leq j\leq n_f$ such that $g$ and $f$ are positive in their $i$th (resp.~$j$th) coordinate,
\begin{align*}	
	(R_g^{[i]}[z', \overline{z}^{\, i}])^{[01]} \subseteq R_g^{[i]}[z', \overline{z}^{\, i}] \qquad   &\qquad	(R_f^{[j]}[z', \overline{z}^{\, j}])^{[10]} \subseteq R_f^{[i]}[z', \overline{z}^{\, i}], \\
\end{align*}
	and for every $1\leq i\leq n_g$ and every $1\leq j\leq n_f$ such that $g$ and $f$ are negative in their $i$th (resp.~$j$th) coordinate,
\begin{align*}	
	(R_g^{[i]}[z', \overline{z}^{\, i}])^{[10]} \subseteq R_g^{[i]}[z', \overline{z}^{\, i}] \qquad   & \qquad	(R_f^{[j]}[z', \overline{z}^{\, j}])^{[01]} \subseteq R_f^{[i]}[z', \overline{z}^{\, i}]. \\
\end{align*}
	\end{definition}

\begin{remark}
Let $\mathbb{F}$ be a graph based frame as above. By Proposition \ref{prop:transitive forces Heyting}, when $E: = {\leq}$ is reflexive and transitive, the closure operator $(-)^{[10]}$ coincides with the upward closure  $(-){\uparrow}$ of $E$ (and  order-dually, the closure operator $(-)^{[01]}$ coincides with the downward closure  $(-){\downarrow}$ of $E$). In this case, the compatibility conditions on the additional relations of $\mathbb{F}$ are equivalent to the well-known interaction conditions between accessibility relations and order on Kripke frames for distributive modal logic.  As an illustration, let us show that   the compatibility conditions on $R_{\Box}$ project to  $({\leq}\circ R_{\Box}\circ {\leq})\subseteq R_{\Box}$. 

By \cite[Lemma 4]{graph-based-wollic}, the conditions $(R_{\Box}^{[1]}[b])^{[01]}\subseteq R_{\Box}^{[1]}[b]$ and $(R_{\Box}^{[0]}[y])^{[10]}\subseteq R_{\Box}^{[0]}[y]$ for all $b, y\in Z$ are respectively equivalent  to $R_{\Box}^{[0]}[Y^{[01]}] = R_{\Box}^{[0]}[Y]$ and  $(R_{\Box}^{[0]}[Y])^{[10]}\subseteq R_{\Box}^{[0]}[Y]$   for every $Y\subseteq Z$, which, as discussed above, can be equivalently rewritten as  $R_{\Box}^{[0]}[Y{\downarrow}] = R_{\Box}^{[0]}[Y]$ and $(R_{\Box}^{[0]}[Y]){\uparrow}\subseteq R_{\Box}^{[0]}[Y]$ for every $Y\subseteq Z$.

The former condition is clearly equivalent to $(R_{\Box}^{[0]}[Y{\downarrow}])^c = (R_{\Box}^{[0]}[Y])^c$, while the latter  requires $R_{\Box}^{[0]}[Y]$ to be an up-set, and is hence  equivalent to the one requiring $(R_{\Box}^{[0]}[Y])^c$ to be a down-set for every $Y\subseteq Z$, i.e.~$(R_{\Box}^{[0]}[Y])^{c}{\downarrow}\subseteq (R_{\Box}^{[0]}[Y])^c$.  Note that, for every $Y'\subseteq Z$, 
\begin{eqnarray*}
(R_{\Box}^{[0]}[Y'])^{c}: & = & (\{b\in Z\mid \forall y(y\in Y'\Rightarrow (b, y)\notin R_{\Box})\})^c\\
&  = & (\{b\in Z\mid \forall y(b R_{\Box}y \Rightarrow y\notin Y')\})^c\\
& = &\{b\in Z\mid \exists y(b R_{\Box}y \ \&\ y\in Y')\} \\
& = & R^{-1}_{\Box}[Y'].
\end{eqnarray*}
Hence, $(R_{\Box}^{[0]}[Y{\downarrow}])^c = (R_{\Box}^{[0]}[Y])^c$ and $(R_{\Box}^{[0]}[Y])^{c}{\downarrow}\subseteq (R_{\Box}^{[0]}[Y])^c$ can be equivalently rewritten as $R_{\Box}^{-1}[Y{\downarrow}] = R_{\Box}^{-1}[Y]$ and  $(R_{\Box}^{-1}[Y]){\downarrow}\subseteq R_{\Box}^{-1}[Y]$, respectively. These conditions are equivalent to $(R_{\Box}\circ {\leq})^{-1}[Y] = ({\geq}\circ R_{\Box}^{-1})[Y] =  R_{\Box}^{-1}[Y]$ and $({\leq}\circ R_{\Box})^{-1}[Y] = (R_{\Box}^{-1}\circ {\geq})[Y]\subseteq R_{\Box}^{-1}[Y]$, which are equivalent to $(R_{\Box}\circ {\leq})[Y] =  R_{\Box}[Y]$ and $({\leq}\circ R_{\Box})[Y] \subseteq R_{\Box}[Y]$, which are equivalent to $(R_{\Box}\circ {\leq})=  R_{\Box}$ and  $({\leq}\circ R_{\Box})\subseteq R_{\Box}$, as required.
\end{remark}


\begin{definition}  
	The associated  $\mathcal{L}_{\text{LE}}$-frame of any lattice expansion $\mathbb{A} = (\mathbb{L}, \mathcal{F}^{\mathbb{L}}, \mathcal{G}^{\mathbb{L}})$ is  defined by
	$\mathbb{F}_\mathbb{A} = 
	(\mathbb{X}_\mathbb{L}, {R}_\mathcal{F}, \mathcal{R}_\mathcal{G})$, where $\mathbb{X}_\mathbb{L}$ is defined as in Section \ref{ssec:polarities graphs}, $\mathcal{R}_{\mathcal{F}} = \{R_f\mid f\in \mathcal{F}\}$, and $\mathcal{R}_{\mathcal{G}} = \{R_g\mid g\in \mathcal{G}\}$, such that  for each $f\in \mathcal{F}$ and $g\in \mathcal{G}$, the relations $R_f$ and  $R_g$ are defined as follows:
	\begin{equation*} R_f (z, \overline{z}) \; \mbox{ iff }\;  f^{\mathbb{X}^+}(\overline{\mathbf{z}_s}^{(n_f)})\nleq \mathbf{z}_r \quad\quad R_g (z, \overline{z}) \; \mbox{ iff }\;   \mathbf{z}_s\nleq g^{\mathbb{X}^+}(\overline{\mathbf{z}_r}^{(n_g)}), \end{equation*}
	where the $i$-th component of $\overline{\mathbf{z}_s}^{(n_f)}$ is $\mathbf{z}_s$ 
	is ($\mathbf{z}_r$) if $f$ is monotone (antitone) in its $i$-th argument, and the $i$-th
	component of $\overline{\mathbf{z}_r}^{(n_g)}$ is $\mathbf{z}_r$ ($\mathbf{z}_s$) if $g$ is monotone (antitone) in its $i$-th argument.
\end{definition}

\begin{prop}
	The   $\mathcal{L}_{\text{LE}}$-frame $\mathbb{F}_{\mathbb{A}}$ associated with any lattice expansion $\mathbb{A} = (\mathbb{L}, \mathcal{F}^{\mathbb{L}}, \mathcal{G}^{\mathbb{L}})$ is a graph-based frame.
\end{prop}

The {\em complex algebra} of a graph-based frame $\mathbb{F}= (\mathbb{X},  \mathcal{R}_{\mathcal{F}}, \mathcal{R}_{\mathcal{G}})$ is the algebra $\mathbb{F}^+ = (\mathbb{L}, \{f_{R_f}\mid f\in \mathcal{F}\}, \{g_{R_g}\mid g\in \mathcal{G}\})$,
where $\mathbb{L} := \mathbb{P}_\mathbb{X}^+$, and for all $f \in \mathcal{F}$ and all $g \in \mathcal{G}$,
we let

{\footnotesize
\begin{enumerate}
	
	\item $ f_{R_f}: \mathbb{L}^{n_f}\to \mathbb{L}$ be defined by the assignment $f_{R_f}(\overline{c}) = \left(\left(R_f^{[0]}[\overline{\val{c}}^{(n_f)}]\right)^{[0]}, R_f^{[0]}[\overline{\val{c}}^{(n_f)}]\right)$;
	
	\item $g_{R_g}: \mathbb{L}^{n_g}\to \mathbb{L}$ be defined by the assignment  $g_{R_g}(\overline{c}) = \left(R_g^{[0]}[\overline{\descr{c}}^{(n_g)}], \left(R_g^{[0]}[\overline{\descr{c}}^{(n_g)}]\right)^{[1]}\right)$. 
	
\end{enumerate}
}
%
%
\noindent
As before $\overline{\val{c}}^{(n_f)}$ denotes the $n_f$-tuple of $\val{c}$ and $\descr{c}$ such that for
each $1 \leq i \leq n_f$ if $f$ is monotone in its $i$-th coordinate then the $i$-th coordinate of $\overline{\val{c}}^{(n_f)}$ is $\val{c}$, whereas if $f$ is antitone in its $i$-th coordinate then the $i$-th coordinate of $\overline{\val{c}}^{(n_f)}$ is $\descr{c}$, and similarly  for $\overline{\descr{c}}^{(n_g)}$. 

\begin{prop}\label{prop: graph: F plus is complete LE}
	If $\mathbb{F} = (\mathbb{X}, \mathcal{R}_\mathcal{F}, \mathcal{R}_\mathcal{G})$ is a graph-based frame for  $\mathcal{L}_{\text{LE}}$, then $\mathbb{F}^+ = (\mathbb{P}^+_\mathbb{X}, \{f_{R_f}\mid f\in \mathcal{F}\}, \{g_{R_g}\mid g\in \mathcal{G}\})$ is a complete lattice expansion. 
\end{prop}

\newpage

\paragraph{Willem Conradie} is currently an associate professor of mathematics at the University of the Witwatersrand, Johannesburg. He holds an MSc in Mathematical Logic from the Institute for Logic, Language and Computation at the University of Amsterdam and a PhD in Mathematics from the University of the Witwatersrand. His research interests lie in non-classical logics where he has made contributions to the correspondence and completeness theory of a range of extensions and generalizations of modal logic, and also to the theory and applications of hybrid logics, temporal logics and non-distributive logics. 

\paragraph{Alessandra Palmigiano} holds the Chair of Logic and Management Theory at the School of Business and Economics of the Vrije Universiteit Amsterdam, and is visiting associate professor at the Department of Mathematics and Applied Mathematics at the University of Johannesburg. She holds an MSc in Mathematical Logic from the Institute for Logic, Language and Computation at the University of Amsterdam, and a PhD in Logic and Foundations of Mathematics from the University of Barcelona (Spain). Her research interests lie in the development of algebraic, duality-theoretic, and proof-theoretic methods in non-classical logics and their application to the analysis of social behaviour.

\paragraph{Claudette Robinson} is currently a lecturer in the Department of Mathematics and Applied Mathematics at the University of Johannesburg. Before that she was a postdoctoral researcher in the School of Computer Science and Applied Mathematics at the University of the Witwatersrand. She received a PhD in Mathematics in 2015 from the University of Johannesburg. Her project focused on developing algebraic methods for hybrid logics and applying them to solve logical problems in the fi eld of hybrid logics. Her research interests include modal and temporal logics, hybrid logic, non-classical logics, lattice theory and universal algebra.

\paragraph{Nachoem M. Wijnberg} is professor at the University of Amsterdam Business School. Before that he was professor at the University of Groningen. He is visiting professor at the university of Johannesburg and before held visiting positions at Cass Business School in London and at the Institute for Innovation Research at Hitotsubashi University in Tokyo. He  has a master’s degree in law and a master’s degree in economics from the University of Amsterdam and a PhD from the Rotterdam School of Management.
The focus of his present research is on integrating formal models of logic and AI with theoretical approaches from economics and management science especially concerning the interactions between categorization and competition.
His poetry received the most important Dutch and Belgian literary awards and is translated in many languages.

\end{document}